\documentclass[11pt]{amsart}

\usepackage{amsmath,amssymb,amsfonts,amsthm,graphicx}

\usepackage{pinlabel, subcaption}

\usepackage[usenames,dvipsnames]{color}

\newtheorem{theorem}{Theorem}[section]
\newtheorem{lemma}[theorem]{Lemma}
\newtheorem{corollary}[theorem]{Corollary}
\newtheorem{definition}[theorem]{Definition}
\newtheorem{proposition}[theorem]{Proposition}

\theoremstyle{definition}

\newtheorem{example}[theorem]{Example}
\newtheorem{remark}[theorem]{Remark}
\newtheorem{convention}[theorem]{Convention}

\numberwithin{equation}{section}

\newcommand{\Z}{\mathbb{Z}} 
\newcommand{\R}{\mathbb{R}}
\newcommand{\C}{\mathbb{C}}

\newcommand{\alg}{\mathcal{A}}

\def\rrep{\widetilde{\mbox{Rep}}}
\def\fig#1{\raisebox{-2.2ex}{\includegraphics[height=5.2ex]{#1}}}

\def\figgg#1{\raisebox{-2.2ex}{\includegraphics[width=5.2ex]{#1}}}

\begin{document}

\title[Satellite ruling polynomials, DGA representations, and colored
HOMFLY-PT]{Satellite ruling polynomials, DGA representations, and the
  colored HOMFLY-PT polynomial}

\author{Caitlin Leverson}
\address{Caitlin Leverson, School of Mathematics, Georgia Institute of Technology} \email{leverson@math.gatech.edu}

\author{Dan Rutherford}
\address{Dan Rutherford, Department of Mathematical Sciences, Ball State University} \email{rutherford@bsu.edu}

\begin{abstract}
  We establish relationships between two classes of invariants of
  Legendrian knots in $\R^3$: Representation numbers of the
  Chekanov-Eliashberg DGA and satellite ruling polynomials.  For
  positive permutation braids, $\beta \subset J^1S^1$, we give a
  precise formula in terms of representation numbers for the
  $m$-graded ruling polynomial $R^m_{S(K,\beta)}(z)$ of the satellite
  of $K$ with $\beta$ specialized at $z=q^{1/2}-q^{-1/2}$ with $q$ a
  prime power, and we use this formula to prove that arbitrary
  $m$-graded satellite ruling polynomials, $R^m_{S(K,L)}$, are
  determined by the Chekanov-Eliashberg DGA of $K$.  Conversely, for
  $m\neq 1$, we introduce an $n$-colored $m$-graded ruling polynomial,
  $R^m_{n,K}(q)$, in strict analogy with the $n$-colored HOMFLY-PT
  polynomial, and show that the total $n$-dimensional $m$-graded
  representation number of $K$ to $\mathbb{F}_q^n$, $\mbox{Rep}_m(K,
  \mathbb{F}_q^n)$, is exactly equal to $R^m_{n,K}(q)$.  In the case
  of $2$-graded representations, we show that
  $R^2_{n,K}=\mbox{Rep}_2(K, \mathbb{F}_q^n)$ arises as a
  specialization of the $n$-colored HOMFLY-PT polynomial.
\end{abstract}

\maketitle

\section{Introduction} \label{sec:Intro}

The study of Legendrian knots in the standard contact $\R^3$
benefits from interactions with both symplectic topology and classical
knot theory.  For instance, the symplectic theory of holomorphic
curves applied to a Legendrian knot $K \subset \R^3$ produces an
invariant known as the Chekanov-Eliashberg DGA (differential graded
algebra), $(\mathcal{A}(K),\partial)$, that is capable of
distinguishing knots with the same underlying framed knot type.  On
the other hand, a connection with topological knot theory is made via
the $m$-graded ruling polynomial invariants, $R^m_K(z)$ (see
\cite{Ch}).  As they are characterized by skein relations, the ruling
polynomials can be regarded as Legendrian cousins of well known
topological knot invariants.  This is especially natural since, when
$m=1$ or $2$, $R^m_K(z)$ arises as a specialization of the 2-variable
Kauffman or HOMFLY-PT polynomial (see \cite{R1,R2}).

Ruling polynomials are defined via a count of certain combinatorial
objects called normal rulings, and there is a well studied
correspondence between augmentations of $(\mathcal{A}(K), \partial)$,
i.e. DGA maps into a field $(\mathbb{F},0)$, and normal rulings of
$K$, cf. \cite{F,FI,HenryR,L1,NRSS,NgSabloff,Sab1, Su}.  In
\cite{NgR}, this was strengthened to an equivalence between the
existence of
\begin{itemize}
\item[(a)] finite dimensional representations of
  $(\mathcal{A}(K), \partial)$ over $\Z/2$, and
\item[(b)] normal rulings of links obtained from $K$ via the
  Legendrian satellite construction.
\end{itemize}
Both of these items may be refined to arrive at whole collections of
Legendrian knot invariants.  From (a), normalized counts of
representations of $(\mathcal{A},\partial)$ on vector spaces over
finite fields produce Legendrian invariant {\it representation
  numbers}; see Section \ref{sec:RepNumber}.  From (b), for any fixed
$L \subset J^1S^1$, considering the satellite link $S(K,L)$ provides
an invariant of a Legendrian knot $K \subset \R^3$ via the assignment
$K \mapsto R^m_{S(K,L)}(z)$. We refer to the latter invariants as {\it
  satellite ruling polynomials}.  The results of the present article
show that satellite ruling polynomials provide information about
representation counts of the Chekanov-Eliashberg algebra and at the
same time are determined by such counts.  In addition, we strengthen
the connection with topological knot invariants by showing that a
natural specialization of the $n$-colored HOMFLY-PT polynomial
recovers certain representation numbers of $K$.

We now give a more detailed overview of our results, under the
simplifying assumption that the Legendrian knot $K\subset J^1\R$ has
rotation number $r(K)=0$.  In the following, $m\geq 0$ is a
non-negative integer and $\mathbb{F}_q$ is a finite field of order
$q$.  When $m$ is odd, we assume $\mathit{char}(\mathbb{F}_q) = 2$.
In Section \ref{sec:RepNumber}, we introduce Legendrian isotopy
invariant {\it $m$-graded representation numbers} denoted
$\mbox{Rep}_m(K, (V,d), B)$ that count representations of
$(\mathcal{A},\partial)$ on a differential graded vector space $(V,d)$
over $\mathbb{F}_q$ where certain distinguished generators, $t_i$,
associated to basepoints on $K$ have their images restricted by a
choice of subset $B \subset GL(V)$.  We also consider {\it reduced
  representation numbers}, $\rrep_m(K,(V,d),B)$, that are renormalized
so the unknot takes the value $1$, and a {\it total $n$-dimensional
  representation number}, $\mbox{Rep}_m(K, \mathbb{F}_q^n)$, that
counts all $m$-graded representations to $\mathbb{F}_q^n$ (with
grading concentrated in degree $0$) without restriction on the image
of the $t_i$.

In Theorem \ref{thm:A}, we provide a precise formula for a certain
class of satellite ruling polynomials in terms of reduced
representation numbers.  A special case (see
Proposition~\ref{prop:permutation}) is the following:

\medskip

\noindent {\bf Theorem A.} %{\it Suppose that $K\subset J^1\R$ is a
% (connected) Legendrian knot, and $\mathbb{F}_q$ a finite field.  If
% $m$ is odd, then assume $\mathit{char}(\mathbb{F}_q)=2$; if $m$ is
% even, then assume $r(K) = 0$.}  Let $K \subset J^1\R$ be a connected
% Legendrian knot, and let $m\geq 0$, $m \,|\,2r(K)$ with $r(K)=0$ if
% $m$ is even.
{\it Let $\beta \subset J^1S^1$ be an $n$-stranded positive
  permutation braid $\beta \subset J^1S^1$.  Then,}
\[
R^m_{S(K,\beta)}(z)\big\vert_{z=q^{1/2}-q^{-1/2}} =
q^{-\lambda_m(\beta)/2}(q^{1/2}-q^{-1/2})^{-n}\sum_{d}\rrep_m(K,
(V_\beta,d), B^m_\beta),
\] {\it where $V_\beta$ is a graded vector space over $\mathbb{F}_q$
  associated to $\beta$ (in Section \ref{sec:Sat}); $B_\beta \subset
  GL(V_\beta)$ is the path subset of $\beta$ (defined in Section
  \ref{sec:Path}); $\lambda_m(\beta)$ is a certain signed count of the
  crossings of $\beta$ (defined in Section \ref{sec:Sat}); and the
  summation is over all strictly upper triangular differentials,
  $d:V_\beta \rightarrow V_\beta$, with $\deg(d) = +1$.}
% \end{theorem}

\medskip

\noindent Combining Theorem \ref{thm:A} with a skein relation
argument, we establish in Corollary \ref{cor:A} that for any fixed $L
\subset J^1S^1$ the satellite ruling polynomial $R^m_{S(K,L)}$ is
determined by the Chekanov-Eliashberg DGA of $K$.

We also obtain in Theorem \ref{thm:color} a converse-type formula for
representation numbers in terms of ruling polynomials that we state
here as:

\medskip

\noindent {\bf Theorem B.}  {\it Assume $m\neq 1$.}
% ; if $m$ is even, assume $r(K)=0$.}
{\it The total $n$-dimensional $m$-graded representation number
  satisfies}
\[
\mbox{Rep}_m(K, \mathbb{F}_q^n) = R^m_{n, K}(q),
\] {\it where $R^m_{n,K}(q)$ denotes the $n$-colored $m$-graded ruling
  polynomial (defined in Section \ref{sec:color}).}

\medskip

\noindent Here, the $n$-colored ruling polynomial is a specific linear
combination of satellite ruling polynomials that is so named because
of a strong formal relation with, $P_{n,K}(a,q)$, the $n$-colored
HOMFLY-PT polynomial: The linear combination that defines $R^m_{n,K}$
and obtains the total $n$-dimensional representation number is {\it
  exactly} the same linear combination of braids that defines the
$n$-colored HOMFLY-PT polynomial!

The case $m=2$ is of special interest since the results of \cite{R2}
show the $2$-graded ruling polynomial is a specialization of the
HOMFLY-PT polynomial.  In Theorem \ref{thm:nHOMFLY}, we extend this
result to see that $R^2_{n, K}(q)=\mbox{Rep}_2(K, \mathbb{F}_q^n)$
arises as a specialization of $P_{n,K}(a,q)$,
\[
P_{n,K}(a,q)\big\vert_{a^{-1}=0} = \mbox{Rep}_2(K,\mathbb{F}_q^n).
\]
In the process, we generalize the HOMFLY-PT estimate for $tb(K)
+|r(K)|$ from \cite{FT} to an estimate
\begin{equation} \label{eq:tbr} \mathit{tb}(K) +|r(K)| \leq
  \frac{1}{n} \deg_a \widehat{P}_{n,K}(a,q),
\end{equation} 
where $\widehat{P}_{n,K}$ is the framing independent version of
$P_{n,K}$. Moreover, in Corollary \ref{cor:sharp} we observe
additional results parallel to \cite{R2}; for example, the inequality
(\ref{eq:tbr}) is sharp as an estimate for $tb(K)$ if and only if $K$
has an $n$-dimensional $2$-graded representation.

The colored HOMFLY-PT polynomials (where in general the color is by an
arbitrary partition) compute the quantum
$U_q(\mathfrak{sl}_n)$-invariants of a knot associated to arbitrary
irreducible representations of $\mathfrak{sl}_n$, cf. \cite{AM,
  Aiston}.  This relationship between representations of the
Chekanov-Eliashberg DGA and quantum invariants looks tantalizing, and
it would be interesting to find a more direct connection between the
representation theories of $(\mathcal{A},\partial)$ and
$U_q(\mathfrak{sl}_n)$.

\subsection{Organization} The remainder of the article is organized as
follows.  In Section \ref{sec:background} we collect relevant
background material concerning the Chekanov-Eliashberg DGA and ruling
polynomials.  In particular, the theorem of Henry and the second
author from \cite{HenryR} that relates augmentation numbers with
ordinary ruling polynomials is fundamental to our approach.  Section
\ref{sec:RepNumber} presents definitions of representation numbers and
establishes their invariance under Legendrian isotopy.  In Section
\ref{sec:Path}, we obtain several results about the path matrices,
originally introduced by K{\'a}lm{\'a}n \cite{Kalman}, of positive
braid Legendrians $\beta \subset J^1S^1$.  The path matrix of $\beta$
appears prominently in the calculation of the Chekanov-Eliashberg DGA
of the satellite $S(K,\beta)$ that is given in Section \ref{sec:DGA}
from the point of view of the Lagrangian $(xy)$-projection.  We note
that the exposition in the current article is independent of that from
\cite{NgR} where computations of $S(K,\beta)$ were made using the
front $(xz)$-projection. Overall, the Lagrangian approach is more
efficient for producing precise bijections between augmentations of
$S(K,\beta)$ and higher dimensional representations of $K$.  In
Theorem \ref{thm:bijection} we obtain such a bijection, and the
remainder of Section \ref{sec:Sat} establishes Theorem \ref{thm:A} and
Corollary \ref{cor:A}.  In the concluding Section \ref{sec:color} we
introduce the $n$-colored ruling polynomials and then prove the
formula from Theorem \ref{thm:color}.  The connection with the colored
HOMFLY-PT polynomial, in the case $m=2$, is then established in
Theorem \ref{thm:nHOMFLY} and Corollary \ref{cor:sharp}.

\subsection{Acknowledgments} We thank John Etnyre, Lenny Ng, and Tao
Su for discussions related to this work. Thanks also to Ralph Bremigan
for pointing us to the Bruhat decomposition, and to Michael Abel for
providing us with Mathematica code that produces colored HOMFLY-PT
polynomials for torus knots. Also, we thank the referees for many
helpful suggestions. This work was supported by grants from the Simons
Foundation (\#429536, Dan Rutherford) and the NSF (DMS-1703356,
Caitlin Leverson).

\section{Background}  \label{sec:background}

We work with Legendrian links in the $1$-jet spaces, $J^1\R^1 \cong
\R^3$ and $J^1S^1 \cong S^1 \times \R^2$ coordinatized as $(x,y,z)$,
with their standard contact structures given by the contact form $dz -
y \,dx$.  We assume familiarity with some standard notions including
the {\bf front ($xz$)-projection} and the {\bf Lagrangian
  ($xy$)-projection}; in $J^1S^1$ these projections are to $S^1 \times
\R$ which we illustrate as $[0,1] \times \R$ with left and right
boundaries identified.  When $K$ is connected we denote the {\bf
  rotation number} of $K$ by $r(K)$, and we adopt the convention that
when $K$ has multiple components, $K = \sqcup_{i = 1}^c K_i$, $r(K)$
denotes the greatest common divisor of $r(K_i)$.  For $m\geq 0$ a
non-negative integer, $\Z/m$-valued {\bf Maslov potential}, $\mu$, for
a Legendrian $K$ in $J^1\R^1$ or $J^1S^1$ is an $\Z/m$-valued function
on $K$ that is locally constant except at cusp points where the value
at the upper branch of the cusp is $1$ larger than at the lower
branch.  A connected Legendrian $K$ has a $\Z/m$-valued Maslov
potential iff $m\,\vert\, 2r(K)$. When $m$ is even, we always assume
that Maslov potentials are chosen to be compatible with the
orientation of $K$ so that $\mu$ takes even (resp. odd) values on
strands that are oriented in the increasing (resp. decreasing)
$x$-direction. For further background on Legendrian knots in $\R^3$
and $J^1S^1$ see \cite{E,NgTraynor}.

\subsection{Chekanov-Eliashberg DGA}
In this article, {\bf differential graded algebras}, abbrv. DGA, are
unital; graded by $\Z/m$ for some $m \geq 0$; and have differentials
of degree $-1$ satisfying the graded signed Leibniz rule,
$\partial(xy) = (\partial x)y + (-1)^{\lvert x\rvert} x (\partial y)$,
where if $m$ is odd we assume the coefficient ring has characteristic
$2$.

Let $K \subset J^1\R$ or $J^1S^1$ be a Legendrian link equipped with
some collection of basepoints (allowing for multiple basepoints on
each component of $K$).  We will use the {\it fully non-commutative}
version of the {\bf Chekanov-Eliashberg DGA} (also called the
Legendrian contact homology DGA), $(\mathcal{A}(K), \partial)$, as
defined via the $xy$-projection of $K$. To establish conventions, we
briefly review $(\mathcal{A}(K), \partial)$ here.  More details can be
found in \cite{Ch, ENS}; see also \cite{HenryR} for an exposition
matching our sign conventions (which follow \cite{NgLSFT}), and
\cite{NgR} for the case of the fully non-commutative DGA with multiple
basepoints.

Using $\Z$-coefficients, $\mathcal{A}(K)$ is a unital associative
algebra with non-invertible generators, $a_1, \ldots, a_r$, in
bijection with Reeb chords of $K$, i.e. the double points of
$\pi_{xy}(K)$, and invertible generators $t_1^{\pm1}, \ldots,
t_\ell^{\pm1}$ in bijection with basepoints of $K$.  There are no
relations other than $t_i t_i^{-1} = t_i^{-1} t_i =1$.  A choice of
$2r(K)$-valued Maslov potential, $\mu$, for $K$ leads to a
$\Z/2r(K)$-grading, $\mathcal{A}(K)= \oplus_{k \in \Z/2r(K)}
\mathcal{A}_k$, where $\mathcal{A}_k=\{a\in\mathcal{A}(K):\lvert
a\rvert=k\mod 2r(K)\}$.  In particular, if all components of $K$ have
$r(K_i)=0$, then $\mathcal{A}(K)$ is $\Z$-graded.  Degrees of
generators (mod $2r(K)$) are given by $\lvert t_i\rvert=0$ and $\lvert
a_i\rvert = \mu(U_{i}) - \mu(L_{i}) + \mathit{ind}(a_i)-1$ where $U_i$
and $L_i$ are the upper and lower strands (with respect to the
$z$-coordinate) of $K$ that cross in $\pi_{xy}(K)$ and
$\mathit{ind}(a_i)$ is the Morse index of the critical point of
$z\vert_{U_i}(x)-z\vert_{L_i}(x)$ at $a_i$.

The differential $\partial: \mathcal{A}(K) \rightarrow \mathcal{A}(K)$
satisfies the signed Leibniz rule
% $\partial(ab) = \partial(a)b + (-1)^{\lvert a\rvert}a \partial(b)$,
and has $\partial t_i =0$.  For Reeb chords $a$ and $b_1, \ldots,
b_n$, let $\mathcal{M}(a;b_1, \ldots, b_n)$ denote the set modulo
reparametrization of boundary punctured holomorphic (or equivalently
orientation preserving, immersed) disks in $\R^2=\C$ having boundary
on $\pi_{xy}(K)$ and having a (convex) corner mapped to a positive
quadrant at $a$ and (convex) corners mapped to negative quadrants at
$b_1, \ldots, b_n$ appearing in counter-clockwise order.  Here, we use
the Reeb signs for quadrants as in Figure \ref{fig:Reeb} (left).  We
have
\[
\partial a = \sum_{n \geq 0} \sum_{b_1, \ldots, b_n} \sum_{[u] \in
  \mathcal{M}(a;b_1, \ldots, b_n)} \iota(u) \cdot w(u)
\]
where $w(u) = r_1r_2 \cdots r_N$ is the product of basepoints and
negative corners that appear along $u$ ordered as they appear
(counter-clockwise) along the boundary of $u$ starting at the positive
corner at $a$.  In this product, a $t_i$ generator corresponding to a
basepoint $*_i$ appear with exponent $+1$ (resp. $-1$) when the
orientations of $\partial u$ and $K$ agree (resp. disagree) at $*_i$.
The coefficient $\iota(u) \in \{ \pm1\}$ is given by the product
\[
\iota(u) = \iota' \iota_0 \iota_1 \cdots \iota_n.
\]
Here, $\iota'$ is $+1$ (resp. $-1$) when the orientation of the
initial arc of $\partial u$ that is just counter-clockwise from the
positive corner of $u$ agrees (resp. disagrees) with the orientation
of $K$, and $\iota_0, \ldots, \iota_n$ are the orientation signs of
the quadrants covered by the various (positive and negative) corners
of $u$.  As depicted in Figure \ref{fig:Reeb} (right), at each
$(xy)$-crossing the two quadrants to the right (resp. to the left) of
the understrand (with respect to its orientation) have $-1$
(resp. $+1$) orientation sign.

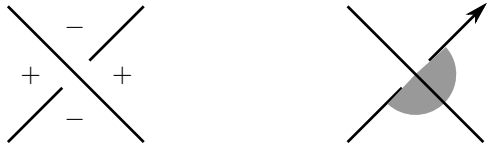
\begin{figure}[t]
  \labellist
  \small\hair 2pt \pinlabel {$-$} at 40 14 \pinlabel {$-$} at 40 68
  \pinlabel {$+$} at 68 40 \pinlabel {$+$} at 14 40
  \endlabellist
  \centering
  \includegraphics[scale=.65]{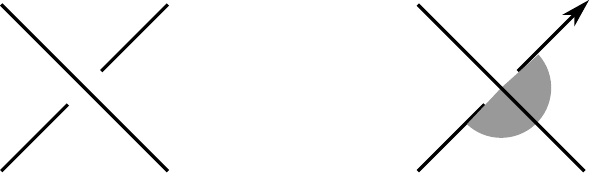}
  \caption{The Reeb signs (left) and orientation signs (right) of
    quadrants at a Reeb chord.  The shaded quadrants have orientation
    sign $-1$.}
  \label{fig:Reeb}
\end{figure}

\subsubsection{$\Z$-gradings on $\mathcal{A}(K)$ with multiple
  basepoints} \label{sec:Zgrading} As long as every component of $K$
has at least one basepoint, it is possible to define a $\Z$-grading on
$\mathcal{A}(K)$.  (The case of a single basepoint per component is
essentially as in \cite{ENS}.)  First, choose degrees $\lvert
t_j\rvert \in 2\Z$ for the generators associated to the basepoints,
$*_1, \ldots, *_\ell$ of $K$, subject only to the restriction that the
sum of the degree of all basepoints on a given component $K_i \subset
K$ is equal to $-2r(K_i)$.  Next, choose a $\Z$-valued Maslov
potential, $\mu$, for $K\setminus\{*_1, \ldots, *_\ell\}$, such that
when $*_i$ is passed in the direction of the orientation of $K_i$ the
value of $\mu$ decreases by $\lvert t_i\rvert$.  Since $\mu$ is
$\Z$-valued, following the definition above provides integer degrees,
$\lvert a_i\rvert \in \Z$, for the $a_i$.

\begin{remark}
  The requirement that $|t_i|$ is even is not strictly necessary.
  However, if some of the $t_i$ have odd degree (we will not actually
  use this case later in the article), we can no longer assume that
  the parity of $\mu$ is compatible with the orientation of $K$,
  i.e. even (resp. odd) at points where $K$ is oriented in the
  increasing (resp. decreasing) $x$-direction.  When this occurs, we
  modify the definition of $\partial$ by, for each $a_i$, multiplying
  all terms of $\partial a_i$ by an extra sign $\iota''$ that is $+1$
  (resp. $-1$) if the parity of $\mu$ is (resp. is not) compatible
  with the orientation of $K$ at the overstrand of $a_i$.
\end{remark}

\begin{proposition} \label{prop:DGAInvariance}
  \begin{enumerate}
  \item For any choice of $\mu$, $(\mathcal{A}(K),\partial)$ is a
    $\Z$-graded DGA, i.e. $\partial$ has degree $-1$, and satisfies
    $\partial^2=0$.

  \item If $(K, \mu)$ and $(K',\mu')$ are Legendrian isotopic (as
    basepointed links, preserving Maslov potentials), then
    $(\mathcal{A}(K), \partial_1)$ and $(\mathcal{A}(K'), \partial_2)$
    are stable tame isomorphic.\footnote{See \cite{ENS} for a complete
      definition of stable tame isomorphism which is a type of
      equivalence for DGAs equipped with explicit generating sets. For
      the fully non-commutative DGA, the definition of elementary
      isomorphism needs to be expanded to allow for isomorphisms that
      multiply a single $a_i$ generator on the right or left by some
      $t^{\pm1}_j$ to account for the effect of repositioning base
      points.  As a result, when DGAs are considered up to tame
      isomorphism, the degree distribution of the set of $a_i$
      generators is only well defined mod the g.c.d. of all $\lvert
      t_j\rvert$.}

  \item Let $(\mathcal{A}_1,\partial_1)$ and
    $(\mathcal{A}_2,\partial_2)$ be formed using two collections of
    basepoints, $B_1$ and $B_2$, for a common Legendrian $K$.  Suppose
    that $B_1$ and $B_2$ are identical except for a single component
    $K_i \subset K$ where $B_1$ only has one basepoint, $*$, and $B_2$
    has a sequence of basepoints, $*_1,\ldots, *_n$, all appearing in
    a small neighborhood of $*$ and ordered according to the
    orientation of $K$.  Let $t \in \mathcal{A}_1$ and $t_1, \ldots,
    t_n \in \mathcal{A}_2$ denote the generators corresponding to $*$
    and $*_1, \ldots, *_n$.  Then, $(\mathcal{A}_2,\partial_2)$ is the
    free product with amalgamation
    \[
    \mathcal{A}_2 = \mathcal{A}*_{t = t_1\ldots t_l} \Z\langle
    t_1^{\pm1}, \ldots, t_l^{\pm1} \rangle
    \]
    with $\partial_2$ induced by $\partial_1$ and the $0$ differential
    on $\Z\langle t_1^{\pm1}, \ldots, t_l^{\pm1} \rangle$.

    % \item If the collections of basepoints $B_1$ and $B_2$ are
    %   isotopic in $K$, then there is a DGA isomorphism
    %   $\phi:(\mathcal{A}_1, \partial_1) \rightarrow
    %   (\mathcal{A}_2, \partial_2)$, with $\phi(t_i) =t_i$ for all
    %   $i$.  that is a composition of isomorphisms that fix all
    %   generators except for a single Reeb chord $a_i$ for some $i$
    %   have $a_i \rightarrow a_i t_{j}^{\pm1}$fix all generators
    %   except for a single reeb chord generator
    %   \]\fix the $t_1, \ldots, t_\ell$ and multipl that multiply a
    %   single reeb chord generator
  \end{enumerate}

\end{proposition}

\begin{proof} This is closely related to Theorem 2.20 and 2.21 of
  \cite{NgR} (which considers the $\Z/2$-coefficient case with
  $\Z/2r(K)$-grading).

  Item (3) is immediately verified from definitions.  [The algebra
  $\mathcal{A}_2$ is just $\mathcal{A}_1$ with $t^{\pm1}$ removed from
  the generating set and replaced with $t_1^{\pm1}, \ldots,
  t_n^{\pm1}$.  The differentials are related as claimed since every
  time $*$ appears along the boundary of a disk used in computing
  $\partial_1$, the entire sequence $*_1, \ldots, *_n$ appears in its
  place when computing $\partial_2$.]

  For Items (1) and (2), the result is standard in the case where
  every component has a single basepoint, cf. \cite{Ch, ENS, NgLSFT}.
  Combining this result with item (3), it suffices to consider the
  case where $K = K'$ except for the basepoint sets, $B_1$ and $B_2$,
  which are related by isotopy in $K$.  Moreover, we can assume $B_2$
  is obtained from $B_1$ by pushing a basepoint with generator $t_j$
  through a crossing $a_l$ following the orientation of $K$.  When
  $t_j$ is pushed across the overstrand of $a_l$, one checks that the
  graded algebra isomorphism $\phi$ that fixes all generators other
  than $a_l$ and maps $a_l \mapsto t_j^{-1} a_l$ is a chain
  map. (Here, if $t_j$ has odd degree, the presence of the $\iota''$
  term that modifies the definition of $\partial a_l$ becomes
  important when verifying that $\partial_2 \circ \phi (a_l) =
  (-1)^{\lvert t_j\rvert}t_j^{-1} \partial_2 (a_l)$ agrees with $\phi
  \circ \partial_1(a_l) = \partial_1(a_l)$.)  When $t_j$ crosses
  past $a_l$ along the understrand, $\phi$ is modified to have
  $a_l\mapsto a_l t_j$.
\end{proof}

\begin{remark}
  \begin{enumerate}
  \item Categorically, (3) is the statement that
    $(\mathcal{A}_2, \partial_2)$ is the pushout (in the category of
    DGAs over $\Z$) of $(\mathcal{A}_1, \partial_1)$ and $\Z\langle
    t_1^{\pm1}, \ldots, t_l^{\pm1} \rangle$ with respect to the
    inclusion of the subalgebra $\Z\langle t^{\pm1}\rangle
    \hookrightarrow \mathcal{A}_1$ and the homomorphism $\Z\langle
    t^{\pm1} \rangle \rightarrow \Z\langle t_1^{\pm1}, \ldots,
    t_l^{\pm1} \rangle$, $t\mapsto t_1\cdots t_l$.

  \item When $K$ has only one component with a single basepoint,
    $(\mathcal{A}(K), \partial)$ is independent of the choice of
    $\mu$.  When $K$ has one component, but multiple basepoints,
    $(\mathcal{A}(K), \partial)$, is uniquely determined by the choice
    of degrees for the $t_i$.  For multi-component links, different
    choices of $\mu$ can lead to different gradings on
    $(\mathcal{A}(K),\partial)$.

  \end{enumerate}

\end{remark}

\subsection{Ruling polynomials}

We also briefly recall the ruling polynomial invariants of a
Legendrian link $K$ in $J^1\R$ or $J^1S^1$ equipped with a
$\Z/m$-valued Maslov potential, $\mu$, where $m\,\vert\,2r(K)$.
The {\bf $m$-graded ruling polynomial} of $K$ is
\[R^m_K(z) = \sum_{\rho} z^{j(\rho)},\] where the sum is over all {\bf
  $m$-graded normal rulings} of $K$ and $j(\rho) = \#(\mbox{switches})
- \#(\mbox{right cusps})$.  Here, an $m$-graded normal ruling of $K$
is a combinatorial structure associated to the front projection of
$K$.  Normal rulings were introduced independently by Fuchs and
Chekanov-Pushkar, cf. \cite{F} and \cite{ChP}; a detailed definition
of normal rulings for links in $J^1\R$ and $J^1S^1$ may be found, for
instance, in both \cite{R1} and \cite{NgR}.  The $m$-graded ruling
polynomial satisfies the following skein relations:
\begin{enumerate}
\item[(i)]
  \[
  \fig{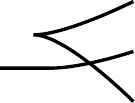} \quad \quad -\fig{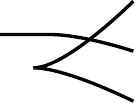}= z\left(
    \delta_1\fig{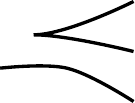} - \delta_2 \fig{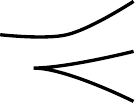}\right)
  \]
\item[(ii)]
  \[
  \fig{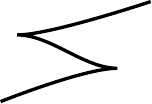} = \fig{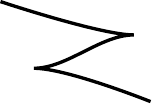} = 0
  \]
\item[(iii)]
  \[
  \fig{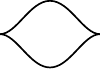} \sqcup K = z^{-1} K
  \]
\end{enumerate}
where in (i), $\delta_1$ (resp. $\delta_2$) is $1$ if the the two
strands that intersect in the first (resp. second) term on the left
hand side have equal Maslov potentials, and is $0$ otherwise; see
\cite{R2,R1} for the cases $m=1,2$.  For Legendrian links in $J^1\R$,
the ruling polynomials are uniquely characterized by the skein
relations (i)-(iii), together with the normalization $R^m_U(z) =
z^{-1}$ on the standard ($\mathit{tb}=-1$) Legendrian unknot
\[
U = \fig{images/Unknot}.
\]
The proof from Lemma 6.10 of \cite{R1} (which concerns the $2$-graded
case), applies with a trivial modification to show that using the
relations (i)-(iii) any Legendrian link in $J^1S^1$ can be written as
a $\Z[z^{\pm1}]$-linear combination of products of the basic fronts,
$A_m$, as pictured in Figure \ref{fig:Basic}.  Here, the product of
two Legendrians in $J^1S^1$ is defined by stacking front diagrams
vertically, and we allow the basic front factors to be equipped with
$\Z/m$-valued Maslov potentials of arbitrary value.

\begin{remark} \begin{enumerate}
  \item The skein relation (iii) follows from (i) and (ii) when $K
    \neq \emptyset$.
  \item When $K$ has only one component, $R^m_K(z)$ is independent of
    the choice of $\Z/m$-valued Maslov potential.  This is not true
    for multi-component links.
  \end{enumerate}
\end{remark}

\begin{figure}[t]
  \labellist
  \pinlabel {$A_3\quad=\,\,$} [r] at -4 46
  \endlabellist
  \centering
  \includegraphics[scale=.65]{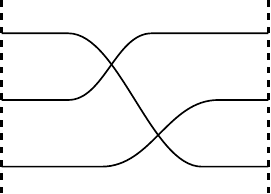}
  \caption{The basic front, $A_m \subset J^1S^1$, for $m \geq 1$ winds
    $m$ times around $S^1$ with $m-1$ crossings and no cusps.}
  \label{fig:Basic}
\end{figure}

\section{Representation numbers} \label{sec:RepNumber}

In this section, we define various Legendrian isotopy invariants from
counts of representations of the Chekanov-Eliashberg DGA.

\subsection{Representations of a DGA}
Let $(\mathcal{A}, \partial)$ be a $\Z$-graded DGA, $\mathcal{A}=
\oplus_{k \in \Z} \mathcal{A}_k$, and let $(\mathcal{B}, \delta)$ be a
(unital) DGA graded by $\Z/M$, $\mathcal{B} = \oplus_{k \in
  \Z/M}\mathcal{B}_k$.  With a divisor $m\,\big\vert\,M$ fixed, for
any $k \in \Z$, we write
\[
\mathcal{B}^m_k = \bigoplus_{\begin{array}{c} l \in \Z/M \\ l = k \,
    \mbox{mod} \, m \end{array}} \mathcal{B}_l.
\]

\begin{definition} An {\bf $m$-graded representation} from
  $(\mathcal{A},\partial)$ to $(\mathcal{B},\delta)$ is a (unital)
  algebra homomorphism $f:(\mathcal{A}(K),\partial) \rightarrow
  (\mathcal{B},\delta)$ that satisfies
  \[
  f \circ \partial = \delta \circ f \quad \quad \mbox{and} \quad
  \quad f(\mathcal{A}_k) \subset \mathcal{B}^m_k, \, \forall k \in \Z.
  \]
\end{definition}

\subsection{Defining $m$-graded representation
  numbers} \label{sec:defnumbers} 

Given a Legendrian link $K = \sqcup_{i=1}^c K_i \subset J^1\R$, fix
$m\geq 0$ with $m \, \vert\, 2r(K)$.  {\it If $m$ is even, we impose
  the requirement that $r(K)=0$} (for all components of $K$).  Let
$\mu$ be a Maslov potential for $K$ that is $\Z$-valued if $m$ is even
and $\Z/2m$-valued if $m$ is odd.  (For $\Z/2m$-valued Maslov
potentials, we maintain the requirement that $\mu$ is even (resp. odd)
valued when $K$ is oriented in the positive (resp. negative)
$x$-direction.)  We will define a class of Legendrian isotopy
invariant $m$-graded representation numbers for $(K, \mu)$ that, in
the case $c=1$, are independent of $\mu$.

\begin{remark}  When $m$ is even, the assumption that $r(K)=0$ is required for the proof of Legendrian invariance of $m$-graded representation numbers in Proposition 3.9.  For $1$-component Legendrian knots, in some cases of interest,  no generality is lost by this assumption.  Indeed, it will be shown in Corollary \ref{cor:rotation} that the existence of $2$-graded representations of certain types  implies that $r(K)=0$.
\end{remark}

We begin by making the following additional choices:
\begin{itemize}
\item[(\textbf{A1})] Equip $K$ with a collection of basepoints so that
  every component of $K$ has at least $1$ basepoint.
\item[(\textbf{A2})] As in Section \ref{sec:Zgrading}, assign degrees,
  $|t_i|$, to generators associated to basepoints with the restriction
  that:
  \begin{itemize}
  \item[(a)] If $m$ is even, then all $t_i$ have $\lvert t_i\rvert=0$.
  \item[(b)] If $m$ is odd, then $2m\,\big\vert\, \lvert t_i\rvert$
    for all basepoints.
  \end{itemize}
  Then, if $m$ is odd, fix an additional $\Z$-valued Maslov potential,
  $\widetilde{\mu}$, (discontinuous at basepoints with $|t_i| \geq 0$)
  that agrees with $\mu$ when reduced mod $2m$.
\item[(\textbf{A3})] For each component $K_i \subset K$, fix an
  initial basepoint, $t_{e_i}$, among the basepoints on $K_i$.
\end{itemize}
Let $(\mathcal{A}(K), \partial)$ be the resulting $\Z$-graded version
of the Chekanov-Eliashberg DGA as in Section \ref{sec:Zgrading}.

Next, given a $\Z/M$-graded DGA $(\mathcal{B}, \delta)$ with
$m\,|\,M$, we choose a subset $T_i \subset (\mathcal{B}^m_0)^{*}$
(consisting of invertible elements in $\mathcal{B}^m_0$) for each
component $K_i \subset K$, and we write $\mathbf{T} =
\{T_i\}_{i=1}^c$.  Denote by
\[
\overline{\mbox{Rep}}_m(K, (\mathcal{B},\delta), \mathbf{T})
\]
the set of all $m$-graded representations from
$(\mathcal{A}(K),\partial)$ to $(\mathcal{B},\delta)$ with the
property that
\[
f(s_i) \in T_i, \quad \mbox{for all} \, 1 \leq i \leq c,
\]
where $s_i$ is the product of all invertible generators from
basepoints of $K_i$ ordered with $t_{e_i}$ first and so that the
factors appear in accordance with the orientation of $K_i$.

\begin{definition} The {\bf shifted Euler characteristic of $K$
    centered at $k \in \Z$} is
  \[
  \chi^k = \sum_{l \geq 0} (-1)^lr_{k +l} + \sum_{l <0} (-1)^{l+1}
  r_{k+l},
  \]
  where $r_l$ is the number of Reeb chords of $\Lambda$ of degree $l$.
  Alternatively, if we define
  \[
  \eta^k: \Z \rightarrow \Z, \quad \eta^k(l) =
  \left\{\begin{array}{cr} (-1)^{l-k}, & l\geq k, \\ (-1)^{l-k+1}, &
      l<k, \end{array}\right.
  \]
  then $\chi^k = \sum_{i} \eta^k(\lvert a_i\rvert)$ where the sum is
  over all Reeb chords of $K$.
\end{definition}

Consider the following conditions on $(\mathcal{B},\delta)$:
\begin{itemize}
\item[(\textbf{B1})] $\mathcal{B}$ has a finite number of elements.
\item[(\textbf{B2})] For all $k \in \Z$, $\lvert\mathcal{B}^m_k\rvert
  = \lvert\mathcal{B}^m_{-k}\rvert$. (Here, and throughout the article, $|X|$ indicates the {\it cardinality} of a set $X$.)
\end{itemize}
\begin{definition}\label{def:repNum}
  Suppose that $(\mathcal{B}, \delta)$ satisfies (B1)-(B2).  We then
  define {\bf $m$-graded representation numbers} of $K$ as
  \begin{equation} \label{eq:numbers} \mbox{Rep}_m(K,
    (\mathcal{B},\delta), \mathbf{T}) =\left(\lim_{N \rightarrow
        +\infty} \prod_{k \in \Z, \lvert k\rvert\leq N}
      \lvert\mathcal{B}^m_k\rvert^{-(\chi^k/2)} \right)\cdot
    \lvert(\mathcal{B}^m_0)^* \cap \ker \delta\rvert^{-\ell} \cdot
    \lvert\overline{\mbox{Rep}}_m(K, (\mathcal{B},\delta),
    \mathbf{T})\rvert,
  \end{equation}
  where $\ell$ is the total number of basepoints on $K$.
\end{definition}

We will show in Proposition~\ref{prop:Invariance} that the
$\mbox{Rep}_m(K, (\mathcal{B},\delta), \mathbf{T})$ are Legendrian
isotopy invariants of $(K,\mu)$ and are independent of the other
choices involved in their definition.  In the proof, we show that stabilizing a DGA
$\alg$ in degree $k$ changes $|\overline{\mbox{Rep}}_m(K,
(\mathcal{B},\delta), \mathbf{T})|$ by a factor of
$|\mathcal{B}^m_k|$, thus leaving \eqref{eq:numbers} unchanged. The
following lemma shows the limit in \eqref{eq:numbers} exists.
  
\begin{lemma} \label{lem:bN} The sequence $(b_N)_{N=1}^\infty$ with
  $\displaystyle b_N = \prod_{k \in \Z, \lvert k\rvert\leq N}
  \lvert\mathcal{B}^m_k\rvert^{-(\chi^k/2)}$ is eventually constant.
  In particular, the limit exists in (\ref{eq:numbers}).
\end{lemma}

\begin{proof}
  Note that when $\lvert l\rvert <k$, we have $\eta^k(l) = (-1)^{l-k
    +1}=-(-1)^{l+k}=-\eta^{-k}(l)$.  It follows that $\chi^{-k} =
  -\chi^{k}$ holds once $k > \lvert\deg(a_i)\lvert$ for all $i$.
  Using (B2), we see that once $k$ is large enough all
  $\lvert\mathcal{B}^m_{k}\rvert^{-\chi^k/2}$ and
  $\lvert\mathcal{B}^m_{-k}\rvert^{-\chi^{-k}/2}$ terms cancel in the
  product.  Thus, $b_N$ is eventually constant.
\end{proof}

\begin{remark}
  \begin{enumerate}
  \item The assumption (B2) can be removed if $m=0$, as the finiteness of $\mathcal{B}$ implies $\lvert
    \mathcal{B}^0_k \rvert = \lvert \mathcal{B}_k \rvert =1$ for $|k|
    \gg 0$. 
  \item If $m>0$ is odd, then (B2) can be removed if the limit in
    (\ref{eq:numbers}) is replaced with
    \[
    \lim_{N \rightarrow \infty} \left(\prod_{k \in \Z, \lvert
        k\rvert\leq 2mN}
      \lvert\mathcal{B}^m_k\rvert^{-(\chi^k/2)}\right).
    \]
  \end{enumerate}
\end{remark}

\begin{definition}\label{def:reducedRep}
  It is also convenient to introduce {\bf reduced representation
    numbers} defined by
  \begin{equation} \label{eq:reduced} \rrep_m(K, (\mathcal{B},\delta),
    \mathbf{T}) = |\mathcal{B}^m_0|^{-1/2} |(\mathcal{B}^m_0)^* \cap
    \ker \delta| \cdot \mbox{Rep}_m(K, (\mathcal{B},\delta),
    \mathbf{T}).
  \end{equation}
\end{definition}

The following example shows that, with this normalization of the
representation numbers and some assumptions about
$(\mathcal{B},\delta)$, the standard Legendrian unknot has reduced
representation number $1$.

\begin{example}
  Let $(\mathcal{B}, \delta)$ be a DGA satisfying (B1) and (B2).
  Assume further that the the subset $T_1\subset (\mathcal{B}_0^m)^*$
  contains all elements of the form $-1+\delta x$ with $x \in
  \mathcal{B}_m^1$; in particular, we require that all such elements
  are invertible.  Under these assumptions, we compute the
  representation number for the standard Legendrian unknot, $U$.

  The DGA of $U$ has a single Reeb chord generator, $b$, with $|b|
  =1$, as well as an invertible generator $t^{\pm1}$ for the base
  point on $U$.  The differential is
  \[
  \partial b = t +1, \quad \partial t = 0.
  \]
  An $m$-graded representation $f:(\mathcal{A}(U), \partial)
  \rightarrow (\mathcal{B}, \delta)$ is uniquely determined by $f(b)
  \in \mathcal{B}_1^m$ since the representation equation
  $f\circ \partial (b) = \delta \circ f (b)$ implies
  \begin{equation} \label{eq:ft} f(t) = -1 + \delta \circ f(b).
  \end{equation}
  Moreover, $f(t)$ as defined by (\ref{eq:ft}) is invertible and
  belongs to $T_1$, so $f(b) \in \mathcal{B}^m_1$ may be chosen
  arbitrarily.  Thus, $\overline{\mbox{Rep}}_m(U,
  (\mathcal{B},\delta), \mathbf{T})$ is in bijection with $B^m_1$.
  Using (B2) and that $U$ has $\chi^k = - \chi^{-k}$ for $k\geq 2$, we
  compute
  \begin{align*}
    \mbox{Rep}_m(U, (\mathcal{B},\delta), \mathbf{T}) &=\big(|\mathcal{B}^m_{-1}|^{-1/2}|\mathcal{B}^m_{0}|^{1/2}|\mathcal{B}^m_{1}|^{-1/2} \big)\cdot |(\mathcal{B}^m_0)^* \cap \ker \delta|^{-1} \cdot |\mathcal{B}^m_{1}| \\
    & = |\mathcal{B}^m_{0}|^{1/2}\cdot |(\mathcal{B}^m_0)^* \cap \ker
    \delta|^{-1}.
  \end{align*}
  In addition, the reduced representation number satisfies $\rrep_m(U,
  (\mathcal{B},\delta), \mathbf{T})=1$.
\end{example}

\subsection{Legendrian invariance}
In establishing invariance when $m$ is odd, it will be useful to have
that the limit used in (\ref{eq:numbers}) depends only on the mod $2m$
degree of Reeb chords.  (This is false when $m$ is even.)

To give a precise statement, when $m$ is odd let
\begin{equation} \label{eq:chi1} \sigma_m= \sum_{l =0}^{m-1}(-1)^l
  \left\lvert\{a_i\,\vert\, \deg(a_i) =l \,\mbox{mod}\,
    m\}\right\rvert
\end{equation}
and for $1 \leq r<m$ set
\[\nu^r_{m} = \sum_{l=0}^{m-1}(-1)^l \left\lvert\{a_i\,\vert\, \deg(a_i)
  =r+l \,\mbox{mod}\, 2m\}\right\rvert.\]

\begin{lemma} \label{lem:odd} When $m$ is odd and $(b_N)$ is as in
  Lemma \ref{lem:bN},
  \[
  \lim_{N\rightarrow +\infty} b_N =
  \lvert\mathcal{B}_0^m\rvert^{-\sigma_m/2}\cdot
  \prod_{r=1}^{m-1}\lvert\mathcal{B}_r^m\rvert^{-\nu^r_m}
  \]
\end{lemma}
\begin{proof}
  Write $b_{2mN}= \beta_N^0\beta_N^1\cdots \beta_N^{m-1}$ where
  \[
  \beta^r_N = \prod_{\begin{array}{c}\lvert k\rvert\leq 2mN, \\ k =r
      \mod m \end{array}} \lvert B^m_k\rvert^{-\chi^k/2}
  \]
  Then, by definition of $B_k^m$, $\beta^r_N = \lvert
  B^m_r\rvert^{-c^r_N/2}$ where
  \[
  c^r_N = \sum_{\begin{array}{c}\lvert k\rvert\leq 2mN, \\ k =r \mod
      m \end{array}} \chi^k,
  \]
  and to establish the lemma we show that $\lim_{N \rightarrow
    +\infty} c^0_N = \sigma_m$ and $\lim_{N \rightarrow +\infty} c^r_N
  = 2\nu^r_m$ for $1 \leq r \leq m-1$.

  Let $\Z^{-\infty, \infty}$ denote the (free) $\Z$-module whose
  elements are bi-infinite sequences of integers, $(s_n)_{n \in \Z}$,
  with only finitely many non-zero terms, and let $e_i =
  (\delta_{n,i})_{n \in \Z}$ denote its standard basis.  The
  definition of $\chi^k$ extends to provide a homomorphism
  \[X^k: \Z^{-\infty,\infty} \rightarrow \Z, \quad X^k\left( (s_n)
  \right) = \sum_{l \geq 0} (-1)^{l} s_{k+l} + \sum_{l<0} (-1)^{l+1}
  s_{k+l}.\]
  Notice that
  \[X^k(e_{i-1} + e_{i}) = \left\{ \begin{array}{cr} 2, & i=k \\ 0, &
      i\neq k. \end{array} \right.\]
  Moreover, the sequence $C^r_N = \sum_{|k| \leq 2mN, \, k=r \mbox{
      mod $m$}} X^k$ is eventually constant when applied to any $(s_n)
  \in \Z^{-\infty, \infty}$.  [This is because once $k$ is far enough
  outside an interval containing the support of $(s_n)$, $X^k((s_n)) =
  -X^{k+m}((s_n))$.]  Therefore, the point-wise limit $C^r = \lim_{N
    \rightarrow \infty} C^r_N$ exists and is a homomorphism
  $\Z^{-\infty, \infty} \rightarrow \Z$.  Note that
  \begin{equation} \label{eq:Linear1} \chi^k = X^k((r_n)) \quad
    \mbox{and} \quad \lim_{N\rightarrow \infty} c^r_N = C^r((r_n)),
  \end{equation}
  where $r_n$ is the number of degree $n$ Reeb chords of $K$.
  Similarly, there are homomorphisms
  \begin{align*}
    S_m:\Z^{-\infty, \infty} \rightarrow \Z,   \quad & S_m((s_n)) = \sum_{l =0}^{m-1}(-1)^l \left(\sum_{i = l \, \mbox{mod} \,m} s_i \right) \\
    T_m^r: \Z^{-\infty, \infty} \rightarrow \Z, \quad & T_m^r((s_n)) =
    \sum_{l=0}^{m-1}(-1)^l \left(\sum_{i =r+l \,\mbox{mod}\, 2m} s_i
    \right)
  \end{align*}
  such that
  \begin{equation} \label{eq:Linear2} \sigma_m = S_m ((r_n)) \quad
    \mbox{and} \quad \nu_m^r = T_m^r((r_n)).
  \end{equation}

  In view of \eqref{eq:Linear1} and \eqref{eq:Linear2}, the proof is
  completed by verifying that $S_m = C^0$ and $2T^r_m = C^r$, $1 \leq
  r \leq m-1$. This is done by easily computing values on the basis
  $\{e_0\} \cup \{e_{i-1} + e_{i} \,|\, i \in \Z \}$ as
  \begin{align*}
    C^0(e_0) = S_m(e_0) = 1, \quad & C^0(e_{i-1}+e_i) = S_m(e_{i-1}+e_i) =\left\{ \begin{array}{cr} 2, & i=0 \, \mbox{mod}\, m \\ 0, & \mbox{else}, \end{array} \right. \\
    C^r(e_0) = 2T^r_m(e_0) = 0, \quad & C^r(e_{i-1}+e_i) = 2
    T^r_m(e_{i-1}+e_i) =\left\{ \begin{array}{cr} 2, & i=r \,
        \mbox{mod}\, m \\ 0, & \mbox{else}. \end{array} \right.
  \end{align*}
\end{proof}

Recall that $K$ is a Legendrian link with $m\,|2r(K)$, $m\geq 0$, and
$r(K)=0$ if $m$ is even; the Maslov potential $\mu$ for $K$ is
$\Z$-valued (resp. $\Z/2m$-valued) if $m$ is even (resp. odd).

\begin{proposition} \label{prop:Invariance} For any
  $(\mathcal{B},\delta)$ and $\mathbf{T}$ as above, the representation
  number $\mbox{Rep}_m(K, (\mathcal{B},\delta), \mathbf{T})$ is a
  Legendrian isotopy invariant of $(K, \mu)$.  If $K$ has only one
  component, then $\mbox{Rep}_m(K, (\mathcal{B},\delta), \mathbf{T})$
  is independent of $\mu$.
\end{proposition}

\begin{proof} Suppose that $(K,\mu)$ is equipped with basepoints and,
  in the case $m$ is odd, an additional $\Z$-valued Maslov potential as specified
  by the choices (A1)-(A3). When $K$ and $K'$ are related by a
  (basepoint and Maslov potential preserving) Legendrian isotopy,
  then, by Proposition~\ref{prop:DGAInvariance} (2), their $\Z$-graded
  DGAs, $(\mathcal{A}(K), \partial)$ and
  $(\mathcal{A}(K'), \partial')$, are stable tame isomorphic. That is,
  after performing some number of algebraic stabilizations (where two
  new generators of degrees $\lvert a\rvert = \lvert b\rvert +1 =k$
  with differentials $\partial a = b$ and $\partial b = 0$ are added)
  on $\mathcal{A}(K)$ and $\mathcal{A}(K')$ we obtain DGAs
  $S\mathcal{A}(K)$ and $S\mathcal{A}(K')$ that are isomorphic via a
  tame isomorphism, $\phi$.  In particular, $\phi$ fixes the
  invertible generators $t_i$ corresponding to basepoints, so that the
  pullback $\phi^*$ gives a bijection between the two relevant sets of
  representations of $S\mathcal{A}(K)$ and $S\mathcal{A}(K')$.  When
  $m$ is even, as all $t_i$ have degree $0$, the tameness of $\phi$
  implies that the degree distribution of generators, and hence also
  all $\chi^k$, agree for $S\mathcal{A}(K)$ and $S\mathcal{A}(K')$.
  It follows that representation numbers as in (\ref{eq:numbers})
  computed from $S\mathcal{A}(K)$ and $S\mathcal{A}(K')$ are equal.
  When $m$ is odd, all $t_i$ have their degrees divisible by $2m$,
  therefore a tame isomorphism between $S\mathcal{A}(K)$ and
  $S\mathcal{A}(K')$ preserves the mod $2m$ degree distribution of
  generators.  We can then apply Lemma \ref{lem:odd} to conclude the
  representation numbers are equal.

  Next, we check that $\mathcal{A}(K)$ and $S\mathcal{A}(K)$ produce
  the same representation numbers.  Stabilizing a DGA $\alg$ in degree
  $k$ changes $|\overline{\mbox{Rep}}_m(K, (\mathcal{B},\delta),
  \mathbf{T})|$ by a factor of $|\mathcal{B}^m_k|$, as the extensions
  of a given $m$-graded representation of $\alg$ to the stabilization
  arise from choosing $f(a) \in \mathcal{B}^m_k$ arbitrarily and
  putting $f(b) = \delta f(a)$.  Since stabilizing in degree $k$
  increases $\chi^{k}$ by $2$ and leaves all $\chi^{l}$ with $l \neq
  k$ unchanged, the product in (\ref{eq:numbers}) remains invariant.

  We next check independence of the choices (A1)-(A3).  For (A2), note
  that no additional choice is made when $m$ is even; when $m$ is odd,
  Lemma \ref{lem:odd} shows that the augmentation numbers only depend
  on the degree distribution of generators mod $2m$, and this is
  uniquely determined by the original $\Z/2m$-valued Maslov potential,
  $\mu$.  For (A3), independence of the choice of initial basepoints,
  $t_{e_i}$, follows since the location of basepoints can be
  cyclically permuted by a Legendrian isotopy of $K$.

  For (A1), we check that the representation numbers are independent
  of the number of basepoints used.  The algebra
  $(\mathcal{A}, \partial)$ with a single basepoint $t$ on a component
  $K_i \subset K$ is related to an algebra $(\mathcal{A}', \partial)$
  with multiple (cyclically ordered) basepoints $t_1, \ldots, t_l$ on
  $K_i$ as in Theorem \ref{prop:DGAInvariance} (3). As a result,
  $m$-graded representations $f: (\mathcal{A}', \partial') \rightarrow
  (\mathcal{B}, \delta)$ are in bijection with pairs $(g,h)$ of
  $m$-graded representations where $g:(\mathcal{A},\partial)
  \rightarrow (\mathcal{B}, \delta)$, $h:(\Z\langle t_1^{\pm1},
  \ldots, t_l^{\pm1} \rangle,0) \rightarrow (\mathcal{B}, \delta)$,
  and $g(t)= h(t_1\cdots t_l)$. All of the $t_i$ are cycles of degree
  $0$ mod $m$, so for a given $g$, $g(t) \in (\mathcal{B}^m_0)^* \cap
  \ker \delta$ and there are always $|(\mathcal{B}^m_0)^* \cap \ker
  \delta|^{l-1}$ ways to choose $h$.  [Put $h(t_1) = g(t)(h(t_2\cdots
  t_l))^{-1}$, and choose $h(t_i)$ with $ 2 \leq i \leq l$
  arbitrarily.]  This increase in the number of representations is
  counteracted by the $|(\mathcal{B}^m_0)^* \cap \ker \delta|^{-\ell}$
  factor in (\ref{eq:numbers}).

  The final statement about independence of $\mu$ when $K$ is
  connected follows since, in the connected case, any other Maslov
  potential for $K$ is related to $\mu$ by the addition of an overall
  constant.  Thus, the degree distribution of generators and the
  resulting grading on $\mathcal{A}(K)$ are independent of $\mu$.
\end{proof}

\subsection{Augmentation numbers}

Representations into a target of the form $(\mathcal{B}, \delta) =
(\mathbb{F}, 0)$ with $\mathbb{F}$ a field located in graded degree
$0$ are also called {\bf augmentations} of $(\mathcal{A}, \partial)$;
the set of all $m$-graded augmentations to $\mathbb{F}$ is denoted
$\overline{\mbox{Aug}}_m(K,\mathbb{F})$.  When $\mathbb{F}=
\mathbb{F}_q$ is the finite field of order $q$ and all $T_i=
\mathbb{F}_q^*$ the representation numbers defined above are called
{\bf augmentation numbers}, denoted $\mbox{Aug}_m(K,q)$, and satisfy
\[\mbox{Aug}_m(K,q) =q^{-\sigma_m/2} (q-1)^{-\ell} \cdot
|\overline{\mbox{Aug}}_m(K,\mathbb{F}_q)|\]
with $\ell$ the total number of basepoints on $K$ and
\begin{equation} \label{eq:chi2} \sigma_{m} = \lim_{N \rightarrow
    \infty} \, \sum_{\begin{array}{c} |k| \leq N \\ k=0\, \mbox{mod}
      \, m \end{array}} \chi^k.
\end{equation}
Note that in terms of the degree distribution of Reeb chords,
$(r_n)_{n \in \Z}$,
\begin{align}
  & \mbox{for $m=0$}, \quad \quad \quad \quad && \sigma_0  = \chi^0 = \sum_{l \geq 0} (-1)^lr_l + \sum_{l<0} (-1)^{l+1}r_l,  \label{eq:sigma11} \\
  & \mbox{for even $m>0$}, \quad \quad \quad \quad && \sigma_m =
  \sum_{k \in \Z} (2k+1)s_k, \quad \quad \mbox{where } s_k =
  \sum_{l=0}^{m-1}(-1)^lr_{mk+l}, \\
  & \mbox{for odd $m$}, \quad \quad \quad \quad && \sigma_m = \sum_{l
    =0}^{m-1}(-1)^l \left(\sum_{i = l \, \mbox{mod} \,m} r_i
  \right). \label{eq:sigma13}
\end{align}
The latter two formulas are verified as in the proof of Lemma
\ref{lem:odd}; $\sigma_m$ can be viewed as the unique homomorphism
$\Z^{-\infty,\infty} \rightarrow \Z$ that satisfies $\sigma_m(e_0) =
1$ and, for all $i\in \Z$, $\sigma_m(e_{i-1}+e_{i})=
\left\{\begin{array}{cr} 2, & i = 0\, \mbox{mod}\, m, \\ 0, &
    \mbox{else.} \end{array} \right.$
\begin{remark}
  From the calculation of $\sigma_m$ in
  (\ref{eq:sigma11})-(\ref{eq:sigma13}), one sees that the
  normalization for $Aug_m(K,q)$ used above agrees with that used in
  \cite{NRSS} and, in the case $q=2$ with $m=0$ or $m$ odd, in
  \cite{NgSabloff}.  The normalization used in \cite{HenryR} differs
  slightly.
\end{remark}

It was proven in \cite{HenryR} that the collection of $m$-graded
augmentation numbers and the $m$-graded ruling polynomial are
equivalent invariants.

\begin{theorem}[\cite{HenryR, NRSS}] \label{thm:HenryR} Let $K\subset
  J^1\R$ be a Legendrian link equipped with a $\Z/2r(K)$-valued Maslov
  potential, and let $m \geq 0$ have $m \,|\, 2r(K)$.  If $m$ is even,
  assume $r(K)=0$.  Then, for any prime power $q$,
  \[
  \mbox{Aug}_m(K,q) = R^m_K(q^{1/2}-q^{-1/2}).
  \]
\end{theorem}
\begin{proof}
  Remark 3.3 (ii) in \cite{HenryR} observes the formula in the case
  where $m=0$ or $m$ is odd as a variant of Theorem 1.1 of
  \cite{HenryR}.  When $m$ is even, the result is obtained in
  Proposition~16 of \cite{NRSS}.
\end{proof}

\begin{remark}
  In \cite{NRSSZ}, a unital $A_\infty$-category, $\mbox{Aug}_+(K)$, is
  introduced having as objects the augmentations of $K$; see \cite{BC}
  for a non-unital precursor.  It is shown in \cite{NRSS} that
  $\mbox{Aug}_0(K,q)$ can alternatively, and perhaps more naturally,
  be interpreted as the homotopy cardinality of $\mbox{Aug}_+(K)$.
  More general representation categories have been defined in
  \cite{CRGG}, but we do not pursue this direction any further in the
  present article.
  \end{remark}

\subsection{Representations on a DG-vector space}\label{sec:repDGVS}
Let $(V, d)$ be a $\Z$-graded vector space
and let ${d:V \rightarrow V}$ be a (cohomologically graded) differential
with degree $+1$ mod $m$, i.e. $d^2=0$ and $d(V_k^m) \subset
V^m_{k+1}$.  There is an induced differential on $\mathit{End}(V)$
given by the graded commutator with $d$,
\[
\delta: \mathit{End}(V) \rightarrow \mathit{End}(V), \quad \delta(T) =
[d,T] = d \circ T - (-1)^{|T|} T \circ d,
\]
where {\it when $m$ is odd we assume that $V$ is defined over a field
  of characteristic $2$}.  Moreover, $(-\mathit{End}(V), \delta)$ is a
$\Z/m$-graded DGA (with homological grading) where $-\mathit{End}(V)$
denotes $\mathit{End}(V)$ equipped with the negative of its standard
grading collapsed mod $m$,
\[
(-\mathit{End}(V))_k = \oplus_{l \in \Z} \mathit{Hom}(V_{l+k}, V_{l}),
\quad \quad -\mathit{End}(V) = \bigoplus_{k \in \Z/m}
(-\mathit{End}(V))_k^m.
\]
Assuming that $V$ is finite dimensional over the finite field $\mathbb{F}_q$, considering the {\bf graded dimension}
\[ {\bf n}: \Z \rightarrow \Z_{\geq 0}, \quad \mathbf{n}(j) =
\dim(V_j)
\]
leads to the formula
\[
|(-\mathit{End}(V))_k| = q^{\sum_{l \in \Z} \mathbf{n}(l+k) \cdot
  \mathbf{n}(l)}.
\]
In particular, $(-\mathit{End}(V), \delta)$ satisfies (B2), so
representation numbers are defined when $V$ is finite.

When $(\mathcal{B}, \delta)=(-\mathit{End}(V), \delta)$ we denote the
sets of representations, $\overline{\mbox{Rep}}_m(K,
(\mathcal{B},\delta), \mathbf{T}) $, by $\overline{\mbox{Rep}}_m(K,
(V,d), \mathbf{T})$, and when $V$ is finite we notate the associated
representation numbers as
\[
\mbox{Rep}_m(K, (V,d), \mathbf{T}).
\]

\begin{example}
  We compute here the $2$-graded representation number,
  $\mbox{Rep}_2(K, (\mathbb{F}_q^2,0), \mathit{GL}(2, \mathbb{F}_q))$,
  with $\mathbb{F}_q^2$ in grading $0$, for the Legendrian $m(5_2)$
  knot pictured in Figure \ref{fig:m52}.  Since $\mathbf{T}=\mathit{GL}(2,
  \mathbb{F}_q)$, no restriction is placed on the image of $t$.  We
  shorten notation to $\mbox{Rep}_2(K, \mathbb{F}_q^2):=
  \mbox{Rep}_2(K, (\mathbb{F}_q^2,0), \mathit{GL}(2, \mathbb{F}_q))$,
  and note that in Section \ref{sec:color} this is called the {\it total
    $2$-dimensional representation number}.  See also
  Example~\ref{ex:trefoil} for an example where $\mathbf{T}$ is a
  proper subset of $\mathit{GL}(n, \mathbb{F}_q)$.

\begin{figure}
\labellist
  \pinlabel {$t$} [l] at 384 260
  \pinlabel {$e_1$} [b] at 310 268
  \pinlabel {$e_2$} [b] at 322 188
  \pinlabel {$e_3$} [b] at 312 114
  \pinlabel {$e_4$} [b] at 302 28
  \pinlabel {$c_1$} [b] at 252 150
  \pinlabel {$c_2$} [b] at 238 70
  \pinlabel {$c_3$} [b] at 24 200
  \pinlabel {$a$} [b] at 278 236
  \pinlabel {$b$} [b] at 204 224
      \endlabellist
  
  \centerline{ \includegraphics[scale=.5]{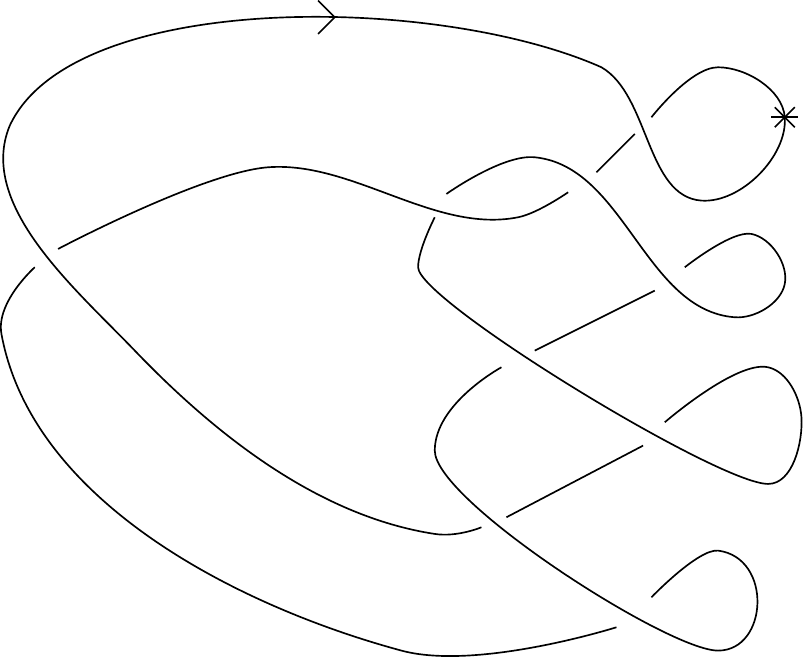} }
  \caption{ An $xy$-diagram (only topologically accurate) for a
    Legendrian $m(5_2)$ knot.}
  \label{fig:m52}
\end{figure}
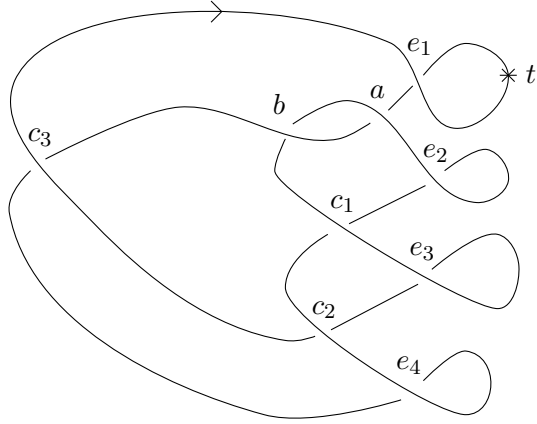

\begin{proposition} \label{prop:Rep2biject}
There is a bijection 
\[
\overline{\mbox{Rep}}_2(K, \mathbb{F}_q^n) \quad \leftrightarrow \quad \big\{ (A,B) \in \mathit{Mat}(n,\mathbb{F}_q) \times \mathit{Mat}(n,\mathbb{F}_q) \,|\, E_{-1}(AB) = \{0\} \big\}
\]
where $E_{-1}(AB)$ denotes the $(-1)$-eigenspace for $AB$.
\end{proposition}

\begin{proof}
  The generators of $\mathcal{A}(K)$ as pictured in Figure
  \ref{fig:m52} have degrees
  \[
  |b|=-2, \quad |c_1|=|c_2|=|c_3| = 0, \quad |e_1|=|e_2|=|e_3|=|e_4| =
  1, \quad |a|=2, \quad |t|=0,
  \]
  and the only non-zero differentials are
  \[
  \partial e_1=t+(-c_3)(1+ba), \quad \partial e_2=1+(1+ab)c_1,
  \quad \partial e_3=1+c_1c_2, \quad \partial e_4= 1+c_2c_3.
  \]
  We claim that the map $f \in \overline{\mbox{Rep}}_2(K,
  \mathbb{F}_q^2) \mapsto (A,B) = (f(a), f(b))$ gives the required
  bijection.  Note that the equation $0 = f\circ \partial(e_2)$ shows
  that $(1+f(a)f(b)) f(c_1) = -1$, so that $1+f(a)f(b)$ is invertible,
  i.e. $-1$ is not an eigenvalue of $AB$ as required.  The inverse map
  takes a pair $(A,B)$ with $E_{-1}(AB) = \{0\}$ to the $2$-graded
  representation defined by
  \begin{align*}
    & f(a)=A, \, f(b) =B, \, f(c_1) = -(1+AB)^{-1}, \\
    & f(c_2) = 1+AB, \, f(c_3)= -(1+AB)^{-1}, \,
    f(t)=-(1+AB)^{-1}(1+BA)
  \end{align*}
  where one should note $(1+BA)$ is invertible (making the definition
  of $f(t)$ valid) since $AB$ and $BA$ have the same eigenvalues.
  [For $\lambda \in \mathbb{F}_q^*$, if $v\neq 0$ and $v \in
  E_{\lambda}(AB)$, then $B(v) \neq 0$ and $B(v) \in E_{\lambda}(BA)$.]
\end{proof}

\begin{proposition} The total count of $2$-graded representations on
  $(\mathbb{F}_q^2,0)$ is
  \[
  |\overline{\mbox{Rep}}_2(K, \mathbb{F}_q^2)| = q^2 - q^3 + 2q^5 -
  q^6 - q^7 + q^8
  \]
\end{proposition}
\begin{proof}
  Use the bijection from Proposition \ref{prop:Rep2biject}, and note
  that the number of pairs of $2\times2$ matrices $(A,B)$ such that
  $-1$ is not an eigenvalue for $AB$ is
  \[
  |\overline{\mbox{Rep}}_2(K, \mathbb{F}_q^2)|= X_2 \cdot Y_2 +
  X_1\cdot Y_1 + X_0 \cdot Y_0
  \]
  where
  \[
  X_k = \left|\big\{ C \in \mathit{Mat}(2, \mathbb{F}_q) \,|\,
    \mathit{rank}(C)=k \quad \mbox{and} \quad E_{-1}(C)=\{0\}
    \big\}\right|
  \]
  and $Y_k$ is the number of ways to factor a rank $k$ matrix $C$ into
  a product $C=AB$ with $A,B \in \mathit{Mat}(2, \mathbb{F}_q)$.  For
  carrying out the counts $X_k$ and $Y_k$, it is useful to fix, for
  each line $\ell \in \mathbb{P}^1 := \mathbb{P}(\mathbb{F}_q^2)$, a
  complementary vector $v_\ell \in \mathbb{F}_q^2$ such that $\ell
  \oplus \mathit{Span}_{\mathbb{F}_q}\{v_{\ell}\} = \mathbb{F}_q^2$.
  The counts are as follows.
 
  \begin{enumerate}

  \item $X_2 = (q^2-1)(q^2-q) - \left[(q+1)(q^2-q-1) +1 \right]$:
    \quad Here, the different terms correspond to writing $X_2 =
    |\mathit{GL}(2,\mathbb{F}_q)| - (|W_1|+|W_2|)$ where $W_i= \{D \in
    \mathit{GL}(2,\mathbb{F}_q) \,|\, \dim E_{-1}(D) = i\}$.  It is
    standard that $|\mathit{GL}(2,\mathbb{F}_q)|=(q^2-1)(q^2-q)$ and
    $W_2 = \{-I\}$.  Finally, notice there is a bijection
    \[
    \{(\ell, w)\in\mathbb{P}^1\times\mathbb{F}_q^2 \,|\, \ell \in \mathbb{P}^1, \, w \neq -v_\ell, \,
    \mbox{and} \, w \notin \ell\} \, \leftrightarrow W_1
    \]
    where $(\ell,w)$ is mapped to the matrix $D$ with $E_{-1}(D) =
    \ell$ and $D(v_{\ell}) = w$.  For each $\ell \in \mathbb{P}^1$
    there are $(q^2-q-1)$ choices for $w$, so
    $|W_1|=|\mathbb{P}^1|\cdot (q^2-q-1) = (q+1)(q^2-q-1)$ as
    required.

  \item $X_1 = (q^2-q-1)(q+1)$: \quad We have $|X_1| = |R_1|- |S_1|$
    where $R_1$ is the set of all rank $1$ matrices and $S_1 \subset
    R_1$ are those that have $-1$ as an eigenvalue.  Note that $S_1$
    is in bijection with the set of ordered pairs $(\ell, w)$ with $w
    \in \ell$, so we can compute
    \[
    |R_1| -|S_1| = \left(1\cdot (q^2-1)+(q^2-1)\cdot q \right) -
    \left((q+1)\cdot q\right)
    \]
    where the $1$-st (resp. $2$-nd) term for $|R_1|$ is the count of
    rank 1 matrices whose first column is $0$ (resp. is non-zero).

  \item $X_0 = 1$: \quad This is obvious.

  \item $Y_2 = (q^2-1)(q^2-q)$: \quad When $C \in GL(2,
    \mathbb{F}_q)$, and $C =AB$, it must be the case that $A,B \in
    GL(2, \mathbb{F}_q)$.  As $B = A^{-1}C$ is uniquely determined by
    $A$ which may be chosen arbitrarily, we have $Y_2=|GL(2,
    \mathbb{F}_q)|$.

  \item $Y_1 = 2 (q^2-1)(q^2-q) +(q^2-1)q$: \quad When $C$ has rank
    $1$ and $C=AB$, there are $3$ disjoint possibilities for $(A,B)$:
    $A \in GL(2)$ and $B \in R_1$; $A \in R_1$ and $B
    \in GL(2)$; or $A \in R_1$ and $B \in R_1$.  In the first two
    cases, the rank 1 matrix is uniquely determined by $C$ and the
    rank $2$ matrix which may be chosen arbitrarily.  These two cases
    account for the $2(q^2-q)(q^2-q) = 2|GL(2)|$ term.  In the last
    case, when $C=AB$ with all matrices rank 1, note that
    $\ker(B) = \ker(C) = \ell$ for some $\ell \in
    \mathbb{P}^1$.  A $B$ with this property is then uniquely
    determined by $Bv_\ell$ which may be chosen arbitrarily in
    $\mathbb{F}_q^2 \setminus \{0\}$.  For each of these $(q^2-1)$
    choices of $B$, $A$ must be chosen to satisfy $A(B(v_\ell)) = C
    v_\ell$ and to have $1$-dimensional kernel.  Thus, we have
    $|\mathbb{P}^1 \setminus \{ \mathit{Span}\{v_\ell\}\}| = q$
    choices for $A$.

  \item $Y_0 = q^4 + (q^2-1)(q+1)q^2+ (q^2-1)(q^2-q)$: The collection
    of ordered pairs $(A,B)$ with $AB=0$ is subdivided into disjoint
    subsets $T_0 \sqcup T_1 \sqcup T_2$ where the subscript denotes
    the rank of $A$.  Note that $|T_0|=q^4$ since $A=0$ and $B$ may be
    arbitrary.  In $T_1$, given $A \in R_1$, $B$ must be chosen with
    $\mbox{im}\, B \subset \ker A$ so there are $q^2$ choices for $B$
    (since each column of $B$ may be an arbitrary vector in $\ker A$).
    Thus, $|T_1| = |R_1|q^2 = (q^2-1)(q+1)q^2$.  Finally, in $T_2$
    there are $|GL(2)|$ choices for $A$, while $B$ must be zero.

  \end{enumerate}

\end{proof}

With the count of representations in hand, we now compute the
$2$-graded representation number $\mbox{Rep}_2(K, \mathbb{F}_q^2)$.
As $(\mathcal{B}, \delta) = (-\mathit{End}(\mathbb{F}_q^2),0)$ is
concentrated in degree $0$, we have $|\mathcal{B}^m_k| =
\left\{ \begin{array}{cr} q^4, & k=0 \,\mbox{mod}\, m \\ 1, &
    \mbox{else}.\end{array} \right.$ The degree distribution for $K$
is $r_{-2}=1, r_0=3, r_1=4, r_2=1$ and all other $r_k=0$.  Therefore,
the factors from (\ref{eq:numbers}) are
\[
\lim_{N \rightarrow +\infty} \prod_{k \in \Z, \lvert k\rvert\leq N}
\lvert\mathcal{B}^2_k\rvert^{-(\chi^k/2)} =
(q^4)^{-(\chi^{-2}+\chi^0+\chi^2)/2} = (q^4)^{-(-1+1+1)/2} = q^{-2}
\]
and
\[
\lvert(\mathcal{B}^m_0)^* \cap \ker \delta\rvert^{-1} =
\left[(q^2-1)(q^2-q)\right]^{-1}.
\]
Thus, we have
\begin{equation} \label{eq:Count} \mbox{Rep}_2(K, \mathbb{F}_q^2) =
  q^{-2}\cdot \left[(q^2-1)(q^2-q)\right]^{-1} \cdot (q^2 - q^3 + 2q^5
  - q^6 - q^7 + q^8),
\end{equation}
and the reduced representation number is
\[
\widetilde{\mbox{Rep}}_2(K, \mathbb{F}_q^2) = q^{-4} ( q^2 - q^3 +
2q^5 - q^6 - q^7 + q^8) = q^{-2}-q^{-1}+2q^1-q^2 -q^3+q^4.
\]

\end{example}

\section{Braids and path matrices} \label{sec:Path}

A braid with only positive (in the sense of writhe) crossings may be
viewed as a Legendrian $\beta \subset J^1S^1$.  In \cite{Kalman},
K{\'a}lm{\'a}n associated, in the context of front projections, a
matrix to a positive braid called its path matrix.  These matrices
play a key role in describing the DGA of a Legendrian satellite formed
with positive braid as pattern.  In this section, we consider path
matrices in both the $xz$- and $xy$-projection setting, and prove in
Proposition~\ref{prop:Pxy1} an identity for the differential of an
$xy$-path matrix that will be crucial in Section \ref{sec:Sat}.
We then introduce classes of braids that we call ``path injective''
and ``path generated'' for which the path matrix controls certain
aspects of the DGA.  We conclude the section by considering positive
permutation braids which, when represented with reduced permutation
braid words, are shown to satisfy both of the above conditions
(Proposition~\ref{prop:permutation}) and to provide a decomposition of
$GL(n, \mathbb{F})$ via their path matrices
(Proposition~\ref{prop:bruhat}).  In fact, this coincides with the
celebrated Bruhat decomposition of $GL(n, \mathbb{F})$.

\subsubsection{Notational convention} 
We will use the following notation:

\begin{convention} \label{conv:1} For an algebra, $A$, denote by
  $\mathit{Mat}(n,A)$ the algebra of $n \times n$ matrices with
  entries from $A$.  For a linear map $\alpha: A_1 \rightarrow A_2$,
  we use the same notation $\alpha: \mathit{Mat}(n, A_1) \rightarrow
  \mathit{Mat}(n, A_2)$ for the linear map obtained from applying
  $\alpha$ entry-by-entry.
\end{convention}

Note that if $\alpha:A_1 \rightarrow A_2$ is an algebra homomorphism
or derivation, then so is $\alpha: \mathit{Mat}(n, A_1) \rightarrow
\mathit{Mat}(n, A_2)$.

\subsection{Positive braids and path matrices} \label{sec:PathMatrix}

A positive $n$-stranded braid with endpoints at $x=0$ and $x=1$ may be
considered as a Legendrian link $\beta \subset J^1S^1$; the diagram of
the braid is the front projection of the link.  We number braid
strands from top to bottom, so that positive braids may be written as
products (composed left to right) of the braid group generators
$\sigma_i$, $1\leq i \leq n-1$ where $\sigma_1 \in B_n$ and
$\sigma_{n-1} \in B_n$ are crossings of the top two strands and bottom
two strands respectively.

The resolution procedure from \cite{NgTraynor} provides an
$xy$-diagram for $\beta$ of a standard form.  Topologically, this
diagram is obtained from the front projection of $\beta$ by viewing
$S^1 = [0,1]/\{0,1\}$ and adding a full dip to the right of the
crossings of $\beta$.  See Figure \ref{fig:BraidRes}.  When
considering the algebra $\mathcal{A}(\beta)$, we label the crossings
of the $xy$-projection of $\beta$ as
\[p_1, \ldots, p_s; \quad x_{i,j}, y_{i,j}, \mbox{ for $1 \leq i < j
  \leq n$},\]
where the $p_i$ are the crossings from the front projection of $\beta$
labeled from left to right; $x_{i,j}$ (resp. $y_{i,j}$) is the
crossing in the left (resp. right) half of the dip where the $i$-th
strand of $\beta$ crosses over the $j$-th strand of $\beta$.

\labellist
\small 
\pinlabel $x$ [l] at 28 16 
\pinlabel $z$ [b] at 2 44 
\pinlabel $x$ [l] at 380 16 
\pinlabel $y$ [b] at 354 44
\endlabellist

\begin{figure}
  \centerline{ \includegraphics[scale=.6]{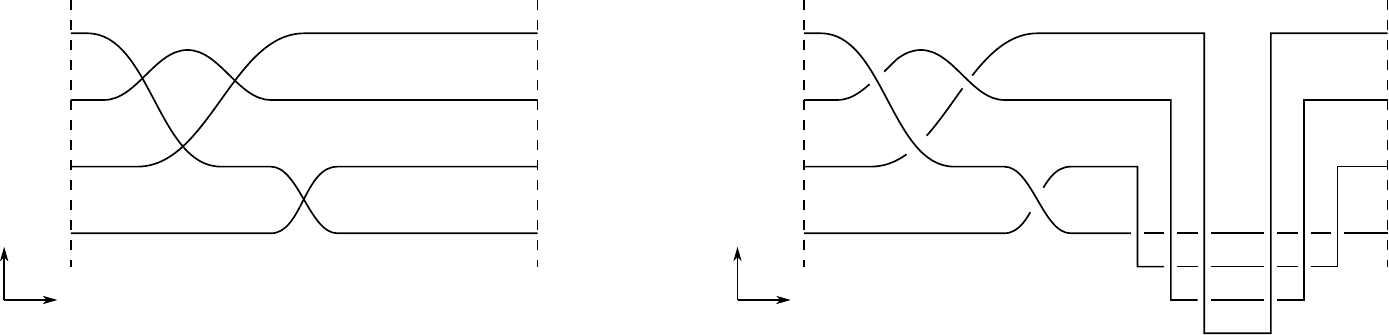} }
  \caption{ The positive braid $\beta = \sigma_1
    \sigma_2\sigma_1\sigma_3$ pictured in the front projection (left)
    and in its resolved Lagrangian projection (right).}
  \label{fig:BraidRes}
\end{figure}

For a positive braid $\beta \subset J^1S^1$ (with orientation), we
next consider several versions of the \emph{path matrices} introduced
in \cite{Kalman}.  Cut $S^1 \times \R$ along the vertical line where
$x=0$ and $x=1$ to view the $xz$- and $xy$- projections of $\beta$ as
subsets of $[0,1]\times \R$.  Call a continuous path $\gamma:[0,1]
\rightarrow [0,1]\times \R$ with $p_{[0,1]} \circ \gamma =
\mathit{id}_{[0,1]}$ and with image in the $xz$-projection
(resp. $xy$-projection) of $\beta$ an {\bf $xz$-section} (resp. {\bf
  $xy$-section}) of $\beta$.  We say that an $xz$- or $xy$-section,
$\gamma$, has {\bf downward} (resp. {\bf upward}) {\bf negative
  corners} if at all of the corners of $\gamma$ (these can only occur
at crossings of the projection) the region of $[0,1] \times \R$ that
lies above (resp. below) $\gamma$ covers a single quadrant of the
crossing with negative Reeb sign.  To each section $\gamma$, we
associate a sign $\iota(\gamma)$ and a word $w(\gamma)$ in the
generators of $\alg(\beta)$.
\begin{itemize}
\item For both $xz$- and $xy$-sections: $\iota(\gamma)$ is the product
  of orientation signs from the corners of $\gamma$ (as in Figure
  \ref{fig:Reeb}).

\item For $xz$-sections: When $\gamma$ has downward (resp. upward)
  negative corners we orient $\gamma$ in the increasing
  (resp. decreasing) $x$-direction.  Then, $w(\gamma)$ is the product
  of corners that $\gamma$ passes ordered according to the orientation
  of $\gamma$.

\item For $xy$-sections: The word $w(\gamma)$ is as in the $xz$-case
  except that the invertible generators $t_i^{\pm1}$ also appear
  whenever a basepoint is passed.  (The exponent $\pm1$ is determined
  by whether the orientations of $\gamma$ and $\beta$ agree or not at
  $t_i$.)
\end{itemize}

\begin{definition} \label{def:pathmatrix} The \textbf{left-to-right}
  $\mathbf{xz}$-\textbf{path matrix}, $P^{xz}_\beta \in
  \mathit{Mat}(n,\alg(\beta))$, has $(i,j)$-entry
  \[
  \sum_{\gamma} \iota(\gamma) w(\gamma)
  \]
  where the sum is over $xz$-sections of $\beta$ with downward
  negative corners with left (resp. right) endpoint on the $i$-th
  (resp. $j$-th) strand of $\beta$.  The $\mathbf{xy}$-\textbf{path
    matrix}, $P^{xy}_\beta\in \mathit{Mat}(n,\alg(\beta))$ is defined
  in the same manner using $xy$-sections.

  Similarly, {\bf right-to-left path matrices}, $Q^{xz}_{\beta}$ and
  $Q^{xy}_{\beta}$, are defined to have $(i,j)$-entry obtained from
  summing over sections with upward corners that have their right
  (resp. left) endpoint on the $i$-th (resp. $j$-th) strand of
  $\beta$.
\end{definition}

\begin{example} \label{ex:nonreduced} If the crossings in the braid in
  Figure~\ref{fig:BraidRes} are labeled $p_1,\ldots,p_4$ from left to
  right, then the left-to-right $xz$-path matrix is
  \[P^{xz}_\beta=\begin{pmatrix}
    (-1)^{\mu_3}p_2+(-1)^{\mu_2+\mu_3}p_1p_3&(-1)^{\mu_2}p_1&(-1)^{\mu_4}p_4&1\\
    (-1)^{\mu_3}p_3&1&0&0\\
    1&0&0&0\\
    0&0&1&0
  \end{pmatrix}.\] Here, the sign $(-1)^{\mu_i}$ is $+1$ (resp. $-1$)
  if the $i$-th strand (as ordered at the left side of the braid) is
  oriented to the right (resp. left).
\end{example}

\medskip

Typically, we will place a collection of basepoints $*_1, \ldots, *_n$
with corresponding generators $t_1, \ldots, t_n$ on the strands of
$\beta$ at an $x$-value just to the left of the crossings of $\beta$;
here, $*_i$ belongs to the $i$-th strand of $\beta$.  Fix a Maslov
potential $\mu_\beta$ for $\beta$, and let $(\mu_1, \ldots, \mu_n)$ be
the values of $\mu_\beta$ at the basepoints $*_1, \ldots, *_n$.  We
define (invertible) diagonal matrices
\begin{equation} \label{eq:SigDelta} \Sigma =
  \mathit{diag}((-1)^{\mu_1}, \ldots, (-1)^{\mu_n}), \quad
  \Delta = \mathit{diag}(t_{1}^{(-1)^{\mu_1}}, \ldots,
    t_n^{(-1)^{\mu_n}}).
\end{equation} 
Moreover, let $X$ and $Y$ denote {\it strictly upper triangular}
matrices with entries $x_{i,j}$ and $y_{i,j}$ for ${1\leq i<j\leq
n}$.

\begin{proposition} \label{prop:xypath} Suppose that $\beta \subset
  J^1S^1$ has basepoints placed as above, and that the $xy$-diagram of
  $\beta$ is related to the $xz$-diagram via the resolution procedure.
  Then,
  \[P^{xy}_{\beta} = \Delta \cdot P^{xz}_{\beta} \cdot (I +X
  \Sigma).\]
\end{proposition}
\begin{proof}
  Recalling that Maslov potentials are required to be even (resp. odd)
  along strands oriented to the right (resp. left), this follows
  examining the $xy$-projection.  [When a path with downward negative
  corners passes through the dip, because of Reeb signs at most a
  single corner can occur, and this can only happen at an $x_{i,j}$
  generator.]
\end{proof}

\begin{proposition} \label{prop:PathProp} Let $\alpha, \beta \subset
  J^1S^1$ be positive $n$-stranded braids with $xy$-diagrams and
  basepoints as above.  The path matrices satisfy:
  \begin{enumerate}
  \item $P^{xz}_{\alpha * \beta} = P^{xz}_{\alpha}\cdot
    P^{xz}_{\beta}$,
  \item $Q^{xz}_{\beta} = (P^{xz}_{\beta})^{-1}$, and $Q^{xy}_{\beta}
    = (P^{xy}_\beta)^{-1}$.
		\item The $(i,j)$-entry of $P^{xy}_\beta$ has degree $\mu_i-\mu_j$ in $\mathcal{A}(\beta)$.  
  \end{enumerate}
\end{proposition}
\begin{proof}
  Over $\Z/2$, the $xz$-statements from (1) and (2) are found in
  Section 3 of \cite{Kalman}, and are extended readily to the
  $\Z$-coefficient case.  For the $xy$-statement, note that
  \[Q^{xy}_{\beta} = W \cdot Q^{xz}_{\beta} \cdot \Delta^{-1} = W
  \cdot (P^{xz}_\beta)^{-1} \cdot \Delta^{-1},\] where $W$ is the
  right-to-left path matrix associated to the dip.  Right to left
  paths, $\gamma$, through the dip starting at the $i$-th strand and
  ending at the $j$-th strand can have many corners at the $x_{i,j}$
  (but none at the $y_{i,j}$), and have words of the form
  $((-1)^{\mu_{j_1}+1}x_{i,j_1})((-1)^{\mu_{j_2}+1}x_{j_1,j_2})\cdots
  ((-1)^{\mu_{j_{m}}+1}x_{j_{m},j})$ for some $i< j_1< \cdots < j_{m}
  < j$ with $m\geq 0$.  The sum of such words is precisely the
  $(i,j)$-entry of $(I+X\Sigma)^{-1} = I +(-X\Sigma) + (-X\Sigma)^2 +
  \cdots$, so $W = (I+X\Sigma)^{-1}$ and $Q^{xy}_{\beta}=
  (P^{xy}_\beta)^{-1}$ follows from Proposition~\ref{prop:xypath}.
	
  To prove (3), consider a path $\gamma$ oriented left-to-right with
  downward corners, and suppose ${w(\gamma) = c_1\cdots c_m}$.  (The
  $c_l$ are either $xz$-crossings of $\beta$ or $x_{r,s}$ generators.)
  Note that the value of the Maslov potential $\mu$ along $\gamma$
  changes from $p$ to $p-|c_i|$ when the corner at $c_i$ is passed, so
  ${\mu_j = \mu_i - \sum_{l}|c_l| = \mu_i - |w(\gamma)|}$.  The result
  follows.
\end{proof}

\begin{proposition} \label{prop:Pxy1} For any positive braid $\beta$,
  in the Chekanov-Eliashberg DGA of $\beta$
  we have the identity
  \begin{equation} \label{eq:Pxy1} \Sigma \cdot \partial(
    P^{xy}_\beta) = P^{xy}_\beta(Y\Sigma)- (Y\Sigma)P^{xy}_\beta.
  \end{equation}
\end{proposition}

\begin{proof}

  For the purpose of induction, we will prove a more general
  statement. Consider a subset $I \times \R \subset S^1\times \R$, with $I$ either an open interval or $S^1$, of some $xy$-diagram such that a
  positive braid $\beta$ appears in $I\times \R$ bordered on the left and on
  the right by two \emph{possibly distinct} dips.  We allow the
  possibility that the diagram may have other crossings outside of $I\times \R$. Let $X_0, Y_0$
  (resp. $X_1, Y_1$) denote strictly upper triangular matrices formed
  with the generators from the dip that occurs before (resp. after)
  $\beta$.  Let $\Sigma_0$ and $\Sigma_1$ be diagonal matrices with
  respective entries $(-1)^{\mu^0_i}$ and $(-1)^{\mu^1_i}$ where
  $\mu^0_i$ and $\mu^1_i$ denote the value of the Maslov potential on
  the $i$-th strand near the left and right dip respectively.  In this
  context, the path matrix of $\beta$, $P^{xy}_\beta$, has entries in
  the DGA of the larger diagram, and has its $(i,j)$-entry defined via
  considering paths that begin on the $i$-th strand just to the right
  of the $X_0,Y_0$ dip and proceed from left to right until they reach
  the $j$-th strand to the right of the $X_1,Y_1$ dip.

  We will show by induction on the length of $\beta$ that
  \begin{equation} \label{eq:Pxy2} \Sigma_0 \cdot \partial
    P^{xy}_{\beta} = P^{xy}_{\beta} (Y_1 \Sigma_1) - (Y_0 \Sigma_0)
    P^{xy}_{\beta}.
  \end{equation}
  When the dip to the left and right of $\beta$ are actually the same
  dip, this reduces to (\ref{eq:Pxy1}).

  \medskip

  \noindent {\bf Base Case:} $\beta = \emptyset$.  We compute from the
  definition
  \[
  \Sigma_0 \cdot \partial(I + X_1\Sigma_1) =
  (I+X_1\Sigma_1)(Y_1\Sigma_1) - \Delta^{-1} (Y_0\Sigma_0) \Delta
  (I+X_1\Sigma_1).
  \]
  [If necessary, see \cite{Sab1} or \cite{L1} where the disks that
  contribute to the differential are pictured.]  Thus, using
  Proposition~\ref{prop:xypath} and the signed Leibniz rule we compute
  \begin{align*}
    \Sigma_0 \cdot \partial( P^{xy}_\beta) &= \Sigma_0 \cdot \Delta \partial(I+X_1\Sigma_1) = \Delta (\Sigma_0 \cdot \partial(I+X_1\Sigma_1))  \\
    &= \Delta(I+X_1\Sigma_1)(Y_1\Sigma_1) - (Y_0\Sigma_0) \Delta
    (I+X_1\Sigma_1) = P^{xy}_\beta(Y_1 \Sigma_1) - (Y_0\Sigma_0)
    P^{xy}_\beta.
  \end{align*}

  \medskip

  The inductive step, has two parts.

  \medskip

  \noindent {\bf Step 1:} We establish the result for $\beta =
  \sigma_k$, where $1\leq k<n$. Let $p$ be the generator associated
  with the crossing in $\sigma_k$ of strands $k$ and $k+1$. Let $V_k$
  be the $n\times n$ identity matrix with the $2\times2$ block
  $\begin{pmatrix} (-1)^{\mu^1_k}p&1\\1&0\end{pmatrix}$ replacing the
  $k$-th and $(k+1)$-th entries on the diagonal. Since
  $P^{xz}_{\sigma_k}=V_k$,
  \[P^{xy}_{\sigma_k}=\Delta V_k(I+X_1\Sigma_1)\] and one can check
  that
  \[\Sigma_1\partial
  X_1\Sigma_1=(I+X_1\Sigma_1)Y_1\Sigma_1-V_k^{-1}\Delta^{-1}\widehat{Y}_0\Sigma_0\Delta
  V_k(I+X_1\Sigma_1),\] where $\widehat{Y}_0$ is $Y_0$ with the
  $(k,k+1)$-entry replaced by a zero.  [See \cite{Sab1} or \cite{L1}
  for individual disks.]  Noting that $\mu^1_{k+1}+\lvert
  p\rvert\equiv \mu^1_k\mod2$ and using the signed Leibniz rule, we
  then compute
  \begin{align}
    \Sigma_0\partial P^{xy}_{\sigma_k}&=\Sigma_0\Delta((\partial((-1)^{\mu^1_k}p))E_{k,k})(I+X_1\Sigma_1)\label{eqn:permutationCheck}\\
    &\quad+\Sigma_0\Delta
    V'_k\left(\Sigma_1(I+X_1\Sigma_1)Y_1\Sigma_1-\Sigma_1V_k^{-1}\Delta^{-1}\widehat{Y}_0\Sigma_0\Delta
      V_k(I+X_1\Sigma_1)\right),\nonumber
  \end{align}
  where $V'_k$ is $V_k$ with $\begin{pmatrix}
    (-1)^{\mu^1_k}p&1\\1&0\end{pmatrix}$ replaced by
  $\begin{pmatrix}(-1)^{\mu^1_{k+1}}p&1\\1&0\end{pmatrix}$ and
  $E_{i,j}$ is the $n\times n$ matrix with a single nonzero entry of
  $1$ as the $(i,j)$-entry.  Since
  \[\partial
  p=(-1)^{\mu^0_k+1}t_k^{(-1)^{\mu^0_k+1}}y^0_{k,k+1}t_{k+1}^{(-1)^{\mu^0_{k+1}}},\]
  the first term in \eqref{eqn:permutationCheck} is
  \[
  \left(\left((-1)^{\mu^1_k+1}y^0_{k,k+1}t_{k+1}^{(-1)^{\mu^0_{k+1}}}\right)E_{k,k}\right)(I+X_1\Sigma_1)
  = -\left(\left( y^0_{k,k+1}E_{k,k+1} \right) \Sigma_0 \Delta V_{k}
  \right)(I+X_1\Sigma_1),
  \]
  as $\mu^1_k = \mu^0_{k+1}$.  One can check that $\Sigma_0\Delta
  V'_k\Sigma_1=\Delta V_k$ and so $\Sigma_0\Delta
  V'_k\Sigma_1V_k^{-1}\Delta^{-1}=I$. Therefore
  \begin{align*}
    \Sigma_0\partial P^{xy}_{\sigma_k}&=-\left( y^0_{k,k+1}E_{k,k+1} \right) \Sigma_0 \Delta V_{k} (I+X_1\Sigma_1)+\Delta V_k(I+X_1\Sigma_1)Y_1\Sigma_1\\
    &\quad-\widehat{Y}_0\Sigma_0\Delta V_k(I+X_1\Sigma_1) \\
    & = \Delta V_k(I+X_1\Sigma_1)Y_1\Sigma_1 - Y_0\Sigma_0\Delta V_k(I+X_1\Sigma_1) \\
    & = P^{xy}_{\sigma_k}(Y_1\Sigma_1)-(Y_0\Sigma_0)P^{xy}_{\sigma_k}
  \end{align*}
  as desired.

  \medskip

  \noindent {\bf Step 2:} We show that if (\ref{eq:Pxy2}) holds for
  $\beta_1$ and $\beta_2$, then it holds for $\beta_1 * \beta_2$.

  Consider algebras $(\mathcal{A}_1, \partial_1)$ and
  $(\mathcal{A}_2, \partial_2)$ corresponding to the two $xy$-diagrams
  in Figure \ref{fig:A2}.  For the algebra $\mathcal{A}_1$ the diagram
  has three sets of dips with corresponding matrices of generators
  $X_k, Y_k$ and Maslov potential matrices $\Sigma_k$, for $0 \leq k
  \leq 2$; the braid $\beta_k$ appears between the $k-1$ and $k$ dips,
  for $1 \leq k \leq 2$. [The resolution procedure from
  \cite{NgTraynor} can always be modified to add extra sets of dips to
  the $xy$-projection between any crossings of the $xz$-projection;
  see \cite{Sab1}.]  Moreover, basepoints $t^k_i$, $1 \leq i \leq n$,
  are placed just to the left of each $\beta_k$, and are placed into
  matrices $\Delta_{k-1}$ with exponents $(\pm1)$ given by the
  corresponding entries of $\Sigma_{k-1}$.  For $\mathcal{A}_2$, the
  dip between $\beta_1$ and $\beta_2$ is removed, and there are no
  basepoints between $\beta_1$ and $\beta_2$.  The generators for
  $\mathcal{A}_2$ are the same as for $\mathcal{A}_1$, except there
  are no $\Delta_1$, $X_1$ and $Y_1$ generators.

  \labellist
  \small 
  \pinlabel $*$ at 148 96 
  \pinlabel $*$ at 148 80 
  \pinlabel $*$ at 148 64 
  \pinlabel $*$ at 148 48 
  \pinlabel $*$ at 136 224 
  \pinlabel $*$ at 136 240 
  \pinlabel $*$ at 136 256 
  \pinlabel $*$ at 136 272
  \pinlabel $*$ at 376 224 
  \pinlabel $*$ at 376 240 
  \pinlabel $*$ at 376 256 
  \pinlabel $*$ at 376 272 
  \pinlabel $\beta_1$ [b] at 200 240
  \pinlabel $\beta_1$ [b] at 232 64 
  \pinlabel $\beta_2$ [b] at 440 240
  \pinlabel $\beta_2$ [b] at 376 64 
  \pinlabel $Y_0$ [t] at 96 176
  \pinlabel $\Delta_0$ [t] at 136 216 
  \pinlabel $X_1$ [t] at 272 176
  \pinlabel $Y_1$ [t] at 336 176 
  \pinlabel $\Delta_1$ [t] at 376 216
  \pinlabel $X_2$ [t] at 512 176 
  \pinlabel $Y_2$ [t] at 576 176
  \pinlabel $Y_0$ [t] at 96 0 
  \pinlabel $\Delta_0$ [t] at 152 40
  \pinlabel $X_2$ [t] at 512 0 
  \pinlabel $Y_2$ [t] at 576 0
\endlabellist

\begin{figure}
  \centerline{ \includegraphics[scale=.5]{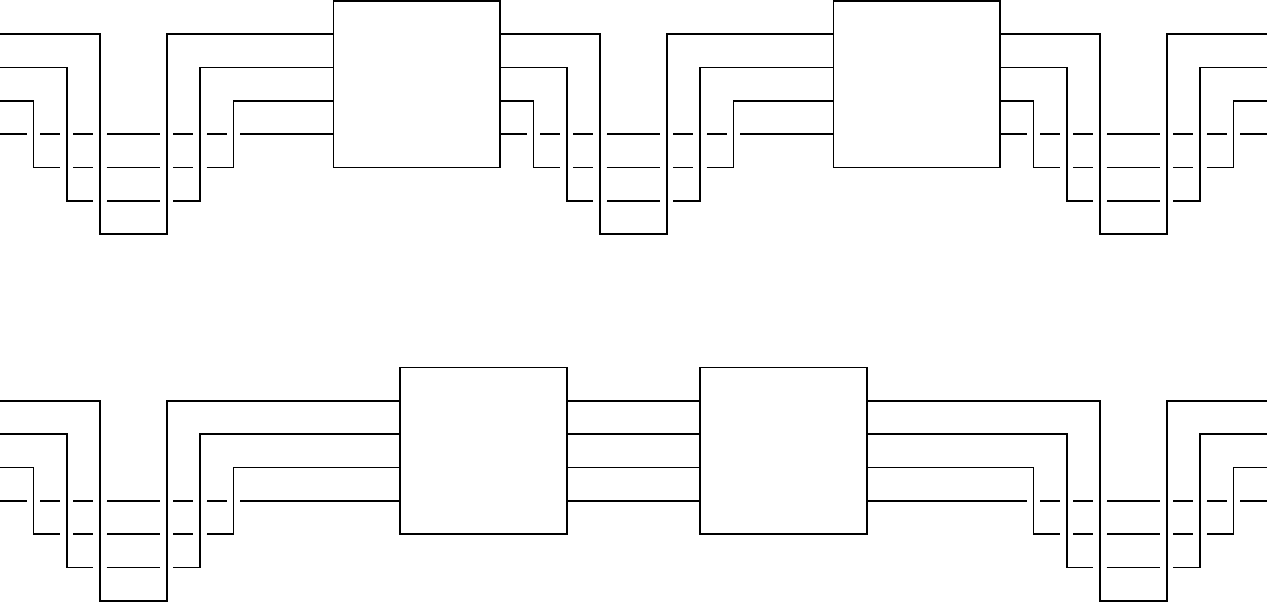} }
  \caption{ The $xy$-diagrams for the algebras
    $(\mathcal{A}_1, \partial_1)$ and $(\mathcal{A}_2, \partial_2)$.}
  \label{fig:A2}
\end{figure}

In $\mathcal{A}_1$, (i.e. in the context of the upper diagram from
Figure \ref{fig:A2}), the respective path matrices for $\beta_1$ and
$\beta_2$ are
\begin{equation} \label{eq:100} P^{xy}_{\beta_1} = \Delta_0
  P^{xz}_{\beta_1} (I+X_1 \Sigma_1), \quad \quad P^{xy}_{\beta_2} =
  \Delta_1 P^{xz}_{\beta_2}(I+X_2 \Sigma_2).
\end{equation}
In $\mathcal{A}_2$, the path matrix for $\beta_1 * \beta_2$ is
\begin{equation} \label{eq:101} P^{xy}_{\beta_1 * \beta_2} = \Delta_0
  P^{xz}_{\beta_1* \beta_2} (I+X_2 \Sigma_2) = \Delta_0
  P^{xz}_{\beta_1}P^{xz}_{\beta_2} (I+X_2 \Sigma_2).
\end{equation}
There is a DGA homomorphism $\Phi: (\mathcal{A}_1, \partial_1)
\rightarrow (\mathcal{A}_2, \partial_2)$ that is the composition
\[\mathcal{A}_1 \rightarrow \mathcal{A}_1' :=
\mathcal{A}_1\big|_{\Delta_1 = I}/\langle X_1, \partial X_1 \rangle
\stackrel{\cong}{\rightarrow} \mathcal{A}_2.\] The first map
specializes the $t^1_i$ to $1$, and then projects to the quotient by
the $2$-sided ideal, $\langle X_1, \partial X_1 \rangle$, generated by
the entries of $X_1$ and their differentials.  [Note that $\langle
X_1, \partial X_1 \rangle$ is a differential ideal.]
The map from $\psi:\mathcal{A}_2 \rightarrow \mathcal{A}_1'$ that
takes generators of $\mathcal{A}_2$ to the equivalence class in
$\mathcal{A}_1'$ of the corresponding generator from $\mathcal{A}_1$
is an algebra isomorphism (because $\partial x^1_{i,j} = y^1_{i,j} +
...$), and $\psi^{-1}$ is the second map in the definition of $\Phi$.

\medskip

\noindent {\bf Claim: $\psi$ is in fact a DGA isomorphism.}  The
required formula $\partial_1' \circ \psi= \psi \circ \partial_2$ with
$\partial_1'$ the differential on $\mathcal{A}_1'$ induced by
$\partial_1$ is essentially Sivek's van Kampen Theorem from
\cite{Sivek}.  We outline the argument in our setting:

The disks that contribute to $\partial_1$ and $\partial_2$ are
identical except that (1) any disk for $\partial_2$ whose image
intersects a vertical line $\ell$ separating $\beta_1$ and $\beta_2$
do not appear in $\partial_1$, and (2) disks for $\partial_1$ with
punctures at any of the $x^1_{i,j}$ or $y^1_{i,j}$ do not appear in
$\partial_2$.  The disks of type (2) have one of the following forms:

\begin{enumerate}
\item[(2a)] Disks with a positive puncture at $x_{i,j}$.  Here, note
  that we can write
  \begin{equation} \label{eq:xyw}
    \partial_1 x^1_{i,j} = y_{i,j} + w_{i,j} + w_x 
  \end{equation}
  where $w_{i,j}$ (resp. $w_x$) is the contribution from disks without
  any (resp. with at least one) negative punctures at some other
  $x^1_{i,j}$. Note that the disks that produce $w_{i,j}$ are in
  bijection with ``bordered disks'' in the $\mathcal{A}_2$ diagram
  that in place of the positive puncture at $x_{i,j}$ have a right
  boundary segment mapped to the line segment $\ell_{i,j} \subset
  \ell$ that has endpoints on the $i$-th and $j$-th strands of the
  Legendrian.
\item[(2b)] Disks with at least one negative puncture at the
  $y_{i,j}$.  These are in bijection with ``bordered disks'' in
  $\mathcal{A}_2$ that have left boundary segments along the
  $\ell_{i,j}$ in place of the punctures at the $y_{i,j}$.
\end{enumerate}
In $\mathcal{A}'_1$, we have $[x_{i,j}] = [\partial x_{i,j}] = 0$, and
using (\ref{eq:xyw}) this gives that $[y_{i,j}] = [w_{i,j}]$.  Thus,
for a generator $a$ of $\mathcal{A}_2$,
$\psi^{-1}\circ\partial_1'\circ \psi(a)$ is computed from
$\partial_1a$ by replacing all occurrences of $x_{i,j}$ with $0$ and
$y_{i,j}$ with $w_{i,j}$.  Since the $xy$-diagram used for
$\mathcal{A}_2$ was formed using the resolution procedure and the dips
prevent disks from wrapping around the $S^1$ factor, it can be shown,
cf. \cite{Sivek}, that any disk of Type (1) is obtained by starting
with a bordered disk of Type (2b) and then gluing bordered disks of
Type (2a) to it along the vertical segments $\ell_{i,j}$ corresponding
to negative punctures at the $y_{i,j}$.  It follows that the above
procedure for computing $\psi^{-1}\circ \partial_1'\circ \psi(a)$
produces precisely $\partial_2(a)$ as required.

\medskip

With the claim established so that $\psi$, and hence $\Phi$, is a DGA
map, we return to Step 2.  Comparing (\ref{eq:100}) and (\ref{eq:101})
with the definition of $\Phi$ leads to
\begin{equation} \label{eq:beta1beta2} \Phi(P^{xy}_{\beta_1}
  P^{xy}_{\beta_2}) = P^{xy}_{\beta_1 * \beta_2}.
\end{equation}
In $(\mathcal{A}_1, \partial_1)$, we use the Leibniz rule (observing
that the $(i,j)$-entry of the path matrix $P^{xy}_{\beta_1}$ has
degree $\mu^0_{i} - \mu^1_j$) and the inductive hypothesis to compute
\begin{align*}
  \Sigma_0 \cdot \partial_1(P^{xy}_{\beta_1} P^{xy}_{\beta_2}) &= \Sigma_0\left[ (\partial_1P^{xy}_{\beta_1}) P^{xy}_{\beta_2} + (\Sigma_0 P^{xy}_{\beta_1} \Sigma_1)( \partial_1 P^{xy}_{\beta_2})\right] \\
  &= (\Sigma_0 \cdot \partial_1 P^{xy}_{\beta_1}) P^{xy}_{\beta_2} +P^{xy}_{\beta_1} ( \Sigma_1 \cdot \partial_1 P^{xy}_{\beta_2}) \\
  &= \left[P^{xy}_{\beta_1}(Y_1 \Sigma_1) - (Y_0 \Sigma_0) P^{xy}_{\beta_1} \right] \cdot P^{xy}_{\beta_2} + P^{xy}_{\beta_1} \cdot\left[ P^{xy}_{\beta_2} (Y_2 \Sigma_2) - (Y_1 \Sigma_1) P^{xy}_{\beta_2}\right] \\
  &= P^{xy}_{\beta_1}P^{xy}_{\beta_2} (Y_2 \Sigma_2) - (Y_0 \Sigma_0)
  P^{xy}_{\beta_1} P^{xy}_{\beta_2}.
\end{align*}
Applying $\Phi$ to the above calculation and using the identity
(\ref{eq:beta1beta2}), gives the required equality
\[
\Sigma_0\cdot\partial_2 P^{xy}_{\beta_1 * \beta_2} =
P^{xy}_{\beta_1*\beta_2} (Y_2 \Sigma_2) - (Y_0 \Sigma_0)
P^{xy}_{\beta_1*\beta_2}.
\]

\end{proof}

\subsection{Path matrices and augmentations} \label{sec:pathaug}

Consider a positive braid, $\beta \subset J^1S^1$, (with $xy$-diagram
obtained from resolution and basepoints $*_1, \ldots, *_n$ placed to
the left of the crossings of $\beta$ as in Section
\ref{sec:PathMatrix}).  Notice that the entries of the path matrix
$P^{xy}_{\beta}$ belong to the (unital) sub-algebra $\mathcal{B} \subset
\mathcal{A}(\beta)$ generated by $p_1, \ldots, p_s$; $x_{i,j}$ for $1
\leq i < j \leq n$; and $t_1^{\pm1}, \ldots, t_n^{\pm1}$,
i.e. $\mathcal{B}$ is the sub-algebra spanned by words not containing
any of the $y_{i,j}$.

As a result, for any (unital, associative) ring $R$, a ring
homomorphism
\[
\alpha: \mathcal{B} \rightarrow R
\]
produces an invertible $n \times n$-matrix with entries in $R$ by
applying $\alpha$ entry-by-entry to $P^{xy}_\beta$.
\begin{definition}\label{def:PathSub}
  Given a ring $R$, denote by $\mathit{Ring}(\mathcal{B}, R)$ the set
  of ring homomorphisms from $\mathcal{B}$ to $R$.  We say that
  $\beta$ is {\bf $R$-path injective} if the map
  \begin{equation} \label{eq:pathmap}
  \mathit{Ring}(\mathcal{B}, R) \rightarrow \mathit{Mat}(n,R), \quad
  \alpha \mapsto \alpha(P^{xy}_\beta)
  \end{equation}
  is injective.  We denote the image of this map as $B_\beta \subset
  \mathit{Mat}(n,R)$ and call it the {\bf path subset} of $\beta$.
\end{definition}

We will also use an $m$-graded version of the path subset.  The Maslov
potential on $\beta$ provides a $\Z$-grading on $\mathcal{B} \subset
\mathcal{A}(\beta)$, and we view $R$ as sitting in grading degree $0$.
For fixed $m \geq 0$, we define the {\bf $m$-graded path subset} of
$\beta$, $B^m_\beta$, to be the image under the map (\ref{eq:pathmap})
of the subset $\mathit{Ring}_m(\mathcal{B}, R) \subset
\mathit{Ring}(\mathcal{B},R)$ consisting of ring homomorphisms that
preserve grading mod
$m$.

\begin{definition} We say that $\beta$ is {\bf path generated}, if
  $\mathcal{B}$ is generated as a ring by $\Z$, the $t_i^{\pm1}$, and
  the entries of $P^{xy}_\beta$.
\end{definition}

\begin{remark}
  For Theorem \ref{thm:A} to hold, the path generated condition may be
  replaced with a seemingly weaker ``$R$-path generated'' where we
  require that the free product $R*\mathcal{B}$ is generated by
  $R$, $t_i^{\pm1}$, and $P^{xy}_\beta$.  However, we currently do not
  know of examples of braids that are $R$-path generated but not
  $\Z$-path generated.
\end{remark}

The following braids will give us examples of braids which are both
$R$-path injective and path generated.

\begin{definition}
  A {\bf positive permutation braid} is a positive braid where every
  pair of strands crosses at most once.
\end{definition}

Positive permutation braids are in one-to-one correspondence with
elements of the symmetric group $S_n$ (see \cite{Elrifai}). In
particular, if $\beta=\sigma_{i_1}\cdots\sigma_{i_r}$, where the
$\sigma_i$ are the standard generators of the braid group $B_n$, then
$\beta$ corresponds to the permutation with each $\sigma_i$ replaced
by the transposition of $i$ and $i+1$. Equivalently, the permutation
$\pi$ associated to a braid $\beta$ has $\pi(i) =j$ if the strand in
position $i$ at the \emph{right} side of $\beta$ is in position
$j$ when followed \emph{left} to the left side of $\beta$.  In
\cite{Garside}, positive permutation braids play an important role in
Garside's solution of the word and conjugacy problems in the braid
group $B_n$.

After a Legendrian isotopy, it is always possible to represent a
positive braid by a braid word where the product
$\sigma_i\sigma_{i+1}\sigma_i$ does not appear for any $i$ by
repeatedly replacing $\sigma_i\sigma_{i+1}\sigma_i$ with
$\sigma_{i+1}\sigma_i\sigma_{i+1}$, see
Figure~\ref{fig:triangleMove}. Such braid words are called {\bf
  reduced braid words}.  The braid in Figure~\ref{fig:BraidRes} is an
example of a positive permutation braid which is not reduced and
corresponds to the permutation
$(1\,2)(2\,3)(1\,2)(3\,4)=(1\,3\,4)$. The reduced braid word for the
braid in Figure~\ref{fig:BraidRes} is
$\sigma_2\sigma_1\sigma_2\sigma_3$. The basic fronts $A_m$, $m\geq 1$
(as in Figure \ref{fig:Basic}) are examples of reduced positive
permutation braids.

\begin{figure}[t]
  \includegraphics[width=3in]{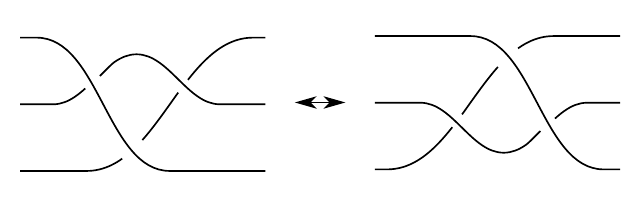}
  \caption{Legendrian isotopic braids. The right is a reduced braid
    word.}
  \label{fig:triangleMove}
\end{figure}

Restated in terms of our notation and with signs added, the following
proposition from K\'alm\'an gives us the form of the $xz$-path matrix
for a reduced permutation braid word for a positive permutation braid.
We use the notation $P_\pi = \sum_{i =1}^n E_{\pi(i),i}$ for the
permutation matrix associated $\pi \in S_n$.

\begin{proposition}[\cite{Kalman} Proposition~3.5] \label{prop:kalman} Let
  $\beta$ be a reduced positive permutation braid and let $\pi\in S_n$
  be the associated permutation. The path matrix $P^{xz}_\beta$ is
  obtained from the permutation matrix $P_\pi$ as follows: Changes are
  only made to entries that are above the $1$ in their column and to
  the left of the $1$ in their row. At each such position, a single
  crossing label multiplied by $(-1)^{\mu_k}$, where $\mu_k$ is the
  value of the Maslov potential on the lower strand incident to the
  crossing, appears in $P^{xz}_\beta$.
\end{proposition}

In other words, there is exactly one $1$ in each row and column of
$P^{xz}_\beta$ and if an entry in $P^{xz}_\beta$ is $1$, then all
entries either below or to the right of it are zero and all entries
either above or to the left of it are either $\pm p_k$ for some $k$ or
zero. The positions that carry different entries in $P^{xz}_\beta$
and $P_\pi$ are in one-to-one correspondence with the crossings of
$\beta$.

\begin{example}
  Let $\beta$ be the braid in Figure~\ref{fig:BraidRes} and let
  $\beta'=\sigma_2\sigma_1\sigma_2\sigma_3$ be the reduced braid word
  for $\beta$. If the crossings of $\beta$ are labeled from left to
  right by $p_1,\ldots,p_4$ and the crossings of $\beta'$ by
  $q_1,\ldots,q_4$, then the path matrix of $\beta$, $P^{xy}_{\beta}$,
  is as in Example \ref{ex:nonreduced}, while in the reduced case
  \[P^{xz}_{\beta'}=\begin{pmatrix}
      (-1)^{\mu_3}q_2&(-1)^{\mu_2}q_3&(-1)^{\mu_4}q_4&1\\
      (-1)^{\mu_3}q_1&1&0&0\\
      1&0&0&0\\
      0&0&1&0
    \end{pmatrix}.\]
\end{example}

\medskip

\begin{proposition} \label{prop:permutation} If $\beta$ is a reduced
  positive permutation braid and $R$ is a unital ring, then $\beta$ is
  $R$-path injective and path generated.
\end{proposition}

\begin{proof}
  Let $P^{xz}_\beta=(P_{i,j})_{1\leq i,j\leq n}$. By
  Proposition~\ref{prop:kalman}, for each $k$ there exist unique $i,j$
  such that $P_{ij}=\pm p_k$. Relabel the $p_k$ by their position in
  $P^{xz}_\beta$: let $p^j_i$ be the $p_k$ in the $i$-th row and
  $j$-th column of $P^{xz}_\beta$ (for those positions $(i,j)$ where
  the entry is $\pm p_k$ rather than $0$ or $1$). We then have that
  \[P_{i,j}=\begin{cases}
    1&\text{if }\pi(j)=i,\\
    0&\text{if }\pi^{-1}(i)>j,\\
    0&\text{if }i>\pi(j), \\
    \pm p^j_i&\text{else.}
  \end{cases}\]
  This is because the entry $P_{i,j}$ appears above (resp. to the left of) the $1$ in the $j$-th column (resp. $i$-th row) of $P_\pi$ if and only if $i>\pi(j)$ (resp. $\pi^{-1}(i)>j$).

  \medskip\noindent{\bf (path generated)} Recalling that
  $P^{xy}_\beta=\Delta\cdot P^{xz}_\beta\cdot(I+X\Sigma)$, we note
  that the entries of $P^{xy}_\beta$ are just entries of $P^{xz}_\beta
  \cdot (I+X \Sigma)$ multiplied on the left by some $t_i^{\pm1}$.
  Therefore, to show that $\beta$ is path generated it suffices to
  show that $\mathcal{B}$ is generated as a ring by $\Z$,
  $t_i^{\pm1}$, and the entries of $P^{xz}_\beta \cdot (I+X \Sigma)$.
	
  Let $\mathcal{B}_k \subset \mathcal{B}$ denote the subring generated
  by $\Z$, the $t_i^{\pm1}$, and entries from the first $k$ columns of
  $P^{xz}_\beta(I+X\Sigma)$. We will prove by induction on $k$ that
  $p_i^j, x_{i,j} \in \mathcal{B}_k$ whenever $j \leq k$.  In
  particular, when $n=k$, we get that $\mathcal{B}_n = \mathcal{B}$ as
  required.

  \medskip\noindent(Induction) Let $k\geq 1$, and suppose the
  statement is known for smaller values of $k$.  First, we check that
  the $x_{i,k} \in \mathcal{B}_k$, for all $i<k$.  To this end, recall
  that $P_{\pi(i),i}=1$ for all $i$ and $P_{\pi(i),j}=0$ for all
  $i<j$. Thus, if $i<k$, then the $(\pi(i),k)$-entry of $P^{xz}_\beta
  \cdot (I+X\Sigma)$ is
  \begin{align*}
    (P^{xz}_\beta \cdot (I+X\Sigma))_{(\pi(i),k)}&=\sum_{\ell=1}^{k-1}(-1)^{\mu_{k}}P_{\pi(i),\ell}x_{\ell,k}+P_{\pi(i),k} \\
    &=\sum_{\ell=1}^{i-1}(-1)^{\mu_{k}}P_{\pi(i),\ell}x_{\ell,k}+x_{i,k}.
  \end{align*}
  The inductive hypothesis gives that the $P_{\pi(i),\ell}$ are in
  $\mathcal{B}_k$; arguing inductively in $i$, we conclude that
  $x_{i,k} \in \mathcal{B}_k$ for all $i<k$.
	
  For generators of the form $p_i^k$, there exists $\mu\in\Z$ such
  that $P_{i,k}=(-1)^\mu p^{k}_i$. We see that
  \begin{align*}
    (P^{xz}_\beta\cdot (I+X\Sigma))_{(i,k)}&= \sum_{\ell=1}^{k-1}(-1)^{\mu_{k}}P_{i,\ell}x_{\ell,k}+P_{i,k}\\
    &=\sum_{\ell=1}^{k-1}(-1)^{\mu_{k}}P_{i,\ell}x_{\ell,k}+(-1)^\mu
    p^{k}_i.
  \end{align*}
  Since $P_{i,\ell}, x_{\ell,k} \in \mathcal{B}_k$ for all $\ell < k$,
  we get that $p^k_i \in \mathcal{B}_k$.

  \medskip\noindent{\bf ($R$-path injective)} Let
  $\alpha,\gamma\in\mathit{Ring}(\mathcal{B},R)$ such that
  $\alpha(P^{xy}_\beta)=\gamma(P^{xy}_\beta)$. To show that
  $\alpha(p_i)=\gamma(p_i)$, $\alpha(x_{i,j})=\gamma(x_{i,j})$, and
  $\alpha(t_i)=\gamma(t_i)$ for all $i$ and $j$, we will prove that
  $\alpha(x_{i,j})=\gamma(x_{i,j})$ for all $i<j$,
  $\alpha(P_{i,j})=\alpha\left(t_i^{(-1)^{\mu_i+1}}\right)\gamma\left(t_i^{(-1)^{\mu_i}}P_{i,j}\right)$
  for all $i$ and $j$, and $\alpha(t_{\pi(i)})=\gamma(t_{\pi(i)})$ for
  all $i$.

  Looking at the first column of $P^{xy}_\beta$ tells us
  \[\alpha\left(t_i^{(-1)^{\mu_i}}P_{i,1}\right)=\gamma\left(t_i^{(-1)^{\mu_i}}P_{i,1}\right)\]
  and so
  \[\alpha(P_{i,1})=\alpha\left(t_i^{(-1)^{\mu_i+1}}\right)\gamma\left(t_i^{(-1)^{\mu_i}}P_{i,1}\right)\]
  as desired. In particular, if $i=\pi(1)$, then we have
  \[1=\alpha(P_{\pi(1),1})=\alpha\left(t_{\pi(1)}^{(-1)^{\mu_{\pi(1)}+1}}\right)\gamma\left(t_{\pi(1)}^{(-1)^{\mu_{\pi(1)}}}P_{\pi(1),1}\right)=\alpha\left(t_{\pi(1)}^{(-1)^{\mu_{\pi(1)}+1}}\right)\gamma\left(t_{\pi(1)}^{(-1)^{\mu_{\pi(1)}}}\right)
  \]
  since $P_{\pi(1),1}=1$. Thus
  $\alpha(t_{\pi(1)})=\gamma(t_{\pi(1)})$.

  Now induct on $k$, the column of $P^{xy}_\beta$. Assume that
  \begin{itemize}
  \item $\alpha(x_{i,j})=\gamma(x_{i,j})$ for $i<j<k$,
  \item
    $\alpha(P_{i,j})=\alpha\left(t_i^{(-1)^{\mu_i+1}}\right)\gamma\left(t_i^{(-1)^{\mu_i}}P_{i,j}\right)$
    for $j<k$, and
  \item $\alpha(t_{\pi(j)})=\gamma(t_{\pi(j)})$ for $j<k$.
  \end{itemize}
  We now establish these three identities in the case that $j=k$.
  
  To obtain the first identity, we will prove by induction on $j$ that
  $\alpha(x_{j,k})=\gamma(x_{j,k})$ for all $j<k$. Looking at the
  $k$-th column of $P^{xy}_\beta$, we see that
  \[\alpha\left(t_i^{(-1)^{\mu_i}}\left[\sum_{\ell=1}^{k-1}(-1)^{\mu_k}P_{i,\ell}x_{\ell,k}+P_{i,k}\right]\right)=\gamma\left(t_i^{(-1)^{\mu_i}}\left[\sum_{\ell=1}^{k-1}(-1)^{\mu_k}P_{i,\ell}x_{\ell,k}+P_{i,k}\right]\right)\]
  for all $i$.

  \medskip\noindent ($j=1$) By setting $i=\pi(j)=\pi(1)$, we see that
  \begin{align*}
    \alpha\left(t_{\pi(1)}^{(-1)^{\mu_{\pi(1)}}}\left[\sum_{\ell=1}^{k-1}(-1)^{\mu_k}P_{i,\ell}x_{\ell,k}+P_{i,k}\right]\right)&=\gamma\left(t_{\pi(1)}^{(-1)^{\mu_{\pi(1)}}}\right)\left[(-1)^{\mu_k}\alpha(P_{\pi(1),1}x_{1,k})\right]\\
    &=\gamma\left(t_{\pi(1)}^{(-1)^{\mu_{\pi(1)}}}\right)(-1)^{\mu_k}\alpha(x_{1,k}),
  \end{align*}
  since $P_{\pi(1),\ell}=0$ if $\ell>1$. Similarly we see that
  \[\gamma\left(t_{\pi(1)}^{(-1)^{\mu_{\pi(1)}}}\left[\sum_{\ell=1}^{k-1}(-1)^{\mu_k}P_{i,\ell}x_{\ell,k}+P_{i,k}\right]\right)=\gamma\left(t_{\pi(1)}^{(-1)^{\mu_{\pi(1)}}}(-1)^{\mu_k}x_{1,k}\right).\]
  Thus $\alpha(x_{1,k})=\gamma(x_{1,k})$.

  \medskip\noindent (Induction) Now suppose for $\ell<j$,
  $\alpha(x_{\ell,k})=\gamma(x_{\ell,k})$. Setting $i=\pi(j)$, we have
  \begin{align*}
    \alpha&\left(t_{\pi(j)}^{(-1)^{\mu_{\pi(j)}}}\left[\sum_{\ell=1}^{k-1}(-1)^{\mu_k}P_{\pi(j),\ell}x_{\ell,k}+P_{\pi(j),k}\right]\right)\\
    &=\gamma\left(t_{\pi(j)}^{(-1)^{\mu_{\pi(j)}}}\right)\left[\sum_{\ell=1}^{j-1}(-1)^{\mu_k}\alpha(P_{\pi(j),\ell}x_{\ell,k})+(-1)^{\mu_k}\alpha(P_{\pi(j),j}x_{j,k})\right]\text{ since }P_{\pi(j),\ell}=0\text{ if }\ell>j\\
    &=\gamma\left(t_{\pi(j)}^{(-1)^{\mu_{\pi(j)}}}\right)\left[\sum_{\ell=1}^{j-1}(-1)^{\mu_k}\alpha\left(t_{\pi(j)}^{(-1)^{\mu_{\pi(j)}+1}}\right)\gamma\left(t_{\pi(j)}^{(-1)^{\mu_{\pi(j)}}}P_{\pi(j),\ell}\right)\gamma(x_{\ell,k})+(-1)^{\mu_k}\alpha(x_{j,k})\right]\\
    &=\gamma\left(t_{\pi(j)}^{(-1)^{\mu_{\pi(j)}}}\right)\left[\sum_{\ell=1}^{j-1}\gamma\left(P_{\pi(j),\ell}x_{\ell,k}\right)+(-1)^{\mu_k}\alpha(x_{j,k})\right],
  \end{align*}
  since $\alpha(t_{\pi(j)})=\gamma(t_{\pi(j)})$ and
  $\alpha(x_{\ell,k})=\gamma(x_{\ell,k})$. Similarly,
  \[\gamma\left(t_{\pi(j)}^{(-1)^{\mu_{\pi(j)}}}\left[\sum_{\ell=1}^{k-1}(-1)^{\mu_k}P_{i,\ell}x_{\ell,k}+P_{i,k}\right]\right)=\gamma\left(t_{\pi(j)}^{(-1)^{\mu_{\pi(j)}}}\left[\sum_{\ell=1}^{j-1}(-1)^{\mu_k}P_{\pi(j),\ell}x_{\ell,k}+(-1)^{\mu_k}x_{j,k}\right]\right).\]
  Therefore $\alpha(x_{j,k})=\gamma(x_{j,k})$ as desired.

  \medskip

  We turn now towards establishing the 2nd and 3rd bullet points of
  the main induction.  Note that the 2nd equation involving $P_{i,k}$
  holds when $i = \pi(j)$ for some $j <k$ since in this case
  $P_{i,k}=0$.  When $i = \pi(j)$ for some $k \leq j$, we have
  \begin{align*}\gamma\left(t_i^{(-1)^{\mu_i}}\left[\sum_{\ell=1}^{k-1}(-1)^{\mu_k}P_{i,\ell}x_{\ell,k}+P_{i,k}\right]\right)&=\alpha\left(t_i^{(-1)^{\mu_i}}\left[\sum_{\ell=1}^{k-1}(-1)^{\mu_k}P_{i,\ell}x_{\ell,k}+P_{i,k}\right]\right)\\
    &=\alpha\left(t_i^{(-1)^{\mu_i}}\right)\left[\sum_{\ell=1}^{k-1}(-1)^{\mu_k}\alpha\left(t_i^{(-1)^{\mu_i+1}}\right)\gamma\left(t_i^{(-1)^{\mu_i}}P_{i,\ell}x_{\ell,k}\right)+\alpha(P_{i,k})\right]\\
    &=\sum_{\ell=1}^{k-1}(-1)^{\mu_k}\gamma\left(t_i^{(-1)^{\mu_i}}P_{i,\ell}x_{\ell,k}\right)+\alpha\left(t_i^{(-1)^{\mu_i}}P_{i,k}\right).
  \end{align*}
  So
  \begin{align*}
    \alpha(P_{i,k})&=\alpha\left(t_i^{(-1)^{\mu_i+1}}\right)\left[\gamma\left(t_i^{(-1)^{\mu_i}}\left[\sum_{\ell=1}^{k-1}(-1)^{\mu_k}P_{i,\ell}x_{\ell,k}+P_{i,k}\right]\right)-\sum_{\ell=1}^{k-1}(-1)^{\mu_k}\gamma\left(t_i^{(-1)^{\mu_i}}P_{i,\ell}x_{\ell,k}\right)\right]\\
    &=\alpha\left(t_i^{(-1)^{\mu_i+1}}\right)\gamma\left(t_i^{(-1)^{\mu_i}}P_{i,k}\right).
  \end{align*}
  If $i=\pi(k)$, then
  \[1=\alpha\left(t_{\pi(k)}^{(-1)^{\mu_{\pi(k)+1}}}\right)\gamma\left(t_{\pi(k)}^{(-1)^{\mu_{\pi(k)}}}\right)\]
  so $\alpha(t_{\pi(k)})=\gamma(t_{\pi(k)})$.

  So we have shown that
  \begin{itemize}
  \item $\alpha(x_{i,j})=\gamma(x_{i,j})$ for $i<j$,
  \item
    $\alpha(P_{i,j})=\alpha(t_i^{(-1)^{\mu_i+1}})\gamma(t_i^{(-1)^{\mu_i}}P_{i,j})$
    for all $i,j$, and
  \item $\alpha(t_{\pi(i)})=\gamma(t_{\pi(i)})$ for all $i$.
  \end{itemize}
  But $\pi\in S_n$, so $\alpha(t_{\pi(i)})=\gamma(t_{\pi(i)})$ for all
  $i$ if and only if $\alpha(t_i)=\gamma(t_i)$ for all $i$. Thus
  $\alpha(P_{i,j})=\gamma(P_{i,j})$ for all $i,j$. So
  \begin{itemize}
  \item $\alpha(x_{i,j})=\gamma(x_{i,j})$ for all $i<j$,
  \item $\alpha(P_{i,j})=\gamma(P_{i,j})$ for all $i,j$, and
  \item $\alpha(t_i)=\gamma(t_i)$ for all $i$.
  \end{itemize}
  By Proposition~\ref{prop:kalman}, we know for each $\ell$ there
  exist $i,j$ and $\mu\in\Z$ such that $P_{i,j}=(-1)^\mu p_\ell$, so
  $\alpha(p_\ell)=\gamma(p_\ell)$ for all $\ell$ and we are done.
\end{proof}

\subsection{A decomposition of $\mathit{GL}(n,\mathbb{F})$ from path
  matrices}

We conclude this section by establishing the following decomposition
of $\mathit{GL}(n,\mathbb{F})$ into path subsets (as in Definition
\ref{def:PathSub}) of positive permutation braids.

\begin{proposition} \label{prop:bruhat} For each permutation in $S_n$,
  fix a positive permutation braid, $\beta$, with reduced braid
  word. For any field, $\mathbb{F}$, we have
  \[
  \mathit{GL}(n,\mathbb{F}) = \bigsqcup_{\beta \in S_n} B_\beta.
  \]
\end{proposition}
\begin{proof}
  \noindent {\bf Step 1.}  $\mathit{GL}(n,\mathbb{F}) = \bigcup_{\beta
    \in S_n} B_\beta.$

  \medskip

  The path subset $B_\beta$ consists of all matrices of the form $D
  S_\pi U$ where $D$ is invertible diagonal; $U$ is upper triangular
  with $1$'s on the main diagonal; and (in view of
  Proposition~\ref{prop:kalman}) $S_\pi$ is a matrix obtained from the
  permutation matrix for $\pi$, $P_\pi$, (where $\pi$ is the
  permutation associated to $\beta$) by allowing arbitrary
  $(i,j)$-entries $s_{i,j} \in \mathbb{F}$ in positions that are above
  the $1$ of $P_\pi$ in column $j$ and to the left of the $1$ of
  $P_\pi$ in row $i$.  Therefore, it suffices to show that for any $A
  \in GL(n, \mathbb{F})$ there exists matrices $D$, $U$, and $S_\pi$
  as above such that
  \begin{equation} \label{eq:factorization} D A U = S_\pi
  \end{equation}
  (since then $A = D^{-1} S_\pi U^{-1} \in B_\beta$).

  Given $A$, we arrive at (\ref{eq:factorization}) via the following
  procedure:
  \begin{enumerate}
  \item In the $1$-st column of $A$, locate the lowest nonzero entry,
    say it is in row $i_1$.  Choose the $(i_1,i_1)$-entry of $D$ to
    scale this entry to $1$.

  \item Right multiply $A$ by an upper-triangular matrix of the form
    $U_1 = I+u_{1,j}E_{1,j}$ where the $u_{1,j}$ are chosen to zero
    out all entries to the right of the $1$ in position $(i_1,1)$.

  \item Move on to the next column and repeat this procedure.
  \end{enumerate}

  Note that the sequence of rows $i_1, \ldots, i_n$ where the $1$'s in
  columns $1, \ldots, n$ are arranged at Step (1) of the procedure
  must all be distinct.  [When $i_{m}$ is determined all the entries
  to the right of the $1$'s in positions $(i_1,1)$, through
  $(i_{m-1},m-1)$ have already become $0$.  However, there is still at
  least one non-zero entry in column $m$ since $A \in GL(n,
  \mathbb{F})$.]  Therefore, the resulting matrix, $D A (U_1U_2 \cdots
  U_n)$, has the required form $S_\pi$.

  \medskip

  \noindent {\bf Step 2.}  The $B_\beta$ are disjoint.

  \medskip

  If this were not the case, then for distinct permutations $\pi_1
  \neq \pi_2$, we can arrange
  \[
  S_{\pi_1} = D S_{\pi_2} U
  \]
  for some $D$, $U$, $S_{\pi_m}$ as above.  Considering the left most
  column where the position of the lowest non-zero entry (necessarily
  a $1$) of $S_{\pi_1}$ and $S_{\pi_2}$ differs leads to a
  contradiction.  [WLOG we can assume the position of this $1$ in
  $S_{\pi_1}$ is lower than in $S_{\pi_2}$.  Then, there is no way for
  the corresponding entry of $DS_{\pi_2} U$ to become non-zero.]

\end{proof}

\begin{remark}
  In fact, the decomposition $GL(n, \mathbb{F}) = \bigsqcup_{\beta \in
    S_n} B_\beta$ is precisely the Bruhat Decomposition of $GL(n,
  \mathbb{F})$ into double cosets
  \[
  GL(n, \mathbb{F}) = \bigsqcup_{\pi \in S_n} B P_\pi B
  \]
  where $B$ is the Borel subgroup of upper triangular invertible
  matrices.  To see this, note that the matrices $S_\pi$ have the form
  $V P_\pi$ with $V$ upper triangular, so that $B_\beta \subset B
  P_\pi B$ follows.  The reverse inclusion then also holds since both
  collections of subsets partition $GL(n, \mathbb{F})$.  \end{remark}

\section{The DGA of a satellite} \label{sec:DGA}

In this section, we consider satellites $S(K,\beta)$ with pattern an
$n$-stranded positive braid $\beta \subset J^1S^1$, and give a
computation of the (fully non-commutative) DGA
$(\alg(S(K,\beta)), \partial)$.

\subsection{Satellites via the $xy$-projection} \label{sec:alg}

Given a (connected) Legendrian knot $K \subset J^1\R$ and a Legendrian
link $L \subset J^1S^1$, recall that the Legendrian satellite of $K$
with pattern $L$ is formed as follows: Using the Weinstein Tubular
Neighborhood Theorem, we find a contactomorphism $\psi: N' \rightarrow
N(K)$ from a neighborhood of the $0$-section  $N' \subset J^1S^1$ to
a neighborhood of $K$, $N(K) \subset J^1\R$.  The satellite, $S(K,L)$,
is then obtained by scaling the $y$- and $z$-coordinates of $L$ to
shrink it into $N'$, and then applying $\psi$.

When $K \subset J^1\R$ is equipped with a basepoint $*$ and $\beta
\subset J^1S^1$ is a positive braid with $n$-strands, we form a
basepointed $xy$-diagram for the satellite $S(K,\beta)$ using the
following procedure.  (See Figure \ref{fig:Satellite} for an
illustration.)
\begin{enumerate}
\item Begin with an $xy$-diagram for $\beta$ as in Section
  \ref{sec:PathMatrix}. Basepoints, $*_1, \ldots, *_n$, appear just to
  the left of the crossings of $\beta$, so that $*_i$ belongs to the
  $i$-th strand (as numbered with decreasing $y$-coordinate).

\item Next, this $xy$-diagram for $\beta$ is placed in an annular
  thickening of the $xy$-projection of $K$, so that all crossings of
  $\pi_{xy}(\beta)$ occur in a small neighborhood of $*$.  Here, we
  use the orientation of $K$ and the blackboard framing to identify an
  immersed neighborhood of the $xy$-diagram of $K$ with $S^1 \times
  (-\epsilon, \epsilon)$.
\end{enumerate}

Note that as required in Section \ref{sec:RepNumber}, every component
of $S(K,\beta)$ receives at least one basepoint, although there may
be components with multiple basepoints.  For instance, if $\beta =
\sigma_1 \sigma_2 \sigma_1 \sigma_3 \in B_4$, then $S(K, \beta)$ has
$2$ components with basepoints $*_1, *_3, *_4$ on one component and
$*_2$ on the other, as seen in Figure~\ref{fig:Satellite}.

\begin{remark}
  To see that the above construction produces an $xy$-diagram for
  $S(K,\beta)$, note that the contactomorphism $\psi$ may be chosen to
  preserve the Reeb vector field, 
  and hence induces a well defined map of Lagrangian projections
  $\pi_{xy}(N') \rightarrow \pi_{xy}(N(K))$.  Alternatively, an
  $xy$-diagram of the form described above can be obtained via the
  combinatorial description (via $xz$-projections) of $S(K,\beta)$
  from \cite{NgTraynor}.
\end{remark}

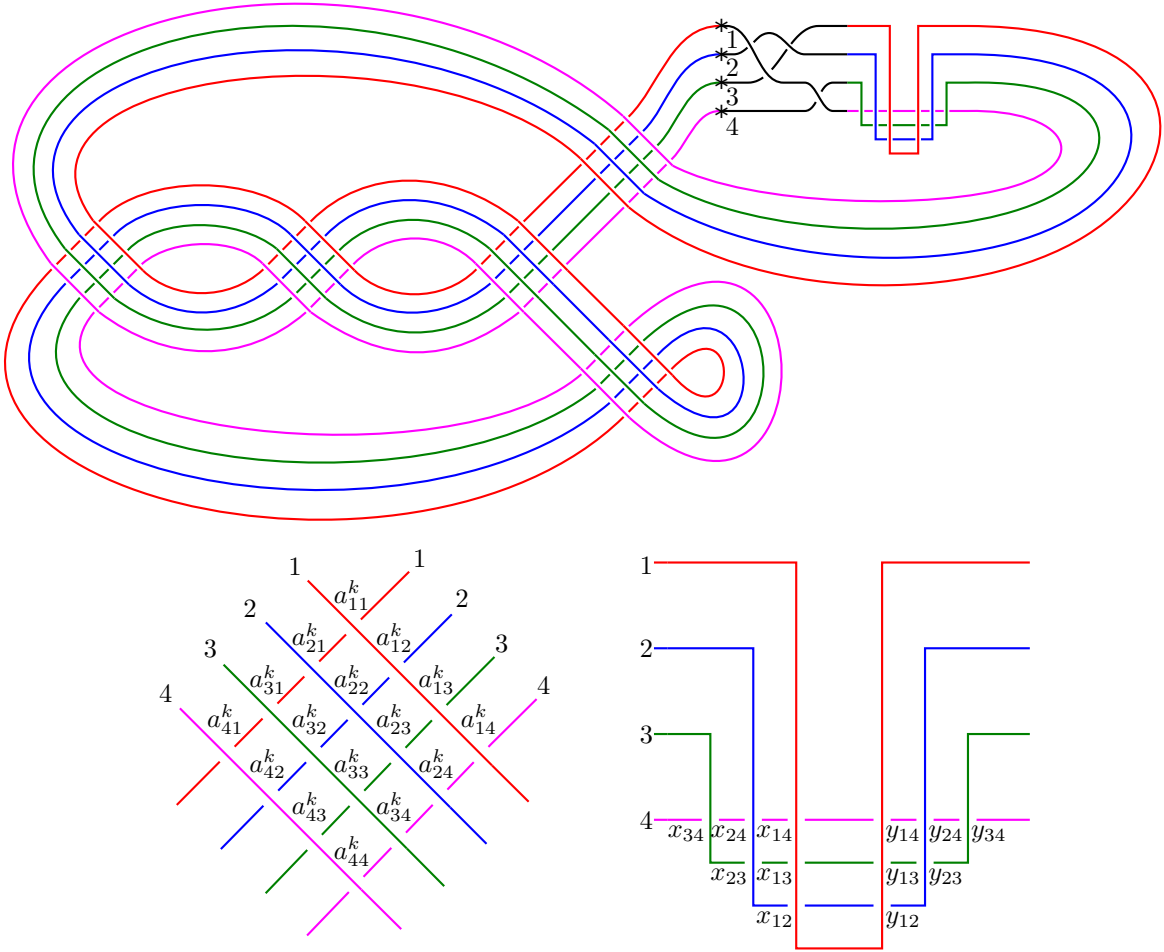
\begin{figure}[t]
  \begin{centering}
    \begin{subfigure}{\textwidth}
      \labellist
      \small
      % Satellite t's
      \pinlabel $1$ [tl] at 808 561
      \pinlabel $2$ [tl] at 808 529 
      \pinlabel $3$ [tl] at 808 497
      \pinlabel $4$ [tl] at 808 463
      \endlabellist
      \includegraphics[width=.9\textwidth]{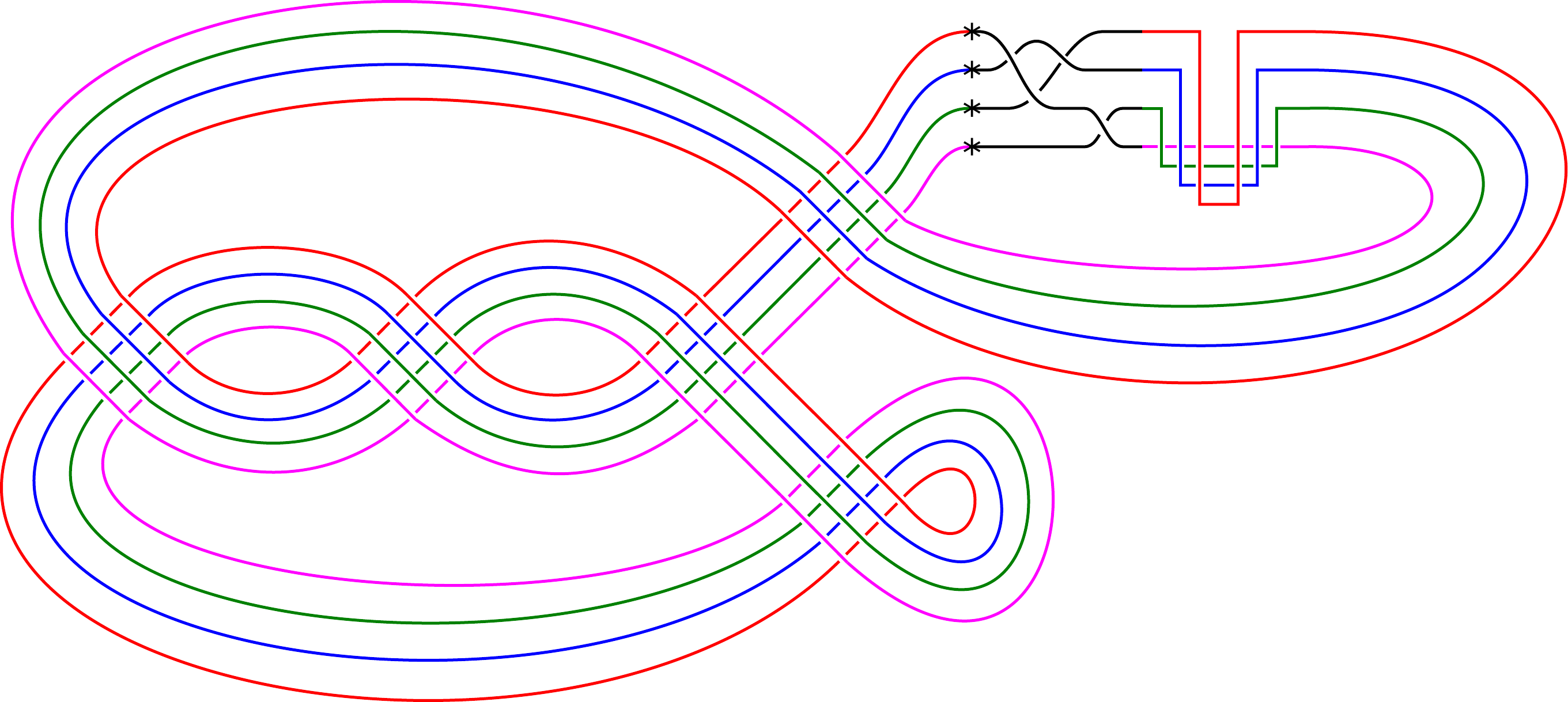}
    \end{subfigure}
  \end{centering}\\
  \begin{subfigure}{.4\textwidth}
    \labellist
    \small
    % a_k lattice
    \pinlabel $1$ [br] at 68 152
    \pinlabel $2$ [br] at 51 136
    \pinlabel $3$ [br] at 36 121
    \pinlabel $4$ [br] at 19 104
    \pinlabel $1$ [bl] at 106 155
    \pinlabel $2$ [bl] at 122 140
    \pinlabel $3$ [bl] at 137 123
    \pinlabel $4$ [bl] at 153 107
    \pinlabel $a^k_{11}$ [b] at 85 138
    \pinlabel $a^k_{12}$ [b] at 101 122
    \pinlabel $a^k_{13}$ [b] at 117 106
    \pinlabel $a^k_{14}$ [b] at 133 91
    \pinlabel $a^k_{21}$ [b] at 69 122
    \pinlabel $a^k_{22}$ [b] at 85 106
    \pinlabel $a^k_{23}$ [b] at 101 91
    \pinlabel $a^k_{24}$ [b] at 117 74
    \pinlabel $a^k_{31}$ [b] at 53 106
    \pinlabel $a^k_{32}$ [b] at 69 91
    \pinlabel $a^k_{33}$ [b] at 85 74
    \pinlabel $a^k_{34}$ [b] at 101 58
    \pinlabel $a^k_{41}$ [b] at 37 91
    \pinlabel $a^k_{42}$ [b] at 53 74
    \pinlabel $a^k_{43}$ [b] at 69 58
    \pinlabel $a^k_{44}$ [b] at 85 41
    \endlabellist
    \includegraphics[width=2.4in]{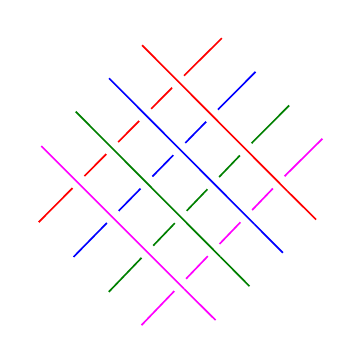}
  \end{subfigure}
  \begin{subfigure}{.4\textwidth}
    \labellist
    \small
    % Dip
    \pinlabel $1$ at -2 155
    \pinlabel $2$ at -2 124
    \pinlabel $3$ at -2 92
    \pinlabel $4$ at -2 60
    \pinlabel $x_{12}$ [tr] at 55 28
    \pinlabel $x_{13}$ [tr] at 55 44
    \pinlabel $x_{14}$ [tr] at 55 60
    \pinlabel $x_{23}$ [tr] at 38 44
    \pinlabel $x_{24}$ [tr] at 38 60
    \pinlabel $x_{34}$ [tr] at 22 60
    \pinlabel $y_{12}$ [tl] at 85 28
    \pinlabel $y_{13}$ [tl] at 85 44
    \pinlabel $y_{14}$ [tl] at 85 60
    \pinlabel $y_{23}$ [tl] at 101 44
    \pinlabel $y_{24}$ [tl] at 101 60
    \pinlabel $y_{34}$ [tl] at 117 60
    \endlabellist
    \includegraphics[width=2in]{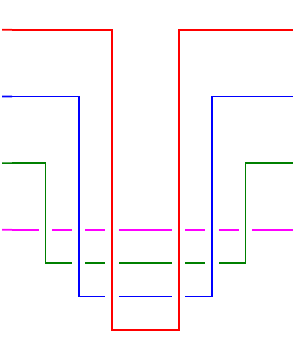}
  \end{subfigure}
 
  \caption{The top figure is an $xy$-diagram of $S(K,\beta)$ for $K$ a
    right-handed trefoil and $\beta$ the braid from
    Figure~\ref{fig:BraidRes}. In the top figure, $t_i$ corresponds to
    the basepoint $*$ with label $i$. The bottom left figure gives
    the labels for the crossings in the lattice of crossings of $S(K,
    \beta)$ arising from the crossing $a_k$ of $K$ if the strands have
    the given labels (as in the three left-most lattices of the top
    figure. The bottom right figure gives the labels of the crossings
    in the $x,y$-lattice (the dip).}
  \label{fig:Satellite}
\end{figure}

\subsection{Generators and grading}
Outside of a neighborhood of $*$ where the crossings from $\beta$ are
located, the $xy$-diagram of $S(K,\beta)$ consists of $n$ parallel
copies of $K$.  We number these copies of $K$ from $1$ to $n$ so that
the labeling increases as the coordinate transverse to $K$ decreases;
the $i$-th copy corresponds to the part of the $xy$-projection of
$\beta$ that stretches around the $S^1$ factor from the right side of
the dip to the basepoint $*_i$.

Choose a labeling of the Reeb chords of $K$ as $a_1, \ldots, a_r$.
Along with the invertible generator $t^{\pm1}$ corresponding to the
basepoint $*$, the $a_k$ generate $\alg(K)$.

The generators of $\mathcal{A}(S(K,\beta))$ are as follows:

\begin{enumerate}
\item For each $1 \leq k \leq r$, there is a collection of $n^2$
  crossings of $S(K,L)$ arranged in a lattice at the location of
  $a_k$.  We label them as
  \[
  a^k_{i,j}, \quad 1 \leq i,j \leq n
  \]
  so that the overstrand (resp. understrand) of $a^k_{i,j}$ belongs to
  the $i$-th (resp. $j$-th) copy of $K$.

\item The crossings of the $xy$-diagram of $\beta$ also appear as
  crossings of $S(K,L)$.  We label them as
  \[
  p_1, \ldots, p_s; \quad x_{i,j}, y_{i,j}, \mbox{ for $1 \leq i < j
    \leq n$}
  \]
  as in Section \ref{sec:PathMatrix}.

\item Label the invertible generators corresponding to the basepoints
  $*_i$ as
  \[
  t_i, \quad 1 \leq i \leq n.
  \]
\end{enumerate}

We choose the grading
\[
|t_i|= -2 r(K)
\]
for all invertible generators, $t_i$.  Then, a choice of $\Z$-valued
(with the value decreasing by $-2r(K)$ at $*$, as in Section
\ref{sec:Zgrading}) Maslov potential, $\mu_K$, on $K$ and $\Z$-valued
Maslov potential $\mu_\beta$ on $\beta$ (continuous at base points) specifies a $\Z$-valued (with
the value decreasing by $-2r(K)$ at all $*_1, \ldots, *_n$) Maslov
potential on $S(K,\beta)$ via $\mu = \mu_K + \mu_\beta$.  Letting
$(\mu_1, \ldots, \mu_n)$ be the values of $\mu_\beta$ at the base
points $*_1, \ldots, *_n$, the grading of generators of $S(K,\beta)$
satisfies
\[
|a^k_{i,j}| = |a_k| + \mu_i - \mu_j; \quad \quad |x_{i,j}| =
\mu_i-\mu_j; \quad \quad |y_{i,j}| = \mu_i-\mu_j -1; \quad \quad |t_i|
= -2 r(K)
\]
and the grading of $p_i$ in $\mathcal{A}(S(K,\beta))$ agrees with its
grading in $\mathcal{A}(\beta)$.

Notice that this choice of basepoints and Maslov potential for
$S(K,\beta)$ is consistent with (A1)-(A3) of Section
\ref{sec:defnumbers}, so that $\mathcal{A}(S(K,\beta))$ is suitable
for computing $m$-graded augmentation numbers of $S(K,\beta)$ for any
$m \big\vert 2r(K)$.  (Note that the rotation number of a component
$C\subset S(K,\beta)$ corresponding to a component of $\beta$ with
winding number $n'$ around $S^1$ is $r(C) = n'r(K)$, so any divisor of
$2r(K)$ also divides $2r(S(K,\beta))$.  When computing augmentation
numbers of $S(K,\beta)$ with $m$ even, we will assume $r(K)=0$, so
that $r(S(K,\beta))=0$ also holds as required.)

\subsection{The differential} \label{sec:Matrix} Collect generators of
$\mathcal{A}(S(K,\beta))$ into $n\times n$ matrices as follows.  For
$1 \leq k \leq r$, let
\[
A_k = (a^k_{i,j}), \quad 1\leq k \leq r.
\]
As in Section \ref{sec:Path}, define \emph{strictly upper triangular}
matrices $X$ and $Y$ to have $(i,j)$-entries $x_{i,j}$ and $y_{i,j}$
when $1 \leq i<j \leq n$ and $0$ when $i \geq j$, and (invertible)
diagonal matrices $\Sigma$ and $\Delta$ as in (\ref{eq:SigDelta}).

Define an algebra homomorphism
\begin{equation} \label{eq:DefPhi} \Phi: \mathcal{A}(K) \rightarrow
  \mathit{Mat}(n,\mathcal{A}(S(K,\beta)), \quad a_k \mapsto A_k\Sigma,
  \quad t \mapsto P^{xy}_\beta.
\end{equation}
(Note that $P^{xy}_\beta$ is invertible; see
Proposition~\ref{prop:PathProp}.)  As in Convention \ref{conv:1}, the
differential on $\mathcal{A}(S(K,\beta))$ applied entry-by-entry
produces $\partial: \mathit{Mat}(n,\mathcal{A}(S(K,\beta)) \rightarrow
\mathit{Mat}(n,\mathcal{A}(S(K,\beta))$.

\begin{proposition} \label{prop:DiffProof} The differential in
  $\mathcal{A}(S(K,\beta))$ agrees with the differential for $\beta
  \subset J^1S^1$ on the sub-algebra $\mathcal{A}(\beta) \subset
  \mathcal{A}(S(K,\beta))$ and satisfies
  \begin{align}
    &\Sigma \cdot \partial( A_k\Sigma)=\Phi\circ\partial(a_k) - (Y\Sigma)\cdot \Phi(a_k)+(-1)^{\lvert a_k\rvert}\Phi(a_k) \cdot (Y\Sigma),\\
    &\Sigma\cdot\partial(Y\Sigma)=-(Y\Sigma)^2.
  \end{align}
\end{proposition}

In particular, the formula for $\partial P^{xy}_\beta$ established in
Proposition~\ref{prop:Pxy1} also holds in $\mathcal{A}(S(K,\beta))$.

\begin{proof}
  Following \cite{NgR}, we will divide the disks contributing terms to
  the differential for $\alg(S(K,\beta))$ into two disjoint sets: {\bf
    thin disks} and {\bf thick disks}. Thin disks are those completely
  contained in the neighborhood $N(K)$ of the $xy$-diagram of $K$
  where the satellite construction is completed and all other disks
  are thick.  
	Note that the disks used in computing the differential in
  $\mathcal{A}(\beta)$ correspond to thin disks in $S(K,\beta)$. After
  some consideration, we see that the only disks contributing to
  $\partial p_k$, $\partial x_{i,j}$, and $\partial y_{i,j}$ are
  exactly these thin disks.

  It only remains to compute the differentials of $a^k_{i,j}$ and
  $y_{i,j}$ for all $i,j,k$.  Since $\partial\Sigma=0$, the signed
  Leibniz rule tells us $\partial(A_k\Sigma)=(\partial
  A_k)\cdot\Sigma$ and $\partial(Y\Sigma)=(\partial Y)\cdot\Sigma$, so
  it suffices to show the following:
  \begin{align*}
    &\partial(A_k)=\Sigma\Phi(\partial(a_k))\Sigma -\Sigma Y\Sigma A_k+(-1)^{\lvert a_k\rvert}\Sigma A_k\Sigma Y,\\
    &\partial(Y)=-(\Sigma Y)^2.
  \end{align*}

  \noindent{\bf Differential of $A_k$:}

  \noindent(Thick disks) Each thick disk, $\widetilde{u}$, of $S(K,\beta)$  corresponds to a disk, $u$, of $K$ as follows.  Observe that (since neighborhoods of positive or negative punctures only cover a single quadrant at crossings) each time the boundary of a thick disk for $S(K,\beta)$ passes through the lattice of crossings $(a^k_{i,j})$ corresponding to a crossing $a_k$ of $K$, it can have at most one puncture.  Then, $\widetilde{u}$ will have a negative  (resp. positive) puncture at some $a^k_{i,j}$ if and only if $u$ has a negative (resp. positive) puncture at $a_k$. In addition, when the boundary $\widetilde{u}$ enters into the region near the base point of $K$ where the crossings from the $xy$-diagram of $\beta$ appear it must follow a path as in the definition of the left-to-right (resp. right-to-left) path matrix $P^{xy}_\beta$ (resp. $Q^{xy}_\beta$), assuming the  orientation of the boundary of the thick disk agrees (resp. disagrees) with the orientation of $K$ when it enters the region; see Definition
    \ref{def:pathmatrix}. The corresponding disk $u$ simply passes by the base point.
	
  In particular, for all $1\leq k\leq m$ and all
  $1\leq i,j\leq n$, the thick disks which contribute to
  $\partial(a^k_{i,j})$ 
  are in many-to-one correspondence with disks
  which contribute to $\partial a_k$. The negative corners of a disk
  for $\partial(a^k_{i,j})$ correspond to the negative corners of a
  disk for $\partial(a_k)$, unless the boundary of the disk passes
  through the basepoint. We will discuss this case later. We will
  drop the $k$'s from $a_k$ and $a^k_{i,j}$ when it remains clear
  which generator we are discussing. Consider the disk contributing to
  $\partial a$ with a positive corner at $a$ and negative corners at
  $b_1,\ldots,b_\ell$. This disk corresponds to the term
  $\iota'\iota_0b_1\iota_{1}\cdots b_\ell\iota_{\ell}$ in $\partial
  a$. (Recall sign conventions from Section~\ref{sec:background}.) The
  disks contributing to $\partial(a_{i,j})$ corresponding to this disk
  have a positive (Reeb sign) corner at $a_{i,j}$ and negative corners
  at $b^r_{i_{r-1},i_r}$ for $1\leq r\leq \ell$, where $i_0=i$,
  $i_\ell=j$, and $1\leq i_r\leq n$ for all $0\leq r\leq \ell$. Recall
  that the orientation of strand $r$ agrees with the orientation of
  $K$ if and only if $\mu_r\equiv0\mod 2$. Moreover, the orientation
  sign for the corner at $a$ (resp. $b_r$) will agree with the
  orientation sign for the corner at $a_{i,j}$ (resp.
  $b^r_{i_{r-1},i_r}$) if and only if the orientation of strand $j$
  (resp. $i_r$) agrees with the orientation of $K$. Therefore such a
  disk contributes the term
  \begin{align*}
    &(-1)^{\mu_0}\iota'(-1)^{\mu_{i_\ell}}\iota_0b^1_{i_0,i_1}(-1)^{\mu_{i_1}}\iota_{1}\cdots b^\ell_{i_{\ell-1},i_\ell}(-1)^{\mu_{i_\ell}}\iota_{\ell}\\
    &\quad=\iota'\iota_0\iota_{1}\cdots\iota_{\ell}\left[(-1)^{\mu_{i_0}}\right]\cdot
    \left[b^1_{i_0,i_1}(-1)^{\mu_{i_1}}\right]\cdots
    \left[b^\ell_{i_{\ell-1},i_\ell}(-1)^{\mu_{i_\ell}}\right]\cdot\left[(-1)^{\mu_{i_\ell}}\right].
  \end{align*}
  Summing over all $1 \leq i_1,\ldots, i_{\ell-1} \leq n$ gives the
  $(i,j)$-entry in the matrix
  \[
  \iota'\iota_0\iota_{1}\cdots\iota_{\ell} \Sigma \cdot
  (B_1\Sigma)\cdots (B_\ell \Sigma)\cdot \Sigma =
  \Sigma\cdot\Phi(\iota'\iota_0\iota_{1}\cdots\iota_{\ell}b_1\cdots
  b_\ell)\cdot\Sigma.
  \]
	
  Now consider the disks for $\partial a_k$ where the boundary of the
  disk passes through the basepoint.  Such a disk, $u$, contributes a
  term of the form $\iota(u) \cdot c_1 c_2 \cdots c_\ell$ where some
  $c_i$ arise from negative corners at Reeb chords and some $c_i=
  t^{\pm1}$ arise when the boundary of $u$ passes the basepoint of
  $K$.  (The $t_{i}^{\pm1}$ terms do not contribute to the sign
  $\iota(u) \in \{\pm1\}$.)  The corresponding thick disks for
  $\partial(a_{i,j})$ again arise from sequences $i=i_0, i_1, \ldots,
  i_\ell = j$, $1 \leq i_1, \ldots, i_{\ell-1} \leq n$ together with,
  for each $c_r=t$ (resp. $c_r= t^{-1}$) term, a left-to-right
  (resp. right-to-left) section of the $xy$-diagram of $\beta$,
  beginning on strand $i_{r-1}$ of $\beta$ and ending on strand
  $i_{r}$ of $\beta$, with corners as in the definition of the path
  matrix $P^{xy}_\beta$ (resp. $Q^{xy}_\beta$); see Definition
  \ref{def:pathmatrix}.  (For terms with $c_r =b_k$, the corresponding
  thick disk has a negative corner at $b^k_{i_{r-1},i_r}$ as before).
  From the definition of the path matrices, since
  \[
  \Phi(t) = P^{xy}_\beta \quad \mbox{and} \quad \Phi(t^{-1}) =
  (P^{xy}_\beta)^{-1} = Q^{xy}_\beta \quad \quad \mbox{(by
    Proposition~\ref{prop:PathProp} (2))},
  \]
  we see that once again the contribution to $\partial(a_{i,j})$ from
  all such thick disks is the $(i,j)$-entry of
  $\Sigma\cdot\Phi(\iota(u) c_1\cdots c_\ell)\cdot\Sigma.$
	
  This completes the entry-by-entry check that the contribution of
  thick disks to $\partial A_k$ is the term
  $\Sigma\cdot\Phi(\partial(a_k))\cdot\Sigma$.

  \medskip\noindent (Thin disks) There are two types of thin disks
  contributing to $\partial(a_{i,j})$: triangles with negative corners
  at $y_{i,\ell}$ and $a_{\ell,j}$ for some $i<\ell$ and triangles
  with negative corners at $a_{i,\ell}$ and $y_{\ell,j}$ for some
  $\ell<j$. The former contributes the term
  \[\iota'\iota_{0}y_{i,\ell}\iota_{1}a_{\ell,j}\iota_{2}=(-1)^{\mu_i+1}\iota_0y_{i,\ell}(-1)^{\mu_\ell}a_{\ell,j}\iota_2=-(-1)^{\mu_i}y_{i,\ell}(-1)^{\mu_\ell}a_{\ell,j}\]
  to $\partial(a_{i,j})$ since $\iota_0=\iota_2$ and thus $-\Sigma Y\Sigma A_k$ to $\partial
  A_k$. The latter contributes the term $\iota'\iota_{0}a_{i,\ell}\iota_{1}y_{\ell,j}\iota_{2}$. By considering the orientations of strand $j$ and the two possible configurations of the disk (the quadrant covered by the disk near $a^k_{i,j}$ can either be to the right or the left of $a^k_{i,j}$), one can check that \[\iota'\iota_2=(-1)^{\lvert a^k_{ij}\rvert+1}=(-1)^{\lvert a_k\rvert+\mu_i+\mu_j+1}.\]
  The signs of the quadrants covered by the disk near $a^k_{i,\ell}$ and $a^k_{i,j}$ only depends on whether strands $j$ and $\ell$ are oriented the same direction or not, so $\iota_0\iota_1=(-1)^{\mu_j+\mu_\ell+1}$. Therefore the disk contributes
  \[\iota'\iota_{0}a_{i,\ell}\iota_{1}y_{\ell,j}\iota_{2}=(-1)^{\lvert a_k\rvert+\mu_i+\mu_j+1}(-1)^{\mu_j+\mu_\ell+1}a_{i,\ell}y_{\ell,j}=(-1)^{\lvert
    a_k\rvert}(-1)^{\mu_i}a_{i,\ell}(-1)^{\mu_\ell}y_{\ell,j}\] to
  $\partial(a_{i,j})$ and thus $(-1)^{\lvert a_k\rvert}\Sigma A_k\cdot\Sigma
  Y$ to $\partial A_k$.

  \medskip\noindent {\bf Differential of $Y$:} As in \cite{L1, Sab1},
  the only disks contributing to $\partial y_{i,j}$ are disks with
  three corners, a positive corner at $y_{i,j}$ and negative corners
  at $y_{i,k}$ and $y_{k, j}$ for some $i<k<j$. Note that the disks are all contained in the dip and the signs for the disks are the same whether the disks are viewed in $J^1S^1$ or $S(\Lambda,\beta)$. We will view the disks as being in $J^1S^1$. We know strand $i$ is
  oriented to the right if and only if $\mu_i\equiv0\mod2$, so
  $\iota'=(-1)^{\mu_i+1}$. The orientation signs at the corners depend
  only on the orientation of the under-strand at the crossing, so the
  disk contributes the
  term \[(-1)^{\mu_i+1}(-1)^{\mu_j+1}y_{i,k}(-1)^{\mu_k}y_{k,j}(-1)^{\mu_j+1}=-(-1)^{\mu_i}y_{i,k}(-1)^{\mu_k}y_{k,j}.\]
  So
  \[\partial
  y_{i,j}=-\sum_{i<k<j}(-1)^{\mu_i}y_{i,k}(-1)^{\mu_k}y_{k,j}.\]
  Therefore $\partial Y=-\Sigma Y\Sigma Y.$

\end{proof}

\begin{corollary} \label{cor:diff} Any generator $x = a_k$ or $x=t$ of
  $\mathcal{A}(K)$ satisfies
  \[
  \Sigma \cdot \partial \circ\Phi(x)=\Phi\circ\partial(x) -
  (Y\Sigma)\cdot \Phi(x)+(-1)^{\lvert x\rvert}\Phi(x) \cdot (Y\Sigma).
  \]
\end{corollary}
\begin{proof}
  The case $x=a_k$ is in Proposition~\ref{prop:DiffProof}; the case
  $x=t$ is Proposition~\ref{prop:Pxy1}.
\end{proof}

\begin{remark}
  \begin{itemize}
  \item[(i)] A partial computation of $S(K,\beta)$ over $\Z/2$ is done
    in \cite{NgR}, Theorem 4.16 from the front diagram perspective.
  \item[(ii)] The special case of the $n$-copy, i.e. when $\beta$ is
    the identity braid, appears in \cite{NRSSZ}, Proposition~3.25.
    That computation has some differences in signs when compared with
    our Proposition~\ref{prop:DiffProof} as a result of different (but
    equivalent, see \cite{NgLSFT}) sign conventions for defining the
    differential over $\Z$.
  \end{itemize}
\end{remark}

\section{Satellite ruling polynomials from representation
  numbers} \label{sec:Sat}

We are now prepared to construct a correspondence from augmentations
of $\alg(S(K, \beta))$ to a set consisting of ordered pairs $(d, f)$
where $d$ is a differential on a fixed $n$-dimensional graded vector
space, $V_\beta$, and $f$ is a representation of $\alg(K)$ on
$(V_\beta,d)$.  The image of $t$ under $f$ is restricted to lie in the
path subset of $\beta$, $B_\beta \subset GL(n,\mathbb{F})$.

\subsection{Augmentations of satellites as representations}
Let the satellite link, $S(K,\beta) \subset J^1\R$, be constructed
from the knot $K \subset J^1\R$ and $n$-stranded positive braid $\beta
\subset J^1S^1$ as in Section \ref{sec:alg}.  In particular, we have
basepoints $* \in K$ and $*_1, \ldots, *_n \in \beta$, and $\Z$-valued
Maslov potentials $\mu_K$ and $\mu_\beta$ where $\mu_K$ is possibly
discontinuous at $*$.

Given a field, $\mathbb{F}$, form a $\Z$-graded $\mathbb{F}$-vector
space
\[
V_\beta = \mbox{Span}_\mathbb{F}\{ e_i \,|\, 1 \leq i \leq n \}, \quad
|e_i| = \mu_i := \mu_\beta(*_i).
\]
In the statement and proof of the following theorem, we implicitly use
the ordered basis $\{e_1, \ldots, e_n\}$ to identify
$\mathit{End}(V_\beta)$ with the matrix ring $\mathit{Mat}(n,
\mathbb{F})$. Using this identification, we write $B^m_\beta \subset
\mathit{GL}(V_\beta)$, where $B^m_\beta$ is the path subset of $\beta$
(from Definition \ref{def:PathSub}).  Note that we call a linear map
$A:V_\beta \rightarrow V_\beta$ {\bf strictly upper triangular} when
its matrix is, i.e. if $A e_{j} = \sum_{i<j} a_{i,j} e_i$. Recall how
$\delta$ is induced by $d$ from Section~\ref{sec:repDGVS}.

Recall the notations $\overline{\mbox{Rep}}_m(K, (V,d), B)$ and
$\overline{\mbox{Aug}}_m(K, \mathbb{F})$ from Section
\ref{sec:RepNumber}, and the algebra homomorphism $\Phi:\mathcal{A}(K)
\rightarrow \mathit{Mat}(n, \mathcal{A}(S(K,\beta)))$ from
(\ref{eq:DefPhi}).

\begin{theorem} \label{thm:bijection} Let $m\,|\, 2 r(K)$, and let
  $\mathbb{F}$ be a field.  If $m$ is odd, then assume
  $\mathit{char}(\mathbb{F})=2$.

  Given an $m$-graded augmentation $\epsilon: \alg(S(K, \beta))
  \rightarrow \mathbb{F}$, there is a strictly upper triangular
  differential
  \[d: V_\beta \rightarrow V_\beta\] with $\deg(d) = +1$ mod $m$ and
  an $m$-graded representation
  \[f: (\alg(K), \partial) \rightarrow (-\mathit{End}(V_\beta),
  \delta)\] (with $\delta$ induced by $d$) defined by the matrices
  \begin{equation} \label{eq:themap} d = \epsilon(Y\Sigma) \quad \mbox{and}
    \quad f = \epsilon \circ \Phi.
  \end{equation}
  This correspondence $\epsilon \mapsto (d,f)$ defines a map
  \[
  \overline{\mbox{Aug}}_m(S(K,\beta), \mathbb{F}) \rightarrow
  \bigsqcup_{d}\overline{\mbox{Rep}}_m(K, (V_\beta,d), B^m_\beta),
  \]
  where the union is over all strictly upper triangular differentials
  $d$ with $\deg(d)=+1$ mod $m$.

  Moreover, the map is injective when $\beta$ is $\mathbb{F}$-path
  injective, and is surjective when $\beta$ is path generated.
\end{theorem}

\begin{proof} 
  First, observe that $\epsilon \mapsto (d,f)$ as in (\ref{eq:themap})
  defines a surjection, $\Psi$, from the set,
  $\mathit{Ring}_m(\mathcal{A}(S(K,\beta)), \mathbb{F})$, of ring
  homomorphisms that preserve grading mod $m$ to the set of ordered
  pairs $(d,f)$ with $d:V_\beta \rightarrow V_\beta$ a strictly
  upper-triangular linear map of degree $+1$ mod $m$ and $f:
  \mathcal{A}(K) \rightarrow -\mathit{End}(V)$ a degree $0$ mod $m$
  ring homomorphism with $f(t) \in B_\beta$.  [To verify the statement
  about grading, note that
  \[
  d e_j = \sum_{i<j} \epsilon(y_{i,j}) (-1)^{\mu_j}e_i.
  \]
  Since $|y_{i,j}| = \mu_{i}-\mu_{j}-1$, when $\epsilon$ is
  $m$-graded, if $e_i$ appears with nonzero coefficient in $d e_j$ we
  must have $\mu_i-\mu_j -1 =0 \, \mbox{mod}\, m$.  This is equivalent
  to $|e_i| = |e_j| +1 \, \mbox{mod} \,m$.  Thus, $d$ has degree $+1$
  mod $m$ in $\mathit{End}(V_\beta)$.  Similarly, since $|a^k_{i,j}| =
  \mu_i -\mu_j + |a_k|$, the map $(\epsilon \circ f)(a_k)$ has degree
  $-|a_k|$ mod $m$ in $\mathit{End}(V)$ which translates to degree
  $|a_k|$ mod $m$ in $-\mathit{End}(V)$.  Finally, from Proposition
  \ref{prop:PathProp} (3), the $(i,j)$-entry of the path matrix
  $P^{xy}_\beta$ has degree $\mu_i-\mu_j$ in $\mathcal{A}(\beta)$, so
  $(\epsilon \circ f)(t) = \epsilon ( P^{xy}_\beta)$ has degree $0$
  mod $m$.  The statement about surjectivity follows from the
  definition of $B^m_\beta$, and since any matrix entries of $d$ or
  $f(a_k)$ that are allowed to be non-zero by the mod $m$ grading
  condition can be made arbitrary by the choice of $\epsilon$.]
  
  Next, since the values of a ring map $\epsilon \in
  \mathit{Ring}_m(\mathcal{A}(S(K,\beta)), \mathbb{F})$ on the
  generators $a^{k}_{i,j}$ (resp. $y_{i,j}$) are always uniquely
  determined by the matrix entries of $f(a_k)$ (resp. $d$), we see
  that $\Psi$ is also injective when $\beta$ is $\mathbb{F}$-path
  injective.

  \medskip

  \noindent{\bf Claim 1.}  If $\epsilon$ satisfies the augmentation
  equation, $\epsilon \circ \partial =0$, then the pair $(d,f)$
  satisfies $d^2=0$ and $f \circ \partial = \delta \circ f$.

  \medskip

  Using the formulas from Proposition~\ref{prop:DiffProof} and
  Corollary~\ref{cor:diff}, we observe the following three
  equivalences:
  \begin{equation} \label{eq:equiv1} (\epsilon\circ \partial)(Y) = 0
    \quad \Leftrightarrow \quad [\epsilon(Y\Sigma)]^2 =0 \quad
    \Leftrightarrow \quad d^2=0.
  \end{equation}  
  \begin{align} \label{eq:equiv2}
    (\epsilon \circ \partial)(A_k) = 0 \quad &\Leftrightarrow \quad  \epsilon\left(\Sigma\cdot\partial(A_k\Sigma)\right) = 0 \\
    &\Leftrightarrow \quad \epsilon \left(\Phi\circ\partial(a_k) - (Y\Sigma)\cdot \Phi(a_k)+(-1)^{\lvert a_k\rvert}\Phi(a_k) \cdot (Y\Sigma)\right) = 0 \notag \\
    &\Leftrightarrow \quad f\circ\partial(a_k)= \epsilon(Y\Sigma)\cdot f(a_k)-(-1)^{\lvert a_k\rvert}f(a_k) \cdot \epsilon(Y\Sigma)  \notag \\
    & \Leftrightarrow \quad f\circ \partial(a_k) = \delta \circ f
    (a_k) \notag
  \end{align}
  A similar sequence of equivalences establishes
  \begin{equation} \label{eq:equiv3} (\epsilon
    \circ \partial)(P^{xy}_\beta) =0 \quad \Leftrightarrow \quad
    f\circ \partial(t) = \delta \circ f (t).
  \end{equation}
  Thus, when $\epsilon \circ \partial =0$ holds, we see that $d^2=0$
  and the representation equation $f \circ \partial = \delta \circ f$
  holds when applied to any generator of $\mathcal{A}$.  This
  establishes Claim 1, which shows that the restriction of $\Psi$ to
  $\overline{\mbox{Aug}}_m(S(K,\beta), \mathbb{F})$ indeed has its
  image in $\bigsqcup_{d}\overline{\mbox{Rep}}_m(K, (V_\beta,d),
  B_\beta)$.

  The remaining statement about surjectivity follows from combining
  the surjectivity of $\Psi$ with the following:

  \medskip

  \noindent {\bf Claim 2.}  When $\beta$ is path generated, the
  converse to Claim 1 also holds.

  \medskip

  Supposing $\Psi(\epsilon) = (d,f)$ with $d^2=0$ and $\delta \circ f
  = f \circ \partial$, the equivalences (\ref{eq:equiv1}) and
  (\ref{eq:equiv2}) show that $\epsilon \circ \partial (z) =0$ for
  generators of the form $z=a^k_{i,j}$ and $z=y_{i,j}$; it is always
  the case that $\epsilon \circ \partial(t^{\pm1}_i) = 0$.  Finally,
  from the equivalence (\ref{eq:equiv3}) we have $\epsilon
  \circ \partial(P^{xy}_\beta)=0$; when $\beta$ is path generated this
  implies (with notation as in \ref{sec:pathaug}) $\epsilon
  \circ \partial|_{\mathcal{B}} = 0$ so that $\epsilon
  \circ \partial(z)=0$ holds for the remaining generators $z= x_{i,j}$
  and $z = p_i$.
\end{proof}

The following corollary generalizes a result of Sabloff from \cite{Sab1}.

\begin{corollary}  \label{cor:rotation}
Let $K \subset J^1\R$ be a (connected) Legendrian knot with base point $*$.  If the DGA $\mathcal{A}(K,*)$ has a $2$-graded representation on a DG-vector space $(V,0)$ with $V$ non-zero and concentrated in degree $0$, then $r(K) =0$.
\end{corollary}
\begin{proof}
Let $f \in \overline{\mbox{Rep}}_2(K, (V,0), GL(V))$ where $\dim V =n$, and fix an isomorphism $V \cong \mathbb{F}^n$ leading to $GL(V) \cong GL(n,\mathbb{F})$.  By Proposition \ref{prop:bruhat}, $f(t) \in B_\beta$ for some reduced positive permutation braid $\beta \in S_n$.  Then, equipping $\beta$ with the constant Maslov potential $\mu_\beta \equiv 0$, we have that $f \in \overline{\mbox{Rep}}_2(K, (V_\beta,0), B^2_\beta)$.  Since $\beta$ is path generated and $\mathbb{F}$-path injective (by Proposition \ref{prop:permutation}), Theorem \ref{thm:bijection} shows that $S(K, \beta)$ has a $2$-graded augmentation to $\mathbb{F}$.  Consequently, (see \cite[Theorem 1.1]{L1}) $S(K,\beta)$ has a $2$-graded normal ruling.  Then, a result of Sabloff, (\cite[Theorem 1.3]{Sab1} extended in a straight-forward manner to the multi-component case), shows that the sum of the rotation numbers of the components of $S(K,\beta)$ must be $0$.  This sum is $n \cdot r(K)$, so $r(K) =0$ follows.   
\end{proof}

To state a precise formula for augmentation numbers of satellites in
terms of representation numbers, we introduce some preliminary
notation.  Given an $n$-stranded positive braid $\beta \subset J^1
S^1$ equipped with a Maslov potential, $\mu_\beta$, let
$\left(q_l\right)_{l \in \Z}$ denote the degree distribution of the
$xz$-crossings of $\beta$, i.e. $q_l$ is the number of $p_i$
generators with $|p_i|=l$.  For $m\geq 0$, let $\lambda_m(\beta) =
\sigma_m((q_l))$, where we view $\sigma_m$ (defined in (\ref{eq:chi2})
or (\ref{eq:sigma11})-(\ref{eq:sigma13})), as a $\Z$-module
homomorphism $\Z^{-\infty,\infty} \rightarrow \Z$.  Note that when
$\mu_\beta$ is constant, $\lambda_m(\beta)$ is just
the length of $\beta$.

\begin{theorem} \label{thm:A} Suppose that $K\subset J^1\R$ is a
  (connected) Legendrian knot, and let $m \, | \, 2r(K)$ and
  $\mathbb{F}_q$ a finite field.  If $m$ is odd, then assume
  $\mathit{char}(\mathbb{F}_q)=2$; if $m$ is even, then assume $r(K) =
  0$.
	
  If $\beta \subset J^1S^1$ is path generated and $\mathbb{F}_q$-path
  injective, then we have
  \begin{equation} \label{eq:thmAug} \mbox{Aug}_m(S(K, \beta),q) =
    q^{-\lambda_m(\beta)/2} (q^{1/2}-q^{-1/2})^{-n} \cdot \sum_d
    \rrep_m(K, (V_\beta, d), B^m_\beta)
  \end{equation}
  where the sum is over all upper triangular differentials $d:V_\beta
  \rightarrow V_\beta$ with $\deg(d) = +1$ mod $m$.
\end{theorem}

\begin{proof}
  From the definition, we have
  \begin{equation} \label{eq:proofEq1} \mbox{Aug}_m(S(K, \beta),q) =
    q^{-\sigma_m(S(K,\beta))/2}(q-1)^{-n}\,
    |\overline{\mbox{Aug}}_m(S(K, \beta),\mathbb{F}_q)|.
  \end{equation}
  Write 
	\[
	\sigma_m(S(K,\beta)) = X_a + X_{x,y} + X_p
	\]
	where the $3$
  terms are the contribution to $\sigma_m(S(K,\beta))$ (as defined in
  equation (\ref{eq:chi2})) from the degree distribution of Reeb
  chords of the form $a^k_{i,j}$; $x_{i,j}$ or $y_{i,j}$; and $p_i$
  respectively. 
	
  \begin{lemma} \label{lem:Xa} Let $\mathcal{B}=-\mathit{End}(V_\beta)$.
    The following identities hold:
    \begin{enumerate}
    \item $\displaystyle q^{-X_a/2} = \lim_{N \rightarrow +\infty}
      \prod_{k \in \Z, \lvert k\rvert\leq N}
      \lvert\mathcal{B}^m_k\rvert^{-(\chi^k/2)}$.
    \item $\displaystyle X_{x,y} = \sum_{k \in \Z/m} 2
      \left( \begin{array}{c} \mathbf{n}(k) \\ 2 \end{array} \right)$
      where $\mathbf{n}: \Z/m \rightarrow \Z_{\geq 0}$ is the
      $\Z/m$-graded dimension of $V_\beta$, i.e. $\mathbf{n}(k) =
      \left|\left\{ e_i \, |\, |e_i| = k \, \mbox{mod} \, m\right\}
      \right|$.
    \item $X_p = \lambda_m(\beta)$.
    \end{enumerate}
  \end{lemma}
	
  \begin{proof}[Proof of Lemma \ref{lem:Xa}]
    {\bf (1):} Notice that as $\lvert\mathcal{B}^m_k\rvert = q^{\dim
      \mathcal{B}^m_k}$ we have
    \[
    \prod_{k \in \Z, \lvert k\rvert\leq N}
    \lvert\mathcal{B}^m_k\rvert^{-(\chi^k/2)} = q^{-\left( \sum_{k \in
          \Z, \lvert k\rvert\leq N} (\dim \mathcal{B}^m_k)
        \chi^k\right)/2},
    \]
    so it suffices to show that
    \begin{equation} \label{eq:Xaequiv} X_a = \lim_{N \rightarrow
        +\infty} \sum_{k \in \Z, \lvert k\rvert\leq N} (\dim
      \mathcal{B}^m_k) \chi^k.
    \end{equation}
	
    Fix $M \gg 0$ so that
    \begin{equation} \label{eq:defM} \chi^{-k} = -\chi^{k} \quad
      \mbox{and} \quad \mathcal{B}_k = \mathcal{B}_{-k} = \{0\} \quad
      \mbox{whenever $|k| \geq M$.}
    \end{equation}
    Then, using that $\dim \mathcal{B}^m_{k} = \dim
    \mathcal{B}^m_{-k}$ holds for all $k$, we have that
    \[
    B:= \lim_{N \rightarrow +\infty} \sum_{k \in \Z, \lvert
      k\rvert\leq N} (\dim \mathcal{B}^m_k) \chi^k = \sum_{\lvert
      k\rvert\leq M} (\dim \mathcal{B}^m_k) \chi^k.
    \]
		
    \begin{lemma} \label{lem:formulaA} For $b \in \Z$, let
      $\chi^b_{a}$ denote the contribution to $\chi^b(S(K,\beta))$
      from generators of the form $a^k_{i,j}$.  Then,
      \[
      \chi^b_{a} = \sum_{k \in \Z} \left(\dim
        \mathcal{B}_k\right)\chi^{b+k}(K)
      \]
      where the sum is over the finitely many terms with
      $\mathcal{B}_k \neq \{0\}$.
    \end{lemma}

\begin{proof}[Proof of Lemma \ref{lem:formulaA}]  
  Let $r_l$ (resp. $s_l$) denote the number of degree $l$ Reeb chords
  of $K$ (resp. of $S(K,\beta)$ of the form $a^k_{i,j}$), and set
  \[
  m_k = \left\lvert\{(i,j) \,|\, 1\leq i,j \leq n, \,\, \mu_j-\mu_i =
    k\}\right\rvert.
  \]
  Note that $s_p = \sum_{k \in \Z} m_k r_{p+k}$, so we can compute
  \begin{align*}
    \chi^b_a =& \sum_{l \geq 0} (-1)^l s_{b+l} + \sum_{l<0} (-1)^{l+1}s_{b+l} \\
    = &  \sum_{l \geq 0} (-1)^l \left(\sum_{k \in \Z} m_kr_{b+l+k}\right) + \sum_{l<0} (-1)^{l+1}\left(\sum_{k \in \Z} m_kr_{b+l+k}\right) \\
    = & \sum_{k \in \Z} m_k \cdot\left(\sum_{l \geq 0} (-1)^l r_{b+l+k} + \sum_{l<0} (-1)^{l+1}r_{b+l+k}\right) \\
    = & \sum_{k \in \Z} (\dim \mathcal{B}_k) \cdot \chi^{b+k}(K).
  \end{align*}
\end{proof}
  
Returning now to the proof of (1) in Lemma \ref{lem:Xa}, if $m=0$ we
have $X_a = \chi_a^0$ and $\mathcal{B}^m_k = \mathcal{B}_k$, so the
required equality (\ref{eq:Xaequiv}) is simply Lemma
\ref{lem:formulaA}.

In the case where $m \neq 0$, we fix $N \gg M$ to be large enough so
that the following hold:
\begin{enumerate}
\item[(a)] $X_a = \chi^0_a + \sum_{c=1}^N\left( \chi^{cm}_a +
    \chi^{-cm}_a \right)$, (see equation (\ref{eq:chi2})).
\item[(b)] For all $l$ with $|l| \leq M$, if $\mathcal{B}_k \neq 0$,
  then $k \in [ l-Nm, l+Nm]$.
\end{enumerate}
We show that $X_a - \chi^0_a = B- \chi^0_a$ by calculating as follows:
{\allowdisplaybreaks\begin{flalign*}
    && X_a - \chi^0_a = & \sum_{c=1}^N \left(\chi^{cm}_a + \chi^{-cm}_a \right) \\
    & \mbox{(Lemma \ref{lem:formulaA})} \quad &  = & \sum_{c=1}^N \left( \sum_{|k| \leq M} (\dim \mathcal{B}_k) \chi^{cm+k} +  \sum_{|k| \leq M} (\dim \mathcal{B}_k) \chi^{-cm+k} \right) \\
    & \mbox{(Replace $k$ with $-k$ in 2nd sum)} \quad & = & \sum_{c=1}^N \left( \sum_{|k| \leq M} (\dim \mathcal{B}_k) \chi^{cm+k} +  \sum_{|k| \leq M} (\dim \mathcal{B}_{-k}) \chi^{-(cm+k)} \right) \\
    & \mbox{(Use (B2))} \quad & = & \sum_{c=1}^N  \left(\sum_{ k \in \Z} \dim \mathcal{B}_k \left( \chi^{cm+k}+ \chi^{-(cm+k)} \right) \right) \\
    & \mbox{(Reindex $l=cm+k$)} \quad & = & \sum_{c=1}^N  \left(\sum_{l \in \Z} \dim \mathcal{B}_{l-cm} \left( \chi^{l}+ \chi^{-l} \right) \right) \\
    & \mbox{(Use (\ref{eq:defM}))} \quad & = & \sum_{c=1}^N  \left(\sum_{|l| \leq M} \dim \mathcal{B}_{l-cm} \left( \chi^{l}+ \chi^{-l} \right) \right) \\
    & \quad & = &  \sum_{|l| \leq M} \left( \chi^{l}+ \chi^{-l} \right) \sum_{c=1}^N \dim \mathcal{B}_{l-cm}  \\
    & \mbox{(Replace $l$ with $-l$ in 2nd sum)} & = &  \sum_{|l| \leq M} \chi^{l}\left(\sum_{c=1}^N \dim \mathcal{B}_{l-cm}\right) + \sum_{|l| \leq M} \chi^{l}\left(\sum_{c=1}^N \dim \mathcal{B}_{-l-cm}\right)  \\
    & \mbox{(Use (B2))} & = &  \sum_{|l| \leq M} \chi^{l}  \left( \sum_{c=1}^N \left(\dim \mathcal{B}_{l-cm} + \dim \mathcal{B}_{l+cm}\right) \right) \\
    & \mbox{(Use (b))} & = &  \sum_{|l| \leq M} \chi^{l} \cdot \left(\dim \mathcal{B}^m_{l} - \dim \mathcal{B}_{l}\right)  \\
    & \mbox{(Lemma \ref{lem:formulaA})} \quad & = & B- \chi^0_a.
\end{flalign*}}
  
\medskip
  
  \noindent {\bf (2):} 
  Next, since $\vert x_{i,j}\rvert=\lvert y_{i,j}\rvert+1$, these pairs of generators
  cancel in all $\chi^l$ with $l=0 \, \mbox{mod} \, m$, except when
  \[
  \lvert x_{i,j}\rvert = 0 \, \mbox{mod} \, m \quad \Leftrightarrow \quad i<j
  \mbox{ and } \mu(i) = \mu(j) \, \mbox{mod}\, m,
  \]
  in which case there is a single such $l$ where the contribution from
  $x_{i,j}$ and $y_{i,j}$ to $\chi^l$ is $2$ rather than $0$.  As the
  $\Z/m$-graded dimension of $V_\beta$ satisfies $\mathbf{n}(k)= |\{i
  \,|\, \mu(i) = k \, \mbox{mod} \, m\}|$, it follows that
  \[
  X_{x,y} = \sum_{k \in \Z/m} 2 \left(\begin{array}{c} \mathbf{n}(k)
      \\ 2 \end{array}\right).
  \]
  
  \medskip
  
  \noindent {\bf (3):} Note that $X_p = \lambda_m(\beta)$ by
  definition.
  
\end{proof}
	
With $q^{-\sigma_m(S(K,\beta))/2}$ evaluated using Lemma~\ref{lem:Xa}
and making use of the bijection from Theorem~\ref{thm:bijection},
equation \eqref{eq:proofEq1} leads to
\begin{align}
  \mbox{Aug}_m(S(K, \beta),q) =& q^{-\lambda_m(\beta)/2}q^{-\sum_{k
      \in \Z/m} (\mathbf{n}(k)^2-\mathbf{n}(k))/2}
  % \left(\begin{array}{c} \mathbf{n}(k) \notag \\
  %     2 \end{array}\right)}
  (q-1)^{-n}\notag \\
  & \quad \quad \times\sum_{d}\left(\lim_{N \rightarrow +\infty}
    \prod_{k \in \Z, \lvert k\rvert\leq N}
    \lvert\mathcal{B}^m_k\rvert^{-(\chi^k/2)}\right)\cdot \left\lvert\overline{\mbox{Rep}}_m(K, (V_\beta,d), B_\beta)\right\rvert \notag \\
  =& q^{-\lambda_m(\beta)/2}q^{-\sum_{k \in \Z/m}
    (\mathbf{n}(k)^2-\mathbf{n}(k))/2} (q-1)^{-n} \sum_{d}
  |\mathcal{B}^m_0|^{1/2} \cdot \rrep_m(K, (V_\beta,d), B_\beta),
  \notag \label{eq:ProofLine3}
\end{align}
since, by Definition~\ref{def:repNum} and \ref{def:reducedRep},
\begin{align*}
  \widetilde{\mbox{Rep}}_m(K,(\mathcal{B},\delta),\mathbb{T})=&\lvert\mathcal{B}^m_0\rvert^{-1/2}\lvert(\mathcal{B}_0^m)^*\cap\ker\delta\rvert\\
  &\quad\times\left(\lim_{N \rightarrow +\infty} \prod_{k \in \Z,
      \lvert k\rvert\leq N} \lvert\mathcal{B}^m_k\rvert^{-(\chi^k/2)}
  \right)\cdot \lvert(\mathcal{B}^m_0)^* \cap \ker
  \delta\rvert^{-\ell} \cdot \lvert\overline{\mbox{Rep}}_m(K,
  (\mathcal{B},\delta), \mathbf{T})\rvert,
\end{align*}
where $\ell=1$ since we are considering the original knot $K$. Recall
from Section~\ref{sec:repDGVS} that ${|\mathcal{B}^m_0|^{1/2} =
q^{\sum_{k \in \Z/m} \mathbf{n}(k)^2/2}}$, so this becomes
\begin{align*}
  & q^{-\lambda_m(\beta)/2} q^{\sum_{k \in \Z} \mathbf{n}(k)/2}
  (q-1)^{-n} \sum_{d} \rrep_m(K, (V_\beta,d), B_\beta) \\
  &\quad= q^{-\lambda_m(\beta)/2} (q^{1/2}-q^{-1/2})^{-n} \sum_{d}
  \rrep_m(K, (V_\beta,d), B_\beta).
\end{align*}
\end{proof}

\begin{example}\label{ex:trefoil}
  We will consider the satellite of the right-handed trefoil with the
  basic front $A_2$ (as in Figure \ref{fig:Basic}) and show by direct
  computation that the equation in Theorem~\ref{thm:A} holds for $q=2$
  when $m=0$.

  Let $K$ be the Legendrian right-handed trefoil with crossings
  labeled as in Figure~\ref{fig:trefoil}. We have $\lvert
  a_1\rvert=\lvert a_2\rvert=\lvert a_3\rvert=0$, $\lvert
  a_4\rvert=\lvert a_5\rvert=1$, and $\lvert t\vert=0$ with
  differential
  \begin{align*}
    &\partial a_4=t^{-1}+a_1+a_3+a_1a_2a_3\\
    &\partial a_5=1-a_1-a_3-a_3a_2a_1\\
    &\partial a_1=\partial a_2=\partial a_3=0
  \end{align*}

  Let $\beta=A_2$ with both strands oriented to the right. Label the
  crossing in $\beta$ by $p$ and for ease of notation set $x=x_{12}$
  and $y=y_{12}$. If we set $\mu_1=0=\mu_2$ (and thus $\Sigma=I$),
  then $\lvert p\rvert=0$. The differential satisfies $\partial
  p=-t_1^{-1}yt_2$ and the path matrix is
  \[P^{xy}_\beta=\Delta\begin{pmatrix}p&px+1\\1&x\end{pmatrix},\]
  where $\Delta$ is the matrix with $t_1$ and $t_2$ along the
  diagonal.

  Since $\mu_1=\mu_2$, in $S(K,\beta)$, we have $\lvert
  a^1_{ij}\rvert=\lvert a^2_{ij}\rvert=\lvert a^3_{ij}\rvert=0$,
  $\lvert a^4_{ij}\rvert=\lvert a^5_{ij}\rvert=1$, $\lvert
  p\rvert=\lvert x\rvert=0$, $\lvert y\rvert=-1$, and $\lvert
  t_1\rvert=\lvert t_2\rvert=0$. The definition of the differential
  and Proposition~\ref{prop:DiffProof} give us that the differential
  satisfies
  \begin{align*}
    &\partial p=-t_1^{-1}yt_2\\
    &\partial X=(I+X)Y-\Delta^{-1}Y\Delta(I+X)\\
    &\partial Y=0\\
    &\partial A_k=-YA_k+A_kY,\quad k=1,2,3\\
    &\partial A_4=(P^{xy}_\beta)^{-1}+A_1+A_3+A_1A_2A_3-YA_4-A_4Y\\
    &\partial A_5=I-A_1-A_3-A_3A_2A_1-YA_5-A_5Y
  \end{align*}
  Between the definition of the differential and the gradings of the
  generators, all $0$-graded augmentations
  $\epsilon:\mathcal{A}(S(K,\beta),\partial)\to\mathbb{F}_q$ with
  $q=2$ satisfy
  \begin{align*}
    &\epsilon(t_1)=\epsilon(t_2)=1\\
    &\epsilon(Y)=\epsilon(A_4)=\epsilon(A_5)=0\\
    &0=\epsilon((P^{xy}_\beta)^{-1}+A_1+A_3+A_1A_2A_3)\\
    &0=\epsilon(I-A_1-A_3-A_3A_2A_1)
  \end{align*}

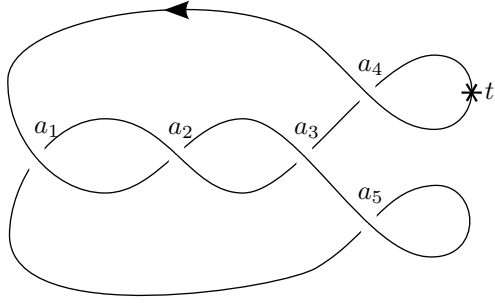
\begin{figure}[t]
  \begin{centering}
    \labellist
    \small
    \pinlabel $a_1$ [b] at 60 229
    \pinlabel $a_2$ [b] at 262 227
    \pinlabel $a_3$ [b] at 452 228
    \pinlabel $a_4$ [b] at 548 322
    \pinlabel $a_5$ [b] at 548 129
    \pinlabel $t$ [l] at 708 310
    \endlabellist
    \includegraphics[width=2.5in]{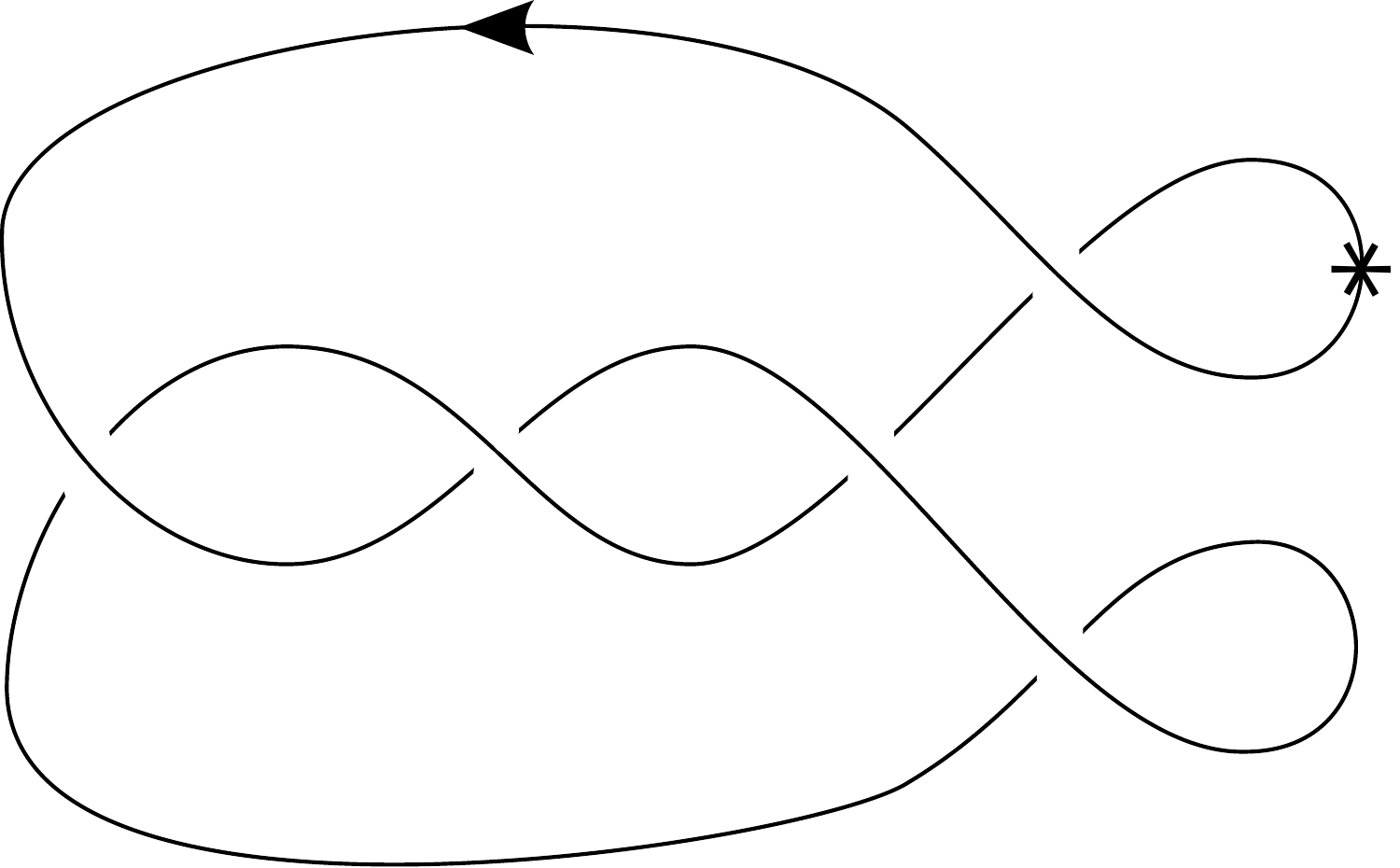} 
    \caption{A right-handed Legendrian trefoil.}
    \label{fig:trefoil}
  \end{centering}
\end{figure}

  With Mathematica one can check that
  \[\lvert\overline{\mbox{Aug}}_0(S(K,\beta),\mathbb{F}_2)\rvert=146.\]
  Recall from \eqref{eq:sigma11} that $\sigma_0=\chi^0$, where
    $\chi^0$ is the shifted Euler characteristic for
    $S(K,\beta$). As the
  shifted Euler characteristic of $S(K,\beta)$ is
  $\chi^0=(-1)^{-1+1}1+(-1)^{0}14+(-1)^{1}8=7$, by definition, the
  augmentation number for $q=2$ of $S(K,\beta)$ is
  \[\mbox{Aug}_0(S(K,\beta),2)=2^{-\sigma_0/2}(2-1)^{-\ell}\lvert\overline{\mbox{Aug}}_0(S(K,\beta),\mathbb{F}_2)\rvert=\frac{73\sqrt2}8.\]

  To compute the right hand side of the equation in
  Theorem~\ref{thm:A}, we first compute
  $\lvert\overline{\mbox{Rep}}_0(K,(V_\beta,d),B_\beta^0)\rvert$ for
  all upper triangular differentials $d:V_\beta\to V_\beta$ with
  $\deg(d)=+1$. Recall that $V_\beta=\mbox{Span}(e_1,e_2)$ where
  $\lvert e_i\rvert=\mu_i=0$. Thus, the only such differential is
  $d=0$.

  First, find the $0$-graded path subset
  \[B_\beta^0=\left\{\alpha(P^{xy}_\beta)\big\vert\,0-\text{graded }\alpha\in
  Ring(\mathcal{B},\mathbb{F}_2)\right\}.\] This amounts to finding for
  which matrices $A\in GL(2,\mathbb{F}_2)$ we can solve
  \[\alpha\begin{pmatrix}p&px+1\\1&x\end{pmatrix}=A\]
  for $\alpha(p)$ and $\alpha(x)$ (since
  $\alpha(t_1)=\alpha(t_2)=1$). Thus
  \[B_\beta^0=\left\{\begin{pmatrix}0&1\\1&0\end{pmatrix},\begin{pmatrix}1&1\\1&0\end{pmatrix},\begin{pmatrix}1&0\\1&1\end{pmatrix},\begin{pmatrix}0&1\\1&1\end{pmatrix}\right\}\subset GL(2,\mathbb{F}_2).\]

  Since $d=0$ on $V_\beta$ induces the differential $\delta=0$ on
  $-\mbox{End}(V_\beta)$ and by definition
  \[\overline{\mbox{Rep}}_0(K,(V_\beta,0),B_\beta^0)=\left\{0-\text{graded
  }f:(\mathcal{A}(K),\partial)\to(-\mbox{End}(V_\beta),\delta)\text{
    s.t.  }f\circ\partial=\delta\circ f, f(t)\in B_\beta^0\right\},\]
  $\lvert\overline{\mbox{Rep}}_0(K,(V_\beta,0),B_\beta^0)\rvert$ is
  the number of $0$-graded
  $f:(\mathcal{A}(K),\partial)\to(-\mbox{End}(V_\beta),\delta)$ such
  that
  \begin{align*}
    &0=\delta\circ f(a_4)=f\circ\partial(a_4)=f(t^{-1}+a_1+a_3+a_1a_2a_3)\\
    &0=\delta\circ f(a_5)=f\circ\partial(a_5)=f(1-a_1-a_3-a_3a_2a_1)\\
    &f(t)\in\left(B_\beta^0\right)^*
  \end{align*}
  Using Mathematica, one can check that there are $40$ such $f$ if
  $f(t)=(\begin{smallmatrix}0&1\\1&0\end{smallmatrix})$, $33$ if
  $f(t)=(\begin{smallmatrix}1&1\\1&0\end{smallmatrix})$, $40$ if
  $f(t)=(\begin{smallmatrix}1&0\\1&1\end{smallmatrix})$, and $33$ if
  $f(t)=(\begin{smallmatrix}0&1\\1&1\end{smallmatrix})$. Thus
  $\lvert\overline{\mbox{Rep}}_0(K,(V_\beta,d),B_\beta^0)\rvert=146$. Notice
  that this matches Theorem~\ref{thm:bijection} which tells us
  $\lvert\overline{\mbox{Aug}}_0(S(K,\beta),\mathbb{F}_2)\rvert=\lvert\overline{\mbox{Rep}}_0(K,(V_\beta,d),B_\beta^0)\rvert$.

  Now we have the $0$-graded representation number of $K$ for $q=2$ is
  \begin{align*}
    &\mbox{Rep}_0(K,(V_\beta,0),B_\beta^0)\\
    &\quad=\left(\lim_{N \rightarrow +\infty} \prod_{k \in \Z, \lvert
        k\rvert\leq N}
      \lvert(-\mbox{End}(V_\beta))^0_k\rvert^{-(\chi^k/2)}
    \right)\cdot\left\lvert(-\mbox{End}(V_\beta))_0^* \cap \ker
    \delta\right\rvert^{-\ell} \cdot
    \left\lvert\overline{\mbox{Rep}}_0(K, (V_\beta,0),B_\beta^*)\right\rvert\\
    &\quad=\lvert(-\mbox{End}(V_\beta))_0\rvert^{-\chi^0/2}\cdot\lvert(-\mbox{End}(V_\beta))_0^*\rvert^{-1}\cdot 146\text{ since }\delta=0\text{ and }(-\mbox{End}(V_\beta))_k^0=(-\mbox{End}(V_\beta))_k\neq0\text{ iff }k=0\\
    &\quad=\left(2^4\right)^{-\frac12}\cdot6^{-1}\cdot
    146\text{ since }(-\mbox{End}(V_\beta))_0=\mbox{End}(V_\beta)\text{ and }\chi^0=1\\
    &\quad=\frac{73}{12}
  \end{align*}
  and thus the $0$-graded reduced representation number for $q=2$ is
  \begin{align*}
    % &\sum_d\widetilde{\mbox{Rep}}_0(K,(V_\beta,d),B_\beta^0)=
    \widetilde{\mbox{Rep}}_0(K,(V_\beta,0),B_\beta^0)&=\left\lvert(-\mbox{End}(V_\beta))_0\right\rvert^{-\frac12}\left\lvert(-\mbox{End}(V_\beta))_0^*\cap\ker\delta\right\rvert\cdot \mbox{Rep}_0(K,(V_\beta,0),B_\beta^0)\\
    &=\left(2^4\right)^{-\frac12}\left\lvert(-\mbox{End}(V_\beta))_0^*\right\rvert\cdot\frac{73}{12}=\frac14\cdot6\cdot\frac{73}{12}=\frac{73}{8}.
  \end{align*}
  Finally, we see that the right hand side of the equation in
  Theorem~\ref{thm:A} is
  \[2^{-\lambda_0(\beta)/2}\left(\sqrt2-\frac1{\sqrt2}\right)^{-2}\cdot\sum_d\widetilde{\mbox{Rep}}_0(K,(V_\beta,d),B_\beta^0)=2^{-\frac12}\cdot2\cdot\frac{73}{8}=\frac{73\sqrt2}8\]
  as desired, where the sum is over all upper triangular differentials
  $d:V_\beta\to V_\beta$ with $\deg(d)=+1$.
\end{example}

\subsection{Satellite ruling polynomials via representation numbers}

For any fixed Legendrian link, $L \subset J^1S^1$, equipped with a
$\Z/m$-valued Maslov potential, the $m$-graded ruling polynomial, $K
\mapsto R^m_{S(K,L)}(z)$, provides a Legendrian isotopy invariant of
(connected) Legendrian knots $K\subset J^1\R$ with $m \,|\, 2r(K)$.
Note that a choice of $\Z/m$-valued Maslov potential on $K$ should be
made to equip $S(K,L)$ with a $\Z/m$-valued Maslov potential, but
$R^m_{S(K,L)}(z)$ is independent of this choice.  We refer to these
invariants as {\bf satellite ruling polynomials}.  Note that in the
case that $L= \beta$ is a positive permutation braid (and $r(K)=0$ if
$m$ is even), the combination of Theorems \ref{thm:HenryR} and
\ref{thm:A} gives the formula
\[
R^m_{S(K,\beta)}(z)\big\vert_{z=q^{1/2}-q^{-1/2}}=q^{-\lambda_m(\beta)/2}(q^{1/2}-q^{-1/2})^{-n}\sum_{d}\rrep_m(K,(V_\beta,d),
B^m_\beta)
\]
as stated in Theorem A from the Introduction.  (Sum over strictly
upper triangular differentials, $d$, of degree $+1$ mod $m$.)

\begin{example}
  In Example~\ref{ex:trefoil}, we saw that if $q=2$, $K$ is the
  right-handed trefoil in Figure~\ref{fig:trefoil}, and $\beta=A_2$,
  then
  \[q^{-\lambda_m(\beta)/2}(q^{1/2}-q^{-1/2})^{-n}\sum_{d}\rrep_m(K,
  (V_\beta,d), B^m_\beta)=\frac{73\sqrt2}8.\] Using a
  \textit{Mathematica} program written by J. Sabloff, we see that the
  $0$-graded ruling polynomial for $S(K,\beta)$ is
  $R^0_{S(K,\beta)}(z)=3z^{-1}+9z+6z^3+z^5$ and so
  \[R^0_{S(K,\beta)}(z)\big\vert_{z=q^{1/2}-q^{-1/2}}=\frac{73\sqrt2}8,\]
  as stated in Theorem~A.
\end{example}

The following corollary shows that the Chekanov-Eliashberg algebra
{\it of the original knot $K$} already contains all information about
satellite ruling polynomials of $K$.

\begin{corollary} \label{cor:A} Fix $m \geq 0$, and consider the class
  of Legendrian knots with $r(K) =0$ if $m$ is even, and $m\,|\,
  2r(K)$ if $m$ is odd.  For any fixed $L \subset J^1S^1$ equipped
  with a $\Z/m$-valued Maslov potential, the satellite ruling
  polynomial invariant
  \[
  K \mapsto R^m_{S(K,L)}(z)
  \]
  is determined by $L$ and the Chekanov-Eliashberg algebra of $K$.  
	\end{corollary}

\begin{proof} The satellite ruling polynomials $R^m_{S(K,L)}(z)$
  satisfy the ruling polynomial skein relations with respect to the
  pattern $L$.  As discussed in Section \ref{sec:background}, after
  applying the skein relations $L$ can be written as a
  $\Z[z^{\pm1}]$-linear combination of products of the basic fronts
  $A_m$ (equipped with arbitrary $\Z/m$-valued Maslov potentials).
  Any such product is a positive permutation braid with reduced braid
  word, and so we can write
  \[
  R^m_{S(K,L)}(z) = \sum_\beta c_\beta R^m_{S(K,\beta)}(z)
  \]
  with $c_\beta \in \Z[z^{\pm1}]$ depending only on $L$, where the sum
  is over all positive permutation braid words $\beta$. That
  $R^m_{S(K,L)}(z)$ is determined by $\mathcal{A}(K)$ at
  $z=q^{1/2}-q^{-1/2}$ for infinitely many values of $q$ follows from
  an application of Theorem A.  Moreover, this uniquely determines
  $R^m_{S(K,L)}$ since it is a Laurent polynomial in $z$.  (Note that
  a choice of $\Z$-valued Maslov potential on $\beta$ lifting the
  $\Z/m$-valued potential on $\beta$ is required for the
  representation numbers appearing on the right hand side of
  (\ref{eq:thmAug}) to be defined, but this choice can be made
  independent from $K$--e.g. by lifting to the range $[0, m)$.)
\end{proof}

\section{Representation numbers from colored ruling
  polynomials} \label{sec:color} The formula from Theorem \ref{thm:A}
shows how augmentation numbers (and hence also ruling polynomials) of
satellites of $K$ can be written in terms of higher dimensional
representation numbers of $K$.  In this section, we provide a
converse-type formula by realizing a total $n$-dimensional
representation number, where no restriction is made on the image of
$t$, in terms of satellite ruling polynomials.  In fact, the
particular combination of satellite ruling polynomials used is itself
a natural object from the point of view of quantum topology, as it is
the Legendrian analog of the colored HOMFLY-PT polynomial.

\begin{definition}  Let $m\geq 0$ have $m\,|\, 2r(K)$ and $m\neq 1$.  
  Define the {\bf $n$-colored $m$-graded ruling polynomial} of $K$ to be
  \[
  R^m_{n,K}(q) = \frac{1}{\alpha_n} \sum_{\beta \in S_n}
  q^{l(\beta)/2} R^m_{S(K, \beta)}(z)
  \]
  where each positive permutation braid $\beta$ has Maslov potential
  $0$; $l(\beta)$ is the length of $\beta$; $z = q^{1/2}-q^{-1/2}$;
  and
  \[
  \alpha_n = (q^{1/2})^{n(n-1)/2}[n][n-1] \cdots [1]
  \]
  with $[r] = \frac{q^{r/2}-q^{-r/2}}{q^{1/2}-q^{-1/2}}$.
\end{definition}

\begin{remark} 
  This is in strict analogy with the definition of the colored
  HOMFLY-PT polynomial.  See Section \ref{sec:coloredHOMFLY}, below.
\end{remark}

\begin{definition}
  We refer to the $m$-graded representation number with $V=
  \mathbb{F}_q^n$ concentrated in degree $0$ (necessarily with $d=0$)
  and $T_1 = GL(V)$ as the {\bf total $n$-dimensional representation
    number} of $K$, and write
  \[
  \mbox{Rep}_m(K, \mathbb{F}_q^n) := \mbox{Rep}_m(K, (V,0), GL(n,
  \mathbb{F}_q)).
  \]
\end{definition}

\begin{theorem} \label{thm:color} Let 
  $K \subset J^1\R$ be Legendrian, and let $m\geq 0$ have $m\,|\, 2r(K)$.  In addition assume $m\neq 1$, and that $r(K)=0$ if $m$ is even.  Then,
  \[
  \mbox{Rep}_m(K, \mathbb{F}_q^n) = R^m_{n, K}(q).
  \]
\end{theorem}

\begin{proof} By Legendrian invariance of $R^m$, we may assume the
  $\beta$ used in computing $R^m_{n,K}$ have reduced braid words.
  
  Note that since $V_\beta$ is concentrated in degree $0$ and $m \neq
  1$, the only $d:V_\beta \rightarrow V_\beta$ of degree $+1$ mod $m$
  is $d=0$.  Thus, using Theorem~A, we compute
  \begin{align}
    R^m_{n,K}(q)=& \frac{1}{\alpha_n} \sum_{\beta \in S_n} q^{l(\beta)/2} \cdot R^m_{S(K, \beta)}(z) \notag \\
    =& \frac{1}{\alpha_n} \sum_{\beta \in S_n} q^{l(\beta)/2}
    q^{-\lambda_m(\beta)/2}z^{-n} \cdot \rrep_m(K,(\mathbb{F}_q^n,0),
    B^m_\beta). \notag
  \end{align}
  Now, since all $p_i$ and $x_{i,j}$ generators have degree $0$,
  $\lambda_m(\beta) = l(\beta)$ and $B^m_\beta = B_\beta$, so the
  above becomes
  \begin{align}
    & \frac{1}{\alpha_n} \sum_{\beta \in S_n} z^{-n} \cdot \rrep_m(K,(\mathbb{F}_q^n,0), B_\beta). \notag \\
    \quad \quad \quad &\quad=\frac{1}{\alpha_nz^n} \cdot \rrep_m(K,(\mathbb{F}_q^n,0), GL(\mathbb{F}_q^n)) \notag \\
    &\quad=\frac{1}{(q^{1/2})^{(n^2-n)/2}\prod_{m=1}^nq^{-m/2}(q^{m}-1)}
    \cdot
    |\mathcal{B}_0|^{-1/2}|GL(\mathbb{F}_q^n)|\cdot\mbox{Rep}_m(K,\mathbb{F}_q^n) \label{eq:line4}
  \end{align}
  where at the first equality we used that $GL(\mathbb{F}_q^n) =
  \bigsqcup_{\beta \in S_n} B_\beta$ by Proposition~\ref{prop:bruhat},
  and at the last we used the definition of the reduced representation
  numbers.  Using that $|\mathcal{B}_0| = q^{n^2}$ and
  \[
  |GL(\mathbb{F}_q^n)| = (q^{n}-1)(q^n-q) \cdots(q^n-q^{n-1}) =
  q^{(n^2-n)/2} \prod_{m=1}^n(q^m-1),
  \]
  (\ref{eq:line4}) simplifies to
  \begin{align*}
    & \frac{q^{-n^2/2}
      q^{(n^2-n)/2}}{(q^{1/2})^{(n^2-n)/2}(q^{1/2})^{(-n^2-n)/2}}
    \cdot \mbox{Rep}_m(K, \mathbb{F}_q^n) = \mbox{Rep}_m(K,
    \mathbb{F}_q^n).
  \end{align*}

\end{proof}

\begin{remark}
  We did not define colored ruling polynomials in the case $m=1$,
  since the above proof breaks down.  (When applying Theorem
  \ref{thm:A}, the many differentials with $d \neq 0$ need to be
  included.)  We leave the correct definition of $R^1_{n,K}$, i.e. one
  where Theorem \ref{thm:color} will hold, as an open problem.
\end{remark}

\subsection{Relation with colored HOMFLY-PT
  polynomials} \label{sec:coloredHOMFLY}

The {\bf framed HOMFLY-PT polynomial} is an invariant of framed links
in $\R^3$ that assigns to $K \subset \R^3$ a Laurent polynomial ${P_K\in\Z[a^{\pm1},z^{\pm1}]}$ which is characterized by the skein
relations (pictured with blackboard framing)
\begin{align}
  & \figgg{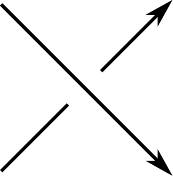}-\figgg{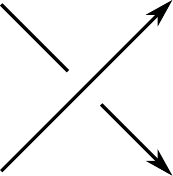} = z \figgg{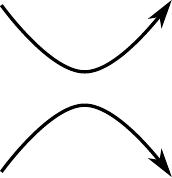}, \\
  & \figgg{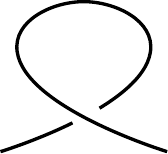}= a \figgg{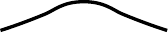}, \quad \quad \quad
  \figgg{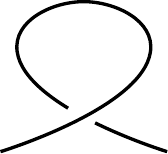} = a^{-1} \figgg{images/HSR6},
\end{align}
together with a choice of normalization for the unknot (framed with
self linking number $0$) which we take to be
$P_{\includegraphics[scale=.5]{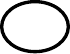}} = (a-a^{-1})/z$.  The
{\bf unframed HOMFLY-PT polynomial}, $\widehat{P}_K \in
\Z[a^{\pm1},z^{\pm1}]$, is then defined by
\[
\widehat{P}_K = a^{-w(K)}P_K
\]
where $w(K)$ is the writhe, i.e. self linking number, of $K$.  Note
that $\widehat{P}_K$ is independent of the framing of $K$.

\begin{definition}
  Given a framed knot $K \subset \R^3$ and a positive $n$-stranded
  braid $\beta$ with the blackboard framing, a diagram for the
  (framed) satellite link $S(K,\beta) \subset \R^3$ arises from
  placing the projection annulus for $\beta$ into an annular
  neighborhood of a diagram for $K$ with blackboard framing.  The {\bf
    framed $n$-colored HOMFLY-PT polynomial} of $K$ is defined by
  \[
  P_{n,K}(a,q) = \frac{1}{\alpha_n} \sum_{\beta \in S_n}
  q^{l(\beta)/2} P_{S(K, \beta)}(a,z)\big|_{z=q^{1/2}-q^{-1/2}}.
  \]
  An {\bf unframed $n$-colored HOMFLY-PT polynomial} can be defined by
  the normalization
  \[
  \widehat{P}_{n,K}(a,q) = \left(a^n q^{n(n-1)/2}\right)^{-w(K)}
  P_{n,K}(a,q).
  \]
\end{definition}

\begin{remark}
  \begin{enumerate}
  \item That $\widehat{P}_{n,K}(a,q)$ is independent of the choice of
    framing on $K$ follows from Theorem 5.5 of \cite{AM} (with $a =
    v^{-1}$, $q^{1/2} = s$, and $x=1$).
  \item In \cite{AM, MM}, more general colored HOMFLY-PT polynomials
    $P_{\lambda, K}$ are defined where $\lambda$ is a partition.  The
    $n$-colored HOMFLY-PT polynomial is the case where $\lambda = n$
    is a partition with only one part.
  \item Notice that $\widehat{P}_{n,K}(a,q)$ belongs to the sub-ring
    of $\mathbb{Q}(a, q^{1/2})$ that is the localization of
    $\Z[a^{\pm1}, q^{\pm 1/2}]$ obtained from inverting
    $q^{r/2}-q^{-r/2}$, $r \geq 1$.  In particular,
    $\widehat{P}_{n,K}$ is a Laurent polynomial in $a$.
  \end{enumerate}
\end{remark}

We now recall two results relating Legendrian knots with the HOMFLY-PT
polynomial.

\begin{theorem}[\cite{FT}] \label{thm:FT} For any Legendrian link $K
  \subset J^1\R$,
  \[
  \mathit{tb}(K) + \rvert r_{\mathit{tot}}(K)\lvert \leq -\deg_a
  \widehat{P}_K,
  \]
  where $r_{\mathit{tot}}(K)$ is the sum of the rotation numbers of
  the components of $K$.
\end{theorem}
Recall that any Legendrian knot $K$ has a framing specified by the
contact planes, and the self-linking number of this framing is
$\mathit{tb}(K)$.  Therefore, an equivalent statement of Theorem
\ref{thm:FT} is that
\[
\deg_a P_K \leq -|r_{\mathit{tot}}(K)|.
\]
In particular, the framed HOMFLY-PT polynomial of a Legendrian knot
does not contain any positive powers of $a$, so it is possible to
specialize $a^{-1}=0$.

\begin{theorem}[\cite{R2}]  \label{thm:R2}
  For any Legendrian link $K \subset J^1\R$, the $2$-graded ruling
  polynomial is the specialization
  \[
  R^2_K(z) = P_K(a,z)\big|_{a^{-1} =0}.
  \]
\end{theorem}

We now generalize these results to apply to the $n$-colored HOMFLY-PT
and ruling polynomials.

\begin{theorem} \label{thm:nHOMFLY} Let $K \subset J^1\R$ be a
  (connected) Legendrian knot.  For all $n \geq 1$,
  \begin{align}
    \label{eq:tb1} \mathit{tb}(K) + \lvert r(K)\rvert &\leq
    \frac{1}{n} \deg_a \widehat{P}_{n,K}, \text{ and }\\
    \label{eq:tb2} \deg_a P_{n,K} &\leq -n \cdot |r(K)|.
  \end{align}
  Moreover, the specialization of $P_{n,K}$ at $a^{-1} = 0$ is the
  2-graded colored ruling polynomial, i.e.
  \begin{equation} \label{eq:tb3} R^2_{n,K}(q) =
    P_{n,K}(a,q)\big|_{a^{-1} = 0}.
  \end{equation}
  If, in addition $r(K)=0$, then
  \begin{equation} \label{eq:tb4} \mbox{Rep}_2(K, \mathbb{F}_q^n) =
    P_{n,K}(a,q)\big|_{a^{-1} = 0}.
  \end{equation}
\end{theorem}
\begin{proof}
 
  For a Legendrian $K \subset J^1\R$, the satellites $S(K,\beta)$
  appearing in the definition of $P_{n,K}$ are themselves Legendrian
  links with $r_{\mathit{tot}}(K) = n \cdot r(K)$, so Theorem
  \ref{thm:FT} gives the inequality
  \[
  \deg_aP_{n,K} = \deg_a\left( \frac{1}{\alpha_n} \sum_{\beta \in S_n}
    q^{l(\beta)/2} P_{S(K, \beta)}\right) \leq - n \cdot |r(K)|.
  \]
  Inequality \eqref{eq:tb1}, then follows from \eqref{eq:tb2} and the
  definition of $\widehat{P}_{n,K}$ since
  \[
  \deg_a \widehat{P}_{n,K} = \deg_a P_{n,K} - n \cdot \mathit{tb}(K)
  \leq -n \cdot (\lvert r(K)\rvert+\mathit{tb}(K)).
  \]

  For (\ref{eq:tb3}), we just apply Theorem \ref{thm:R2} to compute
  \begin{align*}
    R^2_{n,K}(q) &= \frac{1}{\alpha_n} \sum_{\beta \in S_n}
    q^{l(\beta)/2} R^2_{S(K, \beta)}(q^{1/2}-q^{-1/2}) \\ &=
    \frac{1}{\alpha_n} \sum_{\beta \in S_n} q^{l(\beta)/2} P_{S(K,
      \beta)}(a,q^{1/2}-q^{-1/2})\big|_{a^{-1}=0} =
    P_{n,K}\big|_{a^{-1}=0}.
  \end{align*}
  Finally, (\ref{eq:tb4}) is (\ref{eq:tb3}) combined with Theorem
  \ref{thm:color}.
\end{proof}

\begin{corollary} \label{cor:sharp} \begin{enumerate}
  \item The $2$-graded total $n$-dimensional representation number
    $\mbox{Rep}_2(K, \mathbb{F}_q^n)$ depends only on the underlying
    framed knot type of $K$.
  \item The Chekanov-Eliashberg algebra of $K$ has a $2$-graded
    representation on $(\mathbb{F}_q^n,0)$ for some $q$ if and only if
    $\mathit{tb}(K) = \frac{1}{n} \deg_a \widehat{P}_{n,K}$. In
    particular, if such a representation exists, then $K$ must have
    maximal Thurston-Bennequin number within its smooth isotopy class.
  \end{enumerate}
\end{corollary}

\begin{remark}
  The final statement of part (2) of the Corollary was proven for the
  special case $q=2$ in \cite[Theorem 4.9]{NgR}.
\end{remark}

\begin{example}
  For the Legendrian $m(5_2)$ knot, $K$, from Figure \ref{fig:m52}, we
  used the HOMFLY-PT polynomial implementation from \textit{Sage} to
  compute that the $2$-colored HOMFLY-PT polynomial is
  \begin{align*}
    P_{2,K}(a,q) &=  P_{2,O}(a,q) \cdot a^{-10}q^{-5} \cdot \bigg(a^8 \left(q^8-q^7-q^6+2 q^5-q^3+q^2\right) \\
    &\quad +a^6\left(q^8+q^7-2 q^6+3 q^4-q^3-q^2+q\right) -a^4 \left(2 q^5-2 q^3+q^2+q-1\right) \\
    &\quad -a^2 \left(q^4-q^2+q+1\right)+q\bigg).
  \end{align*}
  Here, $\displaystyle P_{2,O}(a,q)= \frac{(a-a^{-1})(a q^{1/2}-
    a^{-1}q^{-1/2})}{(q^{1/2}-q^{-1/2})(q-q^{-1})}$ where $O$ is the
  $0$-framed unknot, so we have
  \begin{align*}
    P_{2,K}\big|_{a^{-1}=0} = & \frac{q^{1/2}}{(q^{1/2}-q^{-1/2})(q-q^{-1})}\cdot q^{-5}\left(q^8-q^7-q^6+2 q^5-q^3+q^2\right)  \\
    % = & \frac{q^3}{(q^2-q)(q^2-1)} q^{-5}\big(q^2
    % \left(q^6-q^5-q^4+2 q^3-q+1\right) \big) \\
    = & q^{-2}\left[(q^2-1)(q^2-q)\right]^{-1} \left(q^8-q^7-q^6+2
      q^5-q^3+q^2\right).
  \end{align*}
  Note that this agrees precisely with the direct computation of
  $\mbox{Rep}_2(K, \mathbb{F}_q^2)$ from (\ref{eq:Count}).
\end{example}

\bibliographystyle{plain}
\bibliography{Bibliography}{}

\end{document}